\documentclass[12pt,a4paper]{article}
\usepackage[centertags]{amsmath}
\usepackage{amsfonts}
\usepackage{amssymb}
\usepackage{amsthm}
\usepackage{newlfont}
\usepackage[T1]{fontenc}
\usepackage[ansinew]{inputenc}
\usepackage{amsmath}
\usepackage{graphicx}
\usepackage{color}
\usepackage[ansinew]{inputenc}
\usepackage[all]{xy}
\theoremstyle{plain}
\newtheorem{thm}{Theorem}[section]
\newtheorem{cor}[thm]{Corollary}
\newtheorem{lem}[thm]{Lemma}
\newtheorem{prop}[thm]{Proposition}
\newtheorem{intro}{Theorem}
\theoremstyle{definition}
\newtheorem{defn}[thm]{Definition}

\newtheorem{rem}[thm]{Remark}
\newtheorem{ex}[thm]{Example}
\theoremstyle{Inductive Assumption}
\newtheorem{inda}[thm]{Inductive Assumption}

\newcommand{\al}{\alpha}
\newcommand{\be}{\beta}
\newcommand{\f}{\varphi}

\newcommand{\s}{\sigma}
\newcommand{\e}{\varepsilon}
\newcommand{\dd}{\partial}
\newcommand{\g}{\gamma}
\newcommand{\la}{\lambda}
\newcommand{\G}{\Gamma}
\newcommand{\Si}{\Sigma}
\newcommand{\om}{\omega}
\newcommand{\Om}{\Omega}
\newcommand{\D}{\Delta}

\renewcommand{\k}{\kappa}
\renewcommand{\t}{\tau}
\renewcommand{\L}{\Lambda}

\newcommand{\N}{\mathbb N}
\newcommand{\R}{\mathbb R}
\newcommand{\Z}{\mathbb Z}
\newcommand{\Q}{\mathbb Q}

\newcommand{\Ø}{\emptyset}
\newcommand{\M}{\mathcal{M}_g}
\newcommand{\proj}{\textup{proj}}
\newcommand{\genus}{\textup{genus}}

\newcommand{\Map}{\textup{Map}}
\newcommand{\Cyk}{\textup{Cyc}}
\newcommand{\im}{\textup{Im}}
\newcommand{\iso}{\cong}

\newcommand{\set}[1]{\left\{#1\right\}}

\newcommand{\To}{\longrightarrow}
\newcommand{\Tto}{\Rightarrow}
\newcommand{\into}{\hookrightarrow}

\newcommand{\id}{\textup{id}}

\newcommand{\indre}[2]{\langle #1,#2 \rangle}

\newcommand{\del}{\subseteq}
\newcommand{\fra}{\setminus}
\newcommand{\rand}{\sharp \dd}
\newcommand{\tensor}{\otimes}
\newcommand{\lil}{\scriptstyle}
\newcommand{\orient}{\circlearrowleft}
\newcommand{\Rel}{\textup{Rel}}
\newcommand{\Sibar}{\overline{\Si}}
\newcommand{\Dbar}{\overline{\D}}

\newcommand{\round}[1]{\lfloor#1\rfloor}

\catcode`\=13
\def{$\bowtie$}
\author{Søren K. Boldsen}
\title{Improved homological stability for the mapping class group with integral or twisted coefficients}
\date{\today}

\begin{document}
\maketitle
\begin{abstract}In this paper we prove stability results for the homology of the mapping class group of a surface. We get a stability range that is near optimal, and extend the result to twisted coefficients.
\end{abstract}
\section*{Introduction}
Let $F_{g,r}$ denote the compact oriented surface of genus $g$ with
$r$ boundary circles, and let $\G_{g,r}$ be the associated mapping
class group,
\begin{equation*}
    \G_{g,r}=\pi_{0}\text{Diff}_+(F_{g,r};\dd),
\end{equation*}
the components of the group of orientation-preserving
diffeomorphisms of $F_{g,r}$ keeping the boundary pointwise
fixed. Gluing a pair of pants onto one  or two boundary circles induce maps
\begin{equation*}
    \Si_{0,1}: \G_{g,r}\To \G_{g,r+1},\quad \Si_{1,-1}: \G_{g,r}\To \G_{g+1,r-1}
\end{equation*}
whose composite $\Si_{1,0}:=\Si_{1,-1}\circ \Si_{0,1}$ corresponds to adding to $F_{g,r}$ a genus one surface with two boundary circles. Using the mapping cone of $\Si_{i,j}$, $(i,j)=(0,1), (1,-1)$ or $(1,0)$ we get a relative homology group, which fits into the exact sequence
\begin{equation*}
   \ldots \To H_n(\Si_{i,j}\G_{g,r})\To H_n(\Si_{i,j}\G_{g,r},\G_{g,r})\To H_{n-1}(\G_{g,r})\To \ldots
\end{equation*}
Homology stability results for the mapping class group can then be derived from the vanishing the relative group (in some range).

We wish to show such a stability result for not only for trivial coefficients but also for so-called coefficients systems of a finite degree. For this, we work in Ivanov's category $\mathfrak{C}$ of marked surfaces, cf. \cite{Ivanov1} and $\S4.1$ below for details. The maps $\Si_{1,0}$ and $\Si_{0,1}$ are functors on $\mathfrak{C}$, and $\Si_{1,-1}$ is a functor on a subcategory.

A coefficient system is a functor $V$ from $\mathfrak{C}$ to the category of abelian groups without infinite division. If the functor is constant, we say $V$ has degree 0. We then define a coefficient system of degree $k$ inductively, by requiring that the maps $V(F){\To}V(\Si_{i,j}F)$ are split injective and their cokernels are coefficient systems of degree $k-1$, see Definition \ref{d:coef}. As an example, the functor $H_1(F;\Z)$ is a coefficients system of degree $1$, and its $k$th exterior power $\L^k H_1(F;\Z)$, considered in \cite{Morita1}, has degree $k$. To formulate our stability result, we consider relative homology group with coefficients in $V$,
\begin{equation*}
    Rel_n^V(\Si_{l,m}F,F)=H_n(\Si_{l,m}\G(F),\G(F); V(\Si_{l,m}F),V(F)).
\end{equation*}
These groups again fit into a long exact sequence. Our main result is
\begin{intro}For $F$ a surface of genus $g$ with at least 1 boundary component, and $V$ a coefficient system of degree $k_V$, we have
\begin{equation*}
    Rel_n^V(\Si_{1,0}F,F)=0 \text{ for } 3n\le 2g-k_V,
\end{equation*}
\begin{equation*}
    Rel_n^V(\Si_{0,1}F,F)=0 \text{ for } 3n\le 2g-k_V.
\end{equation*}
Moreover, if $F$ has at least 2 boundary components, we have
\begin{equation*}
    Rel_q^V(\Si_{1,-1}F,F)=0 \text{ for } 3q\le 2g-k_V+1.
\end{equation*}
\end{intro}
As a corollary, we obtain that $H_n(\G_{g,r}; V(F_{g,r}))$ is independent of $g$ and $r$ for $3n\le 2g-k_V-2$ and $r\ge 1$. For a more precise statement, see Theorem \ref{t:abstwist}. This uses that $\Si_{0,1}$ is always injective, since the composition $\G_{g,r}\stackrel{\Si_{0,1}}{\To}\G_{g,r+1}\stackrel{\Si_{0,-1}}{\To}\G_{g,r}$ is an isomorphism, where $\Si_{0,-1}$ is the map gluing a disk onto a boundary component.

The proof of Theorem 1 with twisted coefficients uses the setup from \cite{Ivanov1}. His category of marked surfaces is slightly different from ours, since we also consider surfaces with more than one boundary component and thus get results for $\Si_{0,1}$ and $\Si_{1,-1}$.

For constant coefficients, $V=\Z$, we also consider the map $\Si_{0,-1}:\G_{g,1}\To \G_{g}$ induced by gluing a
disk onto the boundary circle, where our result is:
\begin{intro}The map
\begin{equation*}
   \Si_{0,-1}: H_k(\G_{g,1};\Z)\To H_k(\G_{g};\Z)
\end{equation*}
is surjective for $2g \ge 3k - 1$, and an isomorphism for $2g \ge 3k
+ 2$.
\end{intro}
The proof of Theorem 2 follows \cite{Ivanov1}, where a stability
result for closed surfaces is deduced from a stability theorem on
surfaces with boundary. We get an improved result, because Theorem 1
has a better bound than Ivanov's stability theorem (which has
isomorphism for $g>2k$).

In this paper, we first prove Theorem 1 for constant integral coefficients, $V=\Z$. Our proof of Theorem 1 in this case is much inspired by Harer's manuscript \cite{Harer2}, which was never published. Harer's manuscript is about rational homology stability. The rational stability results claimed in \cite{Harer2} are ''one degree better'' than what is obtained here with integral coefficients. Before discussing the discrepancy it is convenient to compare the stability with Faber's conjecture.

Let $\M$ be Riemann's moduli space; recall that $H^*(\M;\Q) \iso
H^*(\G_g;\Q)$. From above we have maps
\begin{equation*}
    H^*(\G_g;\Q) \To H^*(\G_{g,1};\Q) \longleftarrow H^*(\G_{\infty,1};\Q)
\end{equation*}
and by \cite{MW},
\begin{equation}\label{e:Ib1}
    H^*(\G_{\infty,1};\Q) = \Q[\k_1,\k_2,\ldots].
\end{equation}
The classes $\k_i\in H^{2i}(\G_{g,r})$ for $r\ge 0$ are the standard classes defined by
Miller, Morita and Mumford ($\k_i$ is denoted $e_i$ by Morita).

The tautological algebra $R^*(\M)$ is the subring of $H^*(\G_g;\Q)$
generated multiplicatively by the classes $\k_i$. Faber conjectured
in \cite{Faber} the complete algebraic structure of $R^*(\M)$. Part
of the conjecture asserts that it is a Poincaré duality algebra
(Gorenstein) of formal dimension $2g-4$, and that it is
generated by $\k_1, \ldots,\k_{[g/3]}$, where $[g/3]$ denotes $g/3$ rounded down. The latter statement was
proved by Morita (cf. \cite{Morita1} prop 3.4).

It follows from our theorems above that $\k_1, \ldots, \k_{[g/3]}$
are non-zero in $H^*(\G_g;\Q)$ when $*\le 2[\frac{g}{3}]-2$. More precisely, if $g\equiv 1,2\:(\textrm{mod } 3)$
then our results show that
\begin{equation}\label{e:Ib2}
    H^*(\G_g;\Q) \iso H^*(\G_{\infty,1};\Q) \quad \text{for }*\le
    2[\textstyle\frac{g}{3}\displaystyle],
\end{equation}
but if $g\equiv 0\:(\textrm{mod } 3)$, our
result only show the isomorphism for $*\le
2[\textstyle\frac{g}{3}\displaystyle]-1$. In contrast, \cite{Harer2}
asserts the isomorphism for $*\le
2[\textstyle\frac{g}{3}\displaystyle]$ for all $g$. We note that is follows from \eqref{e:Ib1} and Morita's result that the best possible stability range for $H^*(\G_{g};\Q)$ is $*\le 2[\textstyle\frac{g}{3}\displaystyle]$. We are ''one
degree off'' when $g\equiv 0\:(\textrm{mod } 3)$.

The stability of \cite{Harer2} is based on three unproven
assertions that I have not been able to verify. I will discuss two
of them below, and the third in section \ref{sc:lemmas}.

Boundary connected sum of surfaces with non-empty boundary defines a
group homomorphism $\G_{g,r}\times \G_{h,s}\To\G_{g+h,r+s-1}$, and
hence a product in homology
\begin{equation*}
    H_*(\G_{g,r})\tensor H_*(\G_{h,s})\To H_*(\G_{g+h,r+s-1}), \quad
    r,s>0.
\end{equation*}
The classes $\k_i$ are primitive with respect to this homology
product, in the sense that $\indre{\k_i}{a\cdot b}=0$ if both $a$
and $b$ have positive degree \cite{Morita2}. Harer proves in \cite{Harer3} that $H^2(\G_{3,1};\Q)=\Q\set{\k_1}$.
Let $\check{\k}_1\in H_2(\G_{3,1};\Q)$ be the dual to $\k_1$, and
let $\check{\k}_1^{\phantom{,}n}$ be the $n$'th power under the
multiplication
\begin{equation*}
    H_2(\G_{3,1})^{\tensor n}\To H_{2n}(\G_{3n,1}).
\end{equation*}
Then $\indre{\k_1^{\phantom{,}n}}{\check{\k}_1^{\phantom{,}n}}=n!$,
so $\check{\k}_1^{\phantom{,}n}\ne 0$ in $H^{2n}(\G_{3n,1};\Q)$, cf.
part $(i)$ of Theorem 1. Dehn twist around the $(r+1)$st boundary circle yields a group homomorphism $\Z\To \G_{1, r+1}$, and hence a class $\t_{r+1}\in H_1(\G_{1,r+1})$.

We can now formulate two of Harer's three assertions one needs in order to improve the rational stability result by ''one degree'' when $g\equiv 0\:(\textrm{mod } 3)$, i.e. from $*\le 2[\frac{g}{3}]-1$ to  $*\le 2[\frac{g}{3}]$. The assertions are:
\begin{itemize}
  \item[$(i)$]$\check{\k}_1^{\phantom{,}n}=0$ in $H_{2n}(\G_{g,r};\Q)$ for $g<3n$.
  \item[$(ii)$]$\t_{r+1}\cdot\check{\k}_1^{\phantom{,}n}$ is non-zero in
  $\textrm{Coker}(H_{2n+1}(\G_{3n+1,r};\Q)\To H_{2n+1}(\G_{3n+1,r+1};\Q)$.
\end{itemize}
The third assertion one needs is stated in Remark \ref{r:third}.

\paragraph{Acknowledgements}This article is part of my ph.d. project at the University of Aarhus. It is a great pleasure to thank my thesis advisor Ib Madsen for his help and encouragement during my years as a graduate student. I am also grateful to Mia Hauge Dollerup for her help in composing this paper.

\newpage
\tableofcontents\newpage
\section{Homology of groups and spectral sequences}

\subsection{Relative homology of groups}\label{S:1,1} For a group $G$, and  $\Z[G]$-modules $M$ and $M'$, left and right modules, respectively, we
have the bar construction:
\begin{equation*}
    B_n(M',G,M)=M'\tensor(\Z[G])^{\tensor n}\tensor M,
\end{equation*}
with the differential
\begin{eqnarray*}
    d_n(m'\tensor g_1\tensor \cdots \tensor g_n \tensor m) &=& (m'g_1)\tensor g_2\tensor \cdots\tensor g_n\tensor m
    \\ &+&
    \sum_{i=1}^{n-1}(-1)^i m'\tensor g_1\tensor \cdots\tensor g_ig_{i+1}\tensor \cdots\tensor g_n\tensor m\\
     &+& (-1)^n m' \tensor g_1\tensor \cdots\tensor g_{n-1}\tensor (g_nm).
\end{eqnarray*}
If either $M$ or $M'$  are free $\Z[G]$-modules, $B_*(M',G,M)$ is
contractible. If $M'=\Z$ with trivial $G$-action, we write
$B_*(G,M)$. Then the $n$th homology group of $G$ with coefficients
in $M$ is defined to be
\begin{equation*}
   H_n(G;M)=H_n(B_*(G,M))\iso \text{Tor}_n^{\Z G}(\Z,M).
\end{equation*}
There is a relative version of this. Suppose $f:G\To H$ is a group
homomorphism and $\f:M\To N$ is an $f$-equivariant map of
$\Z[G]$-modules. One defines the relative homology $H_*(H,G;N,M)$ to
be the homology of the algebraic mapping cone of
\begin{equation*}
    (f,\f)_*: B_*(G,M)\To B_*(H,N),
\end{equation*}
so that there is a long exact sequence
\begin{equation*}
    \cdots\to H_n(G;M)\to H_n(H;N)\to H_n(H,G;M,N)\to
    H_{n-1}(G;M)\to \cdots
\end{equation*}

\subsection{Spectral sequences of group actions}\label{sc:spectral}
Suppose next that $X$ is a connected simplicial complex with a
simplicial action of $G$. Let $C_*(X)$ be the cellular chain complex
of $X$. Given a $\Z[G]$-module $M$, define the chain complex
\begin{equation}\label{e:dagger}
    C_n^\dagger(X;M)=\left\{
                       \begin{array}{ll}
                         0, & \hbox{$n<0$;} \\
                         M, & \hbox{$n=0$;} \\
                         C_{n-1}(X)\otimes_\Z M, & \hbox{$n\ge 1$;}
                                  \end{array}
                     \right.
\end{equation}
with differential $\dd_n^\dagger$ defined to be $\dd_{n-1}\tensor
\id_M$ for $n>1$, and equal to the augmentation $\e\tensor \id_M$
for $n=1$. Note if $X$ is $d$-connected for some $d\ge 1$, or more
generally, if the homology $H_i(X)=0$ for $1\le i \le d$, then
$C_*^\dagger(X;M)$ is exact for $*\le d+1$. This is used below in the spectral sequence.

Again there is a relative version. Let $f:G\To H$, $\f:M\To N$ be as
above, and let $X\del Y$ be a pair of simplicial complexes with a
simplicial action of $G$ and $H$, respectively, compatible with $f$
in the sense that the inclusion $i: X\To Y$ is $f$-equivariant.
Assume in addition that the induced map on orbits,
\begin{equation}\label{e:orbits}\xymatrix{
    i_\sharp:X/G\ar[r]^{\:\:\:\iso} &Y/H}
\end{equation}
is a bijection.

\begin{defn} With $G$, $M$ and $X$ as above, let $\s$ be a $p$-cell of $X$. Let $G_{\s}$ denote the
stabiliser of $\s$, and let $M_\s=M$, but with a twisted
$G_\s$-action, namely
\begin{equation*}
    g*m=\left\{
          \begin{array}{ll}
            gm, & \hbox{if $g$ acts orientation preservingly on $\s$;} \\
            -gm, & \hbox{otherwise.}
          \end{array}
        \right.
\end{equation*}
\end{defn}

\begin{thm}\label{s:ss}
Suppose $X$ and $Y$ are $d$- connected and that the orbit map \eqref{e:orbits} is a bijection. Then there is a spectral sequence $\set{E_{r,s}^n}_n$
converging to zero for $r+s\le d+1$, with
\begin{equation*}
    E_{r,s}^1 \iso \bigoplus_{\s\in \Bar\D_{r-1}} H_s(H_\s,G_\s; N_\s, M_\s).
\end{equation*}
Here $\Bar\D_p=\Bar\D_{p}(X)$ denotes a set of representatives for
the $G$-orbits of the $p$-simplices in $X$.
\end{thm}

\begin{proof}
Consider the double complex with chain groups
\begin{equation*}
    C_{n,m}= F_n(H)\otimes_{\Z[H]}C_m^\dagger(Y,N)\oplus
    F_{n-1}(G)\otimes_{\Z[G]}C_m^\dagger(X,M),
\end{equation*}
where $F_n(G)= B_n(G,\Z[G])$, and differentials (superscripts
indicate horizontal and vertical directions)
\begin{eqnarray}\label{e:dv}
  d^h_m &=& \id\otimes\dd^Y_m\oplus\id\otimes\dd^X_m \nonumber\\
  d^v_n &=&
  \dd^H_n\otimes\id\oplus \left(f_*\otimes(i,\f)_*+\dd^G_{n-1}\otimes\id\right).
\end{eqnarray}

Standard spectral sequence constructions give two spectral sequences
both converging to $H_*(\textrm{Tot}\, C)$, where $\textrm{Tot}\, C$
is the total complex of $C_{*,*}$,
$\displaystyle(\textrm{Tot}\,C)_k=\bigoplus_{n+m=k}C_{n,m}$ and
$d^{\textrm{Tot}}=d^h+d^v$. The vertical spectral sequence (induced
by $d^v$) has $E^1$ page:
\begin{eqnarray*}
  E^1_{r,s} &=& H_r(C_{s,*}) \\
    &=& H_r\left(F_s(H)\otimes_{\Z[H]}C_*^\dagger(Y;N)\right)\oplus
    H_r\left(F_{s-1}(G)\otimes_{\Z[G]}C_*^\dagger(X;M)\right).
\end{eqnarray*}
Since the resolutions $F_*$ are free, this is zero where
$C_*^\dagger(X;M)$ and $C_*^\dagger(Y;N)$ are exact, i.e. for $r\le
d+1$. So this spectral sequence converges to zero where $r+s\le
d+1$, and we conclude that $H_*(\textrm{Tot}\,C)=0$ for $*\le d+1$.

The horizontal spectral sequence, which consequently also converges
to zero in total degrees $\le d+1$, has $E^1$ page
\begin{equation}\label{e:E1}
    E^1_{r,s} = H_s\left(F_*(H)\otimes_{\Z[H]}C_r^\dagger(Y,N)\oplus
    F_{*-1}(G)\otimes_{\Z[G]}C_r^\dagger(X,M)\right).
\end{equation}
For $r\ge 1$ we have
\begin{eqnarray}\label{e:dagM}
  C_r^\dagger(X,M) &=& C_{r-1}(X)\tensor_{\Z[G]} M \iso \bigoplus_{\s\in\D_{r-1}(X)}\Z[G\cdot
\s]\tensor_{\Z[G]}
M \nonumber \\
    &\iso& \bigoplus_{\s\in\Bar\D_{r-1}}\Z[G]\tensor_{\Z[G_\s]} M_\s = \bigoplus_{\s\in\Bar\D_{r-1}}\textup{Ind}_{G_{\s}}^G M_{\s},
\end{eqnarray}
where $\D_p(X)$ denotes the $p$-cells in $X$, and where
$\Bar\D_p\del \D_p(X)$ is a set of representatives for the
$G$-orbits. Finally, $\textup{Ind}_{G_{\s}}^G
M_{\s}=\Z[G]\otimes_{\Z[G_{\s}]}M_{\s}$.

By assumption \eqref{e:orbits}, the image of $\Bar\D_{r-1}$ under
$i$ also works as representatives for the $H$-orbits of
$(r-1)$-cells in $Y$. Therefore we also have:
\begin{equation}\label{e:dagN}
    C_r^\dagger(Y,N)\iso \bigoplus_{\s\in\Bar\D_{r-1}}\textup{Ind}_{H_{\s}}^H
    N_{\s}.
\end{equation}
We insert (\ref{e:dagM}) and (\ref{e:dagN}) into the formula
(\ref{e:E1}) to get for $r\ge1$:
\begin{eqnarray}\label{e:E1rs}
   E^1_{r,s} &=& H_s\left(F_*(H)\otimes_{\Z[H]}C_r^\dagger(Y,N)\oplus
    F_{*-1}(G)\otimes_{\Z[G]}C_r^\dagger(X,M)\right)  \nonumber\\
    &\iso& H_s\left(F_*(H)\otimes_{\Z[H]}
    \bigoplus_{\s\in\Bar\D_{r-1}}\textup{Ind}_{H_{\s}}^H
    N_{\s}\oplus
    F_{*-1}(G)\otimes_{\Z[G]}\bigoplus_{\s\in\Bar\D_{r-1}}
    \textup{Ind}_{G_{\s}}^G M_{\s}\right)  \nonumber\\
    &\iso& \bigoplus_{\s\in\Bar\D_{r-1}} H_s\left(F_*(H)\otimes_{\Z[H]}
    \textup{Ind}_{H_{\s}}^H
    N_{\s} \oplus F_{*-1}(G)\otimes_{\Z[G]} \textup{Ind}_{G_{\s}}^G
    M_{\s} \right)  \nonumber\\
    &\iso& \bigoplus_{\s\in\Bar\D_{r-1}} H_s\left(F_*(H)\otimes_{\Z[H_\s]}
    N_{\s} \oplus F_{*-1}(G)\otimes_{\Z[G_\s]} M_{\s} \right)  \nonumber\\
    &\iso& \bigoplus_{\s\in\Bar\D_{r-1}} H_s(H_\s,G_\s,N_\s,M_\s).
\end{eqnarray}
The final isomorphism above uses that $F_*(H)$ is also a
$\Z[H_\s]$-module. For $r=0$,
\begin{equation*}
    E^1_{0,s}= H_s(H,G;N,M).
\end{equation*}
Thus we set $H_\s=H$ when $\s \in \Bar\D_{-1}=\{\Ø\}$.\end{proof}

For application in the proof of Theorem \ref{t:maintwist}, we need to relax the condition \eqref{e:orbits} to the situation where $i_\sharp$ is only injective:
\begin{thm}\label{s:spectralinj}With the assumptions of Theorem \ref{s:ss}, but with $i_\sharp: X/G \To
Y/H$ is only \emph{injective}, there is a
spectral sequence $\set{E_{r,s}^n}_n$ converging to zero for $r+s\le
d+1$, and
\begin{equation*}
    E_{r,s}^1 \iso \bigoplus_{\s\in \Si_{r-1}(X)} H_s(H_\s,G_\s; N_\s, M_\s)\oplus
    \bigoplus_{\s\in\G_{r-1}(Y)}H_s(H_\s,N_{\s}).
\end{equation*}
Here $\Si_{p}(X)$ denotes a set of representatives for the
$G$-orbits of the $p$-cells in $X$, and $\G_n(Y)$ denotes a set of representatives for those $H$-orbits which do not come from
$n$-cells in $X$ under $i_\sharp$.
\end{thm}

\begin{proof}
We can choose $\Si_n(Y)=i(\Si_n(X))\cup \G_n(Y)$. In this case we obtain:
\begin{equation*}
    E^1_{r,s}\iso \bigoplus_{\s\in\Si_{r-1}}
    H_s(H_\s,G_\s,N_\s,M_\s)\oplus
    \bigoplus_{\s\in\G_{r-1}(Y)}H_s(H_\s,N_{\s}).
\end{equation*}
The first direct sum is obtained in the same way as in the bijective
case. The second consists of absolute homology, since the cells of
$\G_n(Y)$ are not in orbit with cells from $X$.
\end{proof}

We are primarily going to use the absolute case, $Y=\emptyset$:
\begin{cor}\label{c:spectralbad}For a group $G$ acting on a $d$-connected
simplicial complex $X$, and a $G$-module $M$, there is a spectral
sequence converging to zero for $r+s \le d+1$, with
\begin{equation*}
    E^1_{r,s}= \bigoplus_{\s\in\Bar\D_{r-1}}
H_s(G_\s,M_\s),
\end{equation*}
where $\Bar\D_{r-1}$ is a set of representatives of the $G$-orbits
of $(r-1)$-cells in $X$.
\end{cor}

In our applications, we often have a rotation-free group action, in
the following sense:

\begin{defn}\label{d:rotationfree}
A simplicial group action of $G$ on $X$ is rotation-free if for each
simplex $\s$ of $X$, the elements of $G_\s$ fixes $\s$ pointwise.
\end{defn}

\begin{cor}\label{c:spectral} For rotation-free actions, the spectral sequence of Thm.
\ref{s:ss} takes the form:
\begin{equation*}
   E^1_{r,s}\iso \bigoplus_{\s\in\Bar\D_{r-1}} H_s(H_\s,G_\s,N,M)
\end{equation*}
in the relative case, and
\begin{equation*}
    E^1_{r,s}\iso \bigoplus_{\s\in\Bar\D_{r-1}} H_s(G_\s,M)
\end{equation*}
in the absolute case.
\end{cor}
\begin{proof}
The extra assumption implies that each $g\in G_\s$ preserves the
orientation of $\s$. Thus $g$ acts on $M_\s$ in the same way as on
$M$, so $M_\s$ and $M$ are identical as $G_\s$-modules. The same applies to $N$.
\end{proof}

\begin{rem}\label{b:spectral}In some of our applications of the
absolute version of the spectral sequence, $G$ acts both transitively and
rotation-freely on the $n$-simplices of $X$. In this case there is
only one $G$-orbit, so we get
\begin{equation*}
    E^1_{r,s}\iso H_s(G_\s;M),
\end{equation*}
where $\s$ is any $(r-1)$-cell in $X$.
\end{rem}

\subsection{The first differential}\label{sc:diff} We will need a
formula for the first differential $d^1_{r,s}:E^1_{r,s}\To
E^1_{r-1,s}$. From the construction of the spectral sequences of a
double complex, $d^1$ is induced from the vertical differentials
$d^v$ on homology. In the absolute version of the spectral sequence,
assuming that $G$ acts rotation-freely on $X$,
\begin{equation*}
    E^1_{r,s}\iso \bigoplus_{\s\in\Bar\D_{r-1}} H_s(G_\s,M).
\end{equation*}
and it is not hard to se that the differential
\begin{equation*}
   d^1_{r,s}:\bigoplus_{\s\in\Bar\D_{r-1}} H_s(G_\s,M) \To \bigoplus_{\tau\in\Bar\D_{r-2}}
   H_s(G_\tau,M).
\end{equation*}
has the following description (see e.g. \cite{Brown}, Chapter VII,
Prop 8.1.) Let $\s$ be an $(r-1)$-simplex of $X$ and $\tau$ an
$(r-2)$-dimensional face of $\s$. We have the boundary operator
\begin{equation*}
    \dd: C_{r-1}(X,M)\To C_{r-2}(X,M)
\end{equation*}
and we denote its $(\s,\tau)$th component by $\dd_{\s\tau}:M\To M$.
This is a $G_{\s}$-map, so together with the inclusion $G_\s\To
G_\tau$ it induces a map
\begin{equation*}
    u_{\s\tau}:H_*(G_\s,M)\To H_*(G_\tau,M).
\end{equation*}
Up to a sign $u_{\s\tau}$ is the inclusion, because $X$ is a simplicial complex. Consequently
\begin{equation*}
\dd(\s)=\sum_{j=0}^{r-1}(-1)^j(j\text{th face of }\s).
\end{equation*}
So if $\tau$ is the $i$th face of $\s$, then $u_{\s\tau}=(-1)^i$.
For $\s\in\Bar\D_{r-1}$, we cannot be sure that
$\tau\in\Bar\D_{r-2}$, but there is a $g(\tau)\in G$ such that
$g(\tau)\tau=\tau_0\in \Bar\D_{r-2}$. The conjugation, $g\mapsto
g(\tau)gg(\tau)^{-1}$, induces a map from $G_\tau$ to $G_{\tau_0}$
and hence an isomorphism,
\begin{equation*}
    c_{g(\tau)}:H_*(G_\tau,M)\stackrel{\iso}{\To} H_*(G_{\tau_0},M).
\end{equation*}
Now $d^1$ is given by
\begin{equation}\label{e:d1}
   d^1\mid_{H_*(G_\s,M)}=\sum_{\tau \text{ face of
   }\s}u_{\s\tau}c_{g(\tau)}.
\end{equation}
Denoting the $i$th face of $\s$ by $\tau_i$, this can be written:
\begin{equation}\label{e:d1bedre}
   d^1|_{H_*(G_\s,M)}=\sum_{i=0}^{r-1}(-1)^i c_{g(\tau_i)}.
\end{equation}

\section{Arc complexes and permutations}
We write $F_{g,r}$ for a compact oriented surface of genus $g$ with
$r$ boundary components.
\begin{defn}Let $F$ be a surface with boundary. The mapping class group
\begin{equation*}
\G(F)=\pi_0(\textup{Diff}_+(F,\dd F))
\end{equation*}
is the connected components of the group of orientation-preserving
diffeomorphisms which are the identity on a small collar
neighborhood of the boundary. We write $\G_{g,r}=\G(F_{g,r})$.
\end{defn}
To establish stability results about the homology of $\G_{g,r}$, we
will make extensive use of cutting along arcs in $F_{g,r}$. These
arcs will be the vertices in simplicial complexes, the so-called arc
complexes. The mapping class group act on these arc complexes, and
we can use the spectral sequences of section \ref{sc:spectral}. The
differentials in the spectral sequences are closely related to the
homomorphisms of Theorem 1 and Theorem 2 from the introduction.

\subsection{Definitions and basic properties}\label{sc:arc}
Let $F$ be a surface with boundary. To define the ordering of the
vertices used in the arc complexes, we will need the orientation of
$\dd F$. An orientation at a point $p\in\dd F$ is determined
by a tangent vector $v_p$ to the boundary circle at $p$. Let $w_p$ be tangent to $F$ at $p$, perpendicular to $v_p$ and
pointing into $F$. We call the orientation of $\dd F$ at $p$
determined by $v_p$ \emph{incoming} if the pair $(v_p,w_p)$ is
positively oriented, and \emph{outgoing} if $(v_p,w_p)$ is
negatively oriented, and use the same terminology for the connected
component of $\dd F$ that contains $p$.

\begin{defn}\label{d:arc}Given a surface $F$ with non-empty
boundary. Fix two points $b_0$ and $b_1$ in $\dd F$. If $b_0$ and
$b_1$ are on the same boundary component, the arc complex we define
is denoted $C_*(F,1)$. If $b_0$ and $b_1$ are on two different
boundary components of $F$, the resulting arc complex is denoted
$C_*(F;2)$.
\begin{itemize}
\item \noindent A \emph{vertex} of $C_*(F;i)$ is the isotopy class rel endpoints of an
arc (image of a curve) in $F$ starting in $b_0$ and ending in $b_1$,
which has a representative that meets $\dd F$ transversally and only
in $b_0$ and $b_1$.
  \item An \emph{$n$-simplex} $\al$ in $C_*(F;i)$ (called an arc simplex) is set of $n+1$ vertices, such that there are
representatives meeting each other transversally in $b_0$ and $b_1$
and not intersecting each other away from these two points. We
further require that the complement of the $n+1$ arcs be connected.

The set of arcs is ordered by using the incoming orientation of $\dd
F$ at the starting point $b_0$, and we write
$\al=(\al_0,\ldots,\al_n)$.

  \item Let $\D_n(F;i)$ denote the set of
$n$-simplices, and let $C_*(F,i)$ be the chain complex with chain
groups $C_n(F;i)=\Z\D_n(F;i)$ and differentials
  $d: C_n(F;i)\To C_{n-1}(F;i)$ given by:
  \begin{equation*}
  d(\al)=\sum_{j=1}^n(-1)^j\dd_j(\al), \text{ where } \dd_j(\al)=(\al_0,\ldots,\widehat{\al}_j,\ldots,\al_n).
  \end{equation*}
\end{itemize}
\end{defn}
The mapping class group $\G(F)$ acts on $\D_n(F;i)$ (by acting on
the $n+1$ arcs representing an $n$-simplex), and thus on $C_n(F;i)$.
This action is obviously compatible with the differentials $d:
C_n(F;i)\To C_{n-1}(F;i)$, so we can consider the quotient complex
with chain groups $C_n(F;i)/\G(F)$.

To apply the spectral sequence of the action of $\G_{g,r}$ on
$C_*(F_{g,r};i)$, we need to know that the complex is
highly-connected:
\begin{thm}[\cite{Harer1}]\label{s:connected}The chain complex $C_*(F_{g,r};i)$ is
$(2g-3+i)$-connected.
\end{thm}

\begin{defn}\label{d:N}Given an arc simplex $\al$ in
$C_*(F;i)$, we denote by $N(\al)$ the union of a small, open normal
neighborhood of $\al$ with an open collar neighborhood of the
boundary component(s) of $F$ containing $b_0$ and $b_1$. Then the
cut surface $F_\al$ is given by
\begin{equation*}
    F_\al = F\fra N(\al).
\end{equation*}
\end{defn}
For a surface $S$, let $\sharp \dd S$ denote the number of boundary
components of $S$. Then we have the following
\begin{equation}\label{e:boundary}
    \sharp \dd(F_\al)=\sharp\dd N(\al)+r-2i.
\end{equation}

\begin{lem}\label{l:cut}
    Given an $n$-simplex $\al$ in $C_*(F;i)$, the Euler
characteristic of the cut surface $F_ \al$ is
\begin{equation*}
    \chi(F_\al)=\chi(F)+n+1
\end{equation*}
\end{lem}
\begin{proof}
    We prove the formula inductively by removing one arc $\al_0$ at
a time, so it suffices to show that $\chi(F_{\al_0})=\chi(F)+1$.
Give $F$ the structure of a CW complex with $\al_0$ as a $1$-cell
(glued onto the $0$-cells $b_0$ and $b_1$). When we cut along
$\al_0$, we get two copies of $\al_0$; that is, an additional
$1$-cell and two additional $0$-cells. Using the standard formula
for the Euler characteristic of a CW complex, we see that it
increases by $1$.
\end{proof}

\subsection{Permutations}\label{s:permu}
Let $\Si_{n+1}$ denote the group of permutations of the set
$\set{0,1,\ldots,n}$. I will write a permutation $\s\in\Si_n$ as
$\s=\left[\s(0)\,\s(1)\,\ldots\, \s(n)\right]$; e.g. $[0\,2\,1]$ in
$\Si_3$ is the permutation fixing $0$ and interchanging $1$ and $2$.

To each $n$-arc simplex $\al$ in one of the arc complexes $C_*(F;i)$
we assign a permutation $P(\al)$ in $\Si_{n+1}$ as follows: Recall
that the arcs in $\al=(\al_0,\al_1,\ldots, \al_n)$ are ordered using
the incoming orientation of $\dd F$ at the starting point $b_0$. We
use the \emph{outgoing} orientation in the end point $b_1$ to read
off the positions of the $n+1$ arcs at $b_1$: $\al_{j}$ is the
$\s(j)$'th arc at $b_1$, for $j=0,\ldots, n$. In other words, the
arcs at $b_1$ will be ordered
$(\al_{\s^{-1}(0)},\al_{\s^{-1}(1)},\ldots, \al_{\s^{-1}(n)})$. This
gives the permutation $\s=P(\al)$. See Example \ref{ex:perm} below.

So we have a map $P:\D_n(F;i)\To \Si_{n+1}$. Since $\g\in\G(F)$
keeps a small neighborhood of $\dd F$ fixed, this induces a
well-defined map
\begin{equation*}
    P:\D_n(F;i)/\G(F) \To \Si_{n+1}.
\end{equation*}

There are several reasons why it is useful to look at the
permutation $P(\al)$ of an arc simplex $\al$. One is that $P(\al)$
determines the number of boundary components of the cut surface
$F_ \al$, as we shall see below. Before explaining this, we will
need a few preliminary remarks.

Let $\al$ be an arc in $C_*(F;i)$. We orient it from $b_0$ to $b_1$,
and let $t_p(\al)$ be the (positive) tangent vector at $p\in \al$. A
normal vector $v_p$ to $\al$ at $p$ is called \emph{positive} if
$(v_p,t_p(\al))$ is a positive basis of $T_pF$. We say that the
right-hand side of $\al$ is the part of the normal tube given by the
positive normal vectors.

When drawing pictures to aid the geometric intuition, we always
indicate the orientation of $F$ and $\dd F$ (with arrows). Also, the
orientation of $F$ will always be the same, namely the orientation
induced by the standard orientation of this paper. This has the
advantage that orientation-depending properties like the right-hand
side will be consistent throughout the picture, even if we draw two
different areas of one surface.

\begin{ex}\label{ex:perm} Let $\al=(\al_0,\al_1,\al_2)$ be a $2$-simplex in $C_*(F_{g,r};1)$, with
permutation $P(\al)=[1\,2\,0]$. Close to $b_0$ and $b_1$ we see the situation depicted on Figure
\ref{b:alfa2}, with the orientations of $\dd F$ at $b_0$ and $b_1$
used for determining the permutation as indicated.
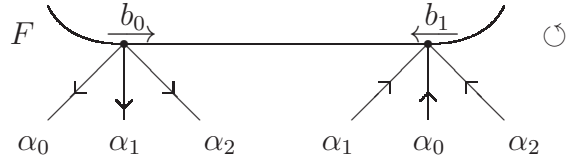
\begin{figure}[!hbt]
\begin{center}
\setlength{\unitlength}{0.5cm}
\begin{picture}(15,4)(0,1)
\put(3,3){\circle*{0.2}}     \put(11,3){\circle*{0.2}}
\put(2.9,3.5){$b_0$} \put(10.9,3.5){$b_1$}

\put(0,3){$F$} \put(14,3){$\orient$} \put(2.5,3.05){$\To$}
\put(10.5,3.05){$\longleftarrow$}
\put(3,3){\line(-1,-1){2}} \put(3,3){\line(0,-1){2}}
\put(3,3){\line(1,-1){2}} \put(11,3){\line(-1,-1){2}}
\put(11,3){\line(0,-1){2}} \put(11,3){\line(1,-1){2}}
\put(0.2,0.2){$\al_0$} \put(2.6,0.2){$\al_1$} \put(5.1,0.2){$\al_2$}
\put(8.2,0.2){$\al_1$} \put(10.6,0.2){$\al_0$}
\put(13.1,0.2){$\al_2$}

\put(10,2){\line(-1,0){0.3}} \put(10,2){\line(0,-1){0.3}}
\put(12,2){\line(1,0){0.3}} \put(12,2){\line(0,-1){0.3}}
\qbezier(11,1.7)(11,1.7)(10.8,1.5)\qbezier(11,1.7)(11,1.7)(11.2,1.5)
\put(1.7,2){\line(0,-1){0.3}} \put(2,1.7){\line(-1,0){0.3}}
\put(4.3,2){\line(0,-1){0.3}} \put(4,1.7){\line(1,0){0.3}}
\qbezier(3,1.3)(3,1.3)(2.8,1.5) \qbezier(3,1.3)(3,1.3)(3.2,1.5)

\qbezier(3,3)(1.5,3)(1,4) \qbezier(11,3)(12.5,3)(13,4)

\put(3,3){\line(1,0){8}}

\end{picture}
\end{center}\caption{An arc with permutation $[1\,2\,0]$ in $C_*(F;1)$.}\label{b:alfa2}\end{figure}

We want to find the number of boundary components of $F_\al$.
This goes as follows. Pick an arc, say $\al_0$, at $b_0$ and start
coloring the right-hand side of it (here, we color it dark grey),
following the arc all the way to $b_1$. See Figure \ref{b:cut1}.
Here, continue to the left-hand side of the next arc; in our case it
is $\al_2$. Note that in general this means going from
$\al_{\s^{-1}(j)}$ to $\al_{\s^{-1}(j-1)}$ (see the definition); in
this example $j=1$. Color the left-hand side of $\al_2$, reaching
$b_0$ again and continuing to the right-hand side of the arc next to
$\al_2$. In this algorithm the boundary component(s) containing
$b_0$ and $b_1$ also counts as arcs, as shown in the figure.
Continue in this fashion until you get back where you started (i.e.
the right-hand side of $\al_0$). This closed, dark grey loop
constitutes one boundary component of $F_\al$. Start over again
with a different color (here light grey) at another arc, and you get
a picture as in Figure \ref{b:cut1}. So there are $2+(r-1)=r+1$
boundary components of $(F_{g,r})_\al$ for $\al \in C_*(F;1)$ with
$P(\al)=[1\,2\,0]$.
\begin{figure}[!hbt]
\begin{center}
\setlength{\unitlength}{0.5cm}
\begin{picture}(15,4)(0,1)
\put(3,3){\circle*{0.2}}     \put(11,3){\circle*{0.2}}
\put(2.9,3.5){$b_0$} \put(10.9,3.5){$b_1$}

\put(0,3){$F$} \put(14,3){$\orient$} \put(2.5,3.05){$\To$}
\put(10.5,3.05){$\longleftarrow$}

\linethickness{0.5mm}
\color[rgb]{0.50,0.50,0.50} \qbezier(2.6,2.85)(2.6,2.85)(0.95,1.2)
\qbezier(2.6,2.85)(1.4,2.9)(0.8,3.9)
\qbezier(3.15,2.6)(3.15,2.6)(3.15,1)
\qbezier(3.15,2.6)(3.15,2.6)(4.8,0.95)
\qbezier(3.4,2.85)(3.4,2.85)(5.05,1.2)
\qbezier(11.15,2.6)(11.15,2.6)(11.15,1)
\qbezier(11.15,2.6)(11.15,2.6)(12.8,0.95)
\qbezier(10.6,2.85)(10.6,2.85)(8.95,1.2)
\qbezier(11.4,2.85)(11.4,2.85)(13.05,1.2)
\qbezier(11.4,2.85)(12.6,2.9)(13.2,3.9)
\qbezier(3.4,2.85)(3.4,2.85)(10.6,2.85)

\color[rgb]{0.8,0.8,0.8} \qbezier(10.85,2.6)(10.85,2.6)(9.2,0.95)
\qbezier(10.85,2.6)(10.85,2.6)(10.85,1)
\qbezier(2.85,2.6)(2.85,2.6)(1.2,0.95)
\qbezier(2.85,2.6)(2.85,2.6)(2.85,1)\color{black}\thinlines

\put(3,3){\line(-1,-1){2}} \put(3,3){\line(0,-1){2}}
\put(3,3){\line(1,-1){2}} \put(11,3){\line(-1,-1){2}}
\put(11,3){\line(0,-1){2}} \put(11,3){\line(1,-1){2}}
\put(0.2,0.2){$\al_0$} \put(2.6,0.2){$\al_1$} \put(5.1,0.2){$\al_2$}
\put(8.2,0.2){$\al_1$} \put(10.6,0.2){$\al_0$}
\put(13.1,0.2){$\al_2$}

\put(10,2){\line(-1,0){0.3}} \put(10,2){\line(0,-1){0.3}}
\put(12,2){\line(1,0){0.3}} \put(12,2){\line(0,-1){0.3}}
\qbezier(11,1.7)(11,1.7)(10.8,1.5)\qbezier(11,1.7)(11,1.7)(11.2,1.5)
\put(1.7,2){\line(0,-1){0.3}} \put(2,1.7){\line(-1,0){0.3}}
\put(4.3,2){\line(0,-1){0.3}} \put(4,1.7){\line(1,0){0.3}}
\qbezier(3,1.3)(3,1.3)(2.8,1.5) \qbezier(3,1.3)(3,1.3)(3.2,1.5)

\qbezier(3,3)(1.5,3)(1,4) \qbezier(11,3)(12.5,3)(13,4)

\put(3,3){\line(1,0){8}} 

\end{picture}
\end{center}\caption{Boundary components of $F_\al$ for $\al$ in
$C_*(F;1)$.}\label{b:cut1}
\end{figure}
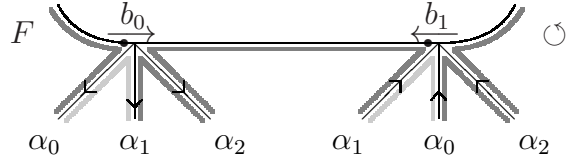

We could consider the same permutation in $C_*(F_{g,r};2)$, and we
would get a different picture (Figure \ref{b:cut2}). So there are
$3+(r-2)=r+1$ boundary components of $(F_{g,r})_\al$ for $\al\in
 C_*(F;2)$ with $P(\al)=[1\,2\,0]$.
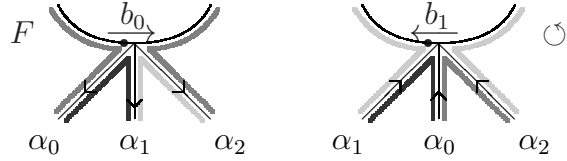
\begin{figure}[!hbt]
\begin{center}
\setlength{\unitlength}{0.5cm}
\begin{picture}(15,4)(0,1)
\put(3,3){\circle*{0.2}}     \put(11,3){\circle*{0.2}}
\put(2.9,3.5){$b_0$} \put(10.9,3.5){$b_1$}

\put(0,3){$F$} \put(14,3){$\orient$} \put(2.5,3.05){$\To$}
\put(10.5,3.05){$\longleftarrow$}

\linethickness{0.5mm}
\color[rgb]{0.50,0.50,0.50} \qbezier(2.6,2.85)(2.6,2.85)(0.95,1.2)
\qbezier(2.6,2.85)(1.4,2.9)(0.8,3.9)
\qbezier(3.4,2.85)(3.4,2.85)(5.05,1.2)
\qbezier(11.15,2.6)(11.15,2.6)(11.15,1)
\qbezier(11.15,2.6)(11.15,2.6)(12.8,0.95)
\qbezier(3.4,2.85)(4.6,2.9)(5.2,3.9)

\color[rgb]{0.20,0.20,0.20} \qbezier(2.85,2.6)(2.85,2.6)(1.2,0.95)
\qbezier(2.85,2.6)(2.85,2.6)(2.85,1)
\qbezier(10.85,2.6)(10.85,2.6)(9.2,0.95)
\qbezier(10.85,2.6)(10.85,2.6)(10.85,1)

\color[rgb]{0.8,0.8,0.8}\qbezier(3.15,2.6)(3.15,2.6)(3.15,1)
\qbezier(3.15,2.6)(3.15,2.6)(4.8,0.95)
\qbezier(10.6,2.85)(10.6,2.85)(8.95,1.2)
\qbezier(11.4,2.85)(11.4,2.85)(13.05,1.2)
\qbezier(11.4,2.85)(12.6,2.9)(13.2,3.9)
\qbezier(10.6,2.85)(9.4,2.9)(8.8,3.9)
 \color{black}\thinlines

\put(3,3){\line(-1,-1){2}} \put(3,3){\line(0,-1){2}}
\put(3,3){\line(1,-1){2}} \put(11,3){\line(-1,-1){2}}
\put(11,3){\line(0,-1){2}} \put(11,3){\line(1,-1){2}}
\put(0.2,0.2){$\al_0$} \put(2.6,0.2){$\al_1$} \put(5.1,0.2){$\al_2$}
\put(8.2,0.2){$\al_1$} \put(10.6,0.2){$\al_0$}
\put(13.1,0.2){$\al_2$}

\put(10,2){\line(-1,0){0.3}} \put(10,2){\line(0,-1){0.3}}
\put(12,2){\line(1,0){0.3}} \put(12,2){\line(0,-1){0.3}}
\qbezier(11,1.7)(11,1.7)(10.8,1.5)\qbezier(11,1.7)(11,1.7)(11.2,1.5)
\put(1.7,2){\line(0,-1){0.3}} \put(2,1.7){\line(-1,0){0.3}}
\put(4.3,2){\line(0,-1){0.3}} \put(4,1.7){\line(1,0){0.3}}
\qbezier(3,1.3)(3,1.3)(2.8,1.5) \qbezier(3,1.3)(3,1.3)(3.2,1.5)

\qbezier(3,3)(1.5,3)(1,4) \qbezier(11,3)(12.5,3)(13,4)


\qbezier(3,3)(4.5,3)(5,4) \qbezier(11,3)(9.5,3)(9,4)

\end{picture}
\end{center}\caption{Boundary components of $F_\al$ for $\al$ in
$C_*(F;2)$.}\label{b:cut2}\end{figure}
\end{ex}

The method of the above example gives a formula -- albeit a rather
cumbersome one -- for $\sharp\dd N(\al)$, and thus by
\eqref{e:boundary} for the number of boundary components of $F_
\al$ in terms of $P(\al)$:

\begin{prop}\label{p:grim}Let $\sharp\dd S$ denote the number of boundary components
in $S$, and let $\s_{k}\in\Si_{k}$ be given by $\s_k=[1\,2\,\cdots\,
k\!-\!1\,0]$. Then
\begin{enumerate}
  \item[$(i)$]If $\al\in C_{n-1}(F;1)$ then $\sharp\dd N(\al) =
  \Cyk \Big(\s_{n+1}
  \widehat{P(\al)}^{-1}\s_{n+1}^{-1}\widehat{P(\al)}\Big)+1$.
  \item[$(ii)$]If $\al\in C_{n-1}(F;2)$ then $\sharp\dd N(\al) =
  \Cyk\Big(
  \s_{n}P(\al)^{-1}\s_{n}^{-1}P(\al)\Big)+2$,
\end{enumerate}
Here $\Cyk: \Sigma_{k} \to \N$ denotes the number of disjoint
cycles in the given permutation, and for
$\tau\in\Si_{k}$, $\widehat{\tau}\in \Si_{k+1}$ is given by
$\widehat{\tau}=[0, \tau+1]$, that is
\begin{equation*}
    \widehat{\tau}(j)=\left\{
                \begin{array}{ll}
                  0, & \hbox{$j=0$,} \\
                  \tau(j-1)+1, & \hbox{$i=1,\ldots,k$.}
                \end{array}
              \right.
\end{equation*}
In particular, $\sharp\dd N(\al)$ depends only on $P(\al)$.
\end{prop}
\begin{proof}This is simply a way to formulate the method described
in Example \ref{ex:perm}. Let us look at $C_*(F;2)$ first, so $b_0$
and $b_1$ are in different boundary components. As in the example,
we start on the right-hand side of one of the arcs at $b_0$, follow
it (using $P(\al)$), then at $b_1$ we go left to the next arc (using
$\s^{-1}$). Now we follow the right side of that arc (using
$P(\al)^{-1}$) ending at $b_0$, and we must now go left to the next
arc (using $\s$). Thus the permutation $P(\al)\s^{-1}
  P(\al)^{-1}\s$ captures how the
boundary of $N(\al)$ behaves, and a boundary component in $\dd
N(\al)$ clearly corresponds to a cycle in the permutation.
Remembering the two extra components corresponding to the components
of $\dd N(\al)$ containing $b_0$ and $b_1$, this proves $(ii)$.

For $C_*(F;1)$, $b_0$ and $b_1$ lie on the same boundary component. We
wish to use $(ii)$, so we consider a new surface $\hat F$ and a new
arc simplex, $\hat\al= (\hat\al_0, \hat\al_1, \ldots, \hat\al_n)$ in
$C_*(\hat F, 2)$, which are constructed from $F$ and $\al$ as follows.

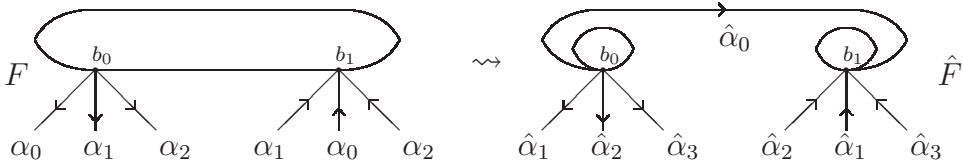
\begin{figure}[!hbt]
\begin{center}
\setlength{\unitlength}{0.4cm}
\begin{picture}(15,4)(0,1)
\put(3,3){\circle*{0.2}}     \put(11,3){\circle*{0.2}}
\put(2.9,3.3){$\lil b_0$} \put(10.9,3.3){$\lil b_1$}

\put(0,2.5){$F$}
\put(3,3){\line(-1,-1){2}} \put(3,3){\line(0,-1){2}}
\put(3,3){\line(1,-1){2}} \put(11,3){\line(-1,-1){2}}
\put(11,3){\line(0,-1){2}} \put(11,3){\line(1,-1){2}}
\put(0.2,0.2){$\al_0$} \put(2.6,0.2){$\al_1$} \put(5.1,0.2){$\al_2$}
\put(8.2,0.2){$\al_1$} \put(10.6,0.2){$\al_0$}
\put(13.1,0.2){$\al_2$}

\put(10,2){\line(-1,0){0.3}} \put(10,2){\line(0,-1){0.3}}
\put(12,2){\line(1,0){0.3}} \put(12,2){\line(0,-1){0.3}}
\qbezier(11,1.7)(11,1.7)(10.8,1.5)\qbezier(11,1.7)(11,1.7)(11.2,1.5)
\put(1.7,2){\line(0,-1){0.3}} \put(2,1.7){\line(-1,0){0.3}}
\put(4.3,2){\line(0,-1){0.3}} \put(4,1.7){\line(1,0){0.3}}
\qbezier(3,1.3)(3,1.3)(2.8,1.5) \qbezier(3,1.3)(3,1.3)(3.2,1.5)

\qbezier(3,3)(1.5,3)(1,4) \qbezier(11,3)(12.5,3)(13,4)
\qbezier(3,5)(1.5,5)(1,4) \qbezier(11,5)(12.5,5)(13,4)

\put(3,3){\line(1,0){8}} \put(3,5){\line(1,0){8}}

\end{picture}
\begin{picture}(1,4)
\put(0,2){$\rightsquigarrow$}
\end{picture}
\begin{picture}(15,4)(0,1)
\put(3,3){\circle*{0.2}}     \put(11,3){\circle*{0.2}}
\put(2.9,3.3){$\lil b_0$} \put(10.9,3.3){$\lil b_1$}

\put(14,2.5){$\hat F$}

\put(3,3){\line(-1,-1){2}} \put(3,3){\line(0,-1){2}}
\put(3,3){\line(1,-1){2}} \put(11,3){\line(-1,-1){2}}
\put(11,3){\line(0,-1){2}} \put(11,3){\line(1,-1){2}}
\put(0.2,0.2){$\hat\al_1$} \put(2.6,0.2){$\hat\al_2$}
\put(5.1,0.2){$\hat\al_3$} \put(8.2,0.2){$\hat\al_2$}
\put(10.6,0.2){$\hat\al_1$} \put(13.1,0.2){$\hat\al_3$}
\put(6.8,3.8){$\hat\al_0$}

\put(10,2){\line(-1,0){0.3}} \put(10,2){\line(0,-1){0.3}}
\put(12,2){\line(1,0){0.3}} \put(12,2){\line(0,-1){0.3}}
\qbezier(11,1.7)(11,1.7)(10.8,1.5)\qbezier(11,1.7)(11,1.7)(11.2,1.5)
\put(1.7,2){\line(0,-1){0.3}} \put(2,1.7){\line(-1,0){0.3}}
\put(4.3,2){\line(0,-1){0.3}} \put(4,1.7){\line(1,0){0.3}}
\qbezier(3,1.3)(3,1.3)(2.8,1.5) \qbezier(3,1.3)(3,1.3)(3.2,1.5)
\qbezier(7,5)(7,5)(6.8,4.8)\qbezier(7,5)(7,5)(6.8,5.2)

\qbezier(3,3)(1.5,3)(1,4) \qbezier(11,3)(12.5,3)(13,4)
\qbezier(3,5)(1.5,5)(1,4) \qbezier(11,5)(12.5,5)(13,4)

\qbezier(3,3)(2.3,3)(2,3.7) \qbezier(11,3)(11.7,3)(12,3.7)
\qbezier(3,4.4)(2.3,4.4)(2,3.7) \qbezier(11,4.4)(11.7,4.4)(12,3.7)
\qbezier(3,3)(3.7,3)(4,3.7) \qbezier(11,3)(10.3,3)(10,3.7)
\qbezier(3,4.4)(3.7,4.4)(4,3.7) \qbezier(11,4.4)(10.3,4.4)(10,3.7)
 \put(3,5){\line(1,0){8}}

\end{picture}
\end{center}\caption{Constructing $\hat F$ and $\hat \al$ from $F$ and $\al$.}\label{b:ref}\end{figure}

We take the boundary component of $F$ containing $b_0$ and $b_1$,
and close up part of it between $b_0$ and $b_1$ so we get two
boundary components, cf. Figure \ref{b:ref}. Then $\hat\al_0$ will
be the arc from $b_0$ to $b_1$ consisting of the part of the old
boundary component which was first (i.e. right-most) in the incoming
ordering at $b_0$ (cf. Figure \ref{b:ref}), and $\hat\al_j =
\al_{j-1}$ for $1\le j \le n$. By this construction, $\sharp\dd
N(\al) = \sharp\dd N(\hat\al)-1$, since we count two boundary
components for $\hat\al\in C_*(\hat F;2)$, and we should count only
one. Clearly $\textstyle P(\hat\al)=\widehat{P(\al)}$, and the
result now follows from $(ii)$.
\end{proof}
I would like to thank my brother, Jens Boldsen, for help with the above proposition.

\begin{prop}\label{p:inj}The permutation map
\begin{equation*}
P:\D_n(F;i)/\G(F) \To \Si_{n+1}
\end{equation*}
 is injective.
\end{prop}
\begin{proof}
We have to show that given two $n$-arc simplices $\al$ and $\be$
with $P(\al)=P(\be)$, there exists $\g\in\G$ such that $\g\al=\be$.
Consider the cut surfaces $F_\al$ and $F_\be$. Since the
permutations are the same, $F_\al$ and $F_\be$ have the same
number of boundary components, by Prop. \ref{p:grim} above. Now
since we have parameterizations of the boundary components and the
curves $\al_0,\ldots, \al_n$ this gives a diffeomorphism
$\f:\dd(F_\al) \To \dd(F_\be)$. The Euler characteristic of
$F_\al$ and $F_\be$ are also the same, according to Lemma
\ref{l:cut}. This implies that $F_\al$ and $F_\be$ have the
same genus. By the classification of surfaces with boundary,
$F_\al\iso F_\be$ via an orientation preserving diffeomorphism
$\Phi$ extending $\f$. Gluing both $F_\al$ and $F_\be$ up
again gives a diffeomorphism $\bar{\Phi}:F\To F$ taking $\al$ to
$\be$. Thus $\al$ and $\be$ are conjugate under
$\g=\left[\bar{\Phi}\right]$ in the mapping class group $\G(F)$.
\end{proof}

Whether $P$ is surjective depends on the genus $g$, cf. Corollary
\ref{c:bij} below.

\begin{rem}The proof of this proposition also shows that the action of $G(F)$ on $C_*(F;i)$ is rotation-free, cf. Def. \ref{d:rotationfree}. For given $\al\in\D_n(F;i)$ and $\g=[\f]\in \G_\al$,
\end{rem}

\subsection{Genus}

\begin{defn}[Genus]\label{d:genus} To an arc simplex $\al$ we associate
the number $S(\al) = $ genus$(N(\al))$, cf. Def. \ref{d:N}. We call
$S(\al)$ the genus of $\al$.
\end{defn}
Note that Harer calls this quantity the \emph{species} of $\al$.
\begin{lem}\label{l:euler}
    For $\al\in \D_n(F;i)$, we have
\begin{equation*}
    \chi(N(\al))=-(n+1)
\end{equation*}
\end{lem}
\begin{proof}
In $C_*(F;1)$, $N(\al)$ has $\al\cup_{b_0,b_1} S^1$ as a retract. Now
there is a homotopy taking $b_1$ to $b_0$ along $S^1$, so up to
homotopy, this is a wedge of $n+2$ copies of $S^1$ coming from
$\al_0,\ldots,\al_n$ and from the boundary component. This gives the
result. For $C_*(F;2)$ the argument is similar.
\end{proof}

\begin{prop}\label{p:genus}Let $\sharp\dd S$ denote the number of boundary components
in a surface $S$. Let $i=1,2$. Then for any $\al\in\D_n(F_{g,r};i)$, the following relations hold:
\begin{enumerate}
  \item[$(i)$] $S(\al)=½\big(n+3-\sharp\dd N(\al)\big)$,
  \item[$(ii)$] $\sharp\dd (F_\al) = r+n-S(\al)+3-2i$,
  \item[$(iii)$] $\textup{genus}(F_\al)= g+S(\al)-(n+2-i)$,
\end{enumerate}
\end{prop}

\begin{proof}
\noindent$(i)$ As $S(\al)$ is the genus of $N(\al)$, we can derive
this
  from the Euler characteristic of $N(\al)$, which by Lemma
  \ref{l:euler} is $-(n+1)$. Using the formula $\chi(N(\al)) =
  2-2S(\al)-\sharp\dd N(\al)$ gives the result.
\newline\newline
 \noindent$(ii)$ This follows from $(i)$ and \eqref{e:boundary}.
\newline\newline
\noindent$(iii)$ As in $(i)$ we use the connection between Euler
characteristic, genus and number of boundary
  components, together with $(i)$ and $(ii)$:
\begin{eqnarray*}
  \textup{genus}(F_\al) &=& \textstyle ½\big(-\chi(F_\al)- \sharp\dd (F_\al) +2\big)\\
   &=& \textstyle ½\big(-(2-2g-r)-(n+1)-(\sharp\dd N(\al)+r-2i)+2\big)\\
   &=& \textstyle ½\big(2g+(n+1-\sharp\dd N(\al)+2)+2i-2-2(n+1)\big)\\
    &=& g+S(\al)-(n+2-i)
\end{eqnarray*}
\end{proof}
Consequently all information about $F_\al$ can be extracted from
$\rand(F_\al)$, so it is important that we can compute this
quantity:
\begin{lem}\label{l:nu}Given $\al\in\D_n(F;i)$ be given, and let $\nu\in\D_0(F;i)$ be an arc such that $\al'=\al\cup \nu$ is an $(n+1)$-simplex. Consider $\al'\in C_*(F_\al;i)$. Then:
\begin{equation*}
    \rand(F_{\al'})=\left\{
                    \begin{array}{ll}
                     \rand(F_\al)+ 1, & \hbox{if $\nu\in \D_0(F_\al;1)$;} \\
                 \rand(F_\al)- 1, &     \hbox{if $\nu\in \D_0(F_\al;2)$.}
                \end{array}
         \right.
\end{equation*}
\end{lem}

\begin{proof}
    Let $k=\rand(F_\al)$. Since all boundary components in $F_{\al'}$ not
intersecting $\nu$ correspond to boundary components in $F_\al$,
it is enough to consider the situation close to $\nu$. There are two
possibilities: Either $\nu$ will start and end on two different
boundary components of $F_\al$, so $\nu\in \D_0(F_\al;2)$, or
$\nu$ will start and end on the same boundary component of $F_\al$, so $\nu\in \D_0(F_\al;1)$. Cf. Figure \ref{b:rand}, where the
boundary components of $F_\al$ are indicated as in Example
\ref{ex:perm}.

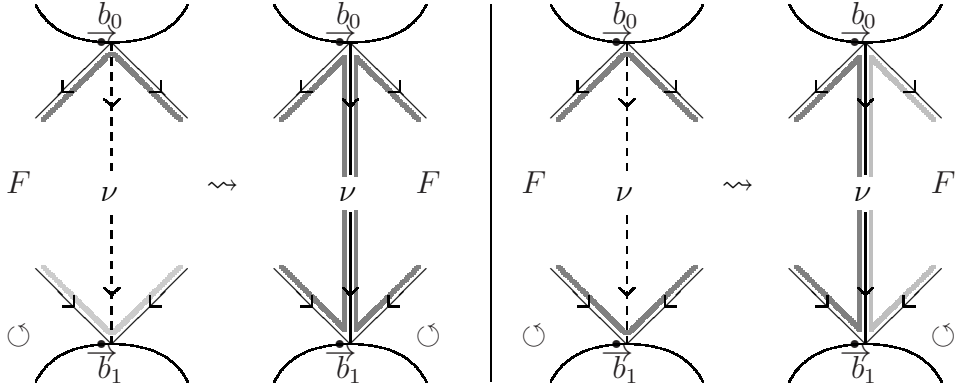
\begin{figure}[!h]
\begin{center}
\setlength{\unitlength}{0.5cm}
\begin{picture}(6,9)(0,1)
\put(3,9){\circle*{0.2}} \put(2.9,9.5){$b_0$}
\put(3,1){\circle*{0.2}} \put(2.9,0.1){$b_1$}

\put(0.5,5){$F$} \put(0.5,1){$\orient$} \put(2.6,9.05){$\to$}
\put(2.6,0.55){$\to$}
\linethickness{0.5mm}
\color[rgb]{0.50,0.50,0.50} \qbezier(2.85,8.6)(2.85,8.6)(1.2,6.95)
\qbezier(3.15,8.6)(3.15,8.6)(4.8,6.95)
\qbezier(3.15,8.6)(3,8.8)(2.85,8.6)

\color[rgb]{0.8,0.8,0.8} \qbezier(3.15,1.4)(3,1.2)(2.85,1.4)

\qbezier(2.85,1.4)(2.85,1.4)(1.2,3.05)

\qbezier(3.15,1.4)(3.15,1.4)(4.8,3.05)\color{black}\thinlines

\put(3,9){\line(-1,-1){2}}
\multiput(3,9)(0,-0.4){9}{\line(0,-1){0.2}}
\put(3,9){\line(1,-1){2}} \put(2.7,4.8){$\nu$}
\put(3,1){\line(-1,1){2}} \multiput(3,1)(0,0.4){9}{\line(0,1){0.2}}
\put(3,1){\line(1,1){2}}

\put(1.7,8){\line(0,-1){0.3}} \put(2,7.7){\line(-1,0){0.3}}
\put(4.3,8){\line(0,-1){0.3}} \put(4,7.7){\line(1,0){0.3}}
\qbezier(3,7.3)(3,7.3)(2.8,7.5) \qbezier(3,7.3)(3,7.3)(3.2,7.5)
\put(1.7,2){\line(1,0){0.3}} \put(2,2){\line(0,1){0.3}}
\put(4.3,2){\line(-1,0){0.3}} \put(4,2){\line(0,1){0.3}}
\qbezier(3,2.3)(3,2.3)(2.8,2.5) \qbezier(3,2.3)(3,2.3)(3.2,2.5)

\qbezier(3,9)(1.5,9)(1,10) 
\qbezier(3,1)(1.5,1)(1,0)
\qbezier(3,9)(4.5,9)(5,10) \qbezier(3,1)(4.5,1)(5,0)
\put(5.5,5){$\rightsquigarrow$}
\end{picture}
\begin{picture}(7,9)(0,1)
\put(3,9){\circle*{0.2}} \put(2.9,9.5){$b_0$}
\put(3,1){\circle*{0.2}} \put(2.9,0.1){$b_1$}

\put(5,5){$F$} \put(5,1){$\orient$} \put(2.6,9.05){$\to$}
\put(2.6,0.55){$\to$}

\linethickness{0.5mm}
\color[rgb]{0.50,0.50,0.50} \qbezier(2.85,8.6)(2.85,8.6)(1.2,6.95)
\qbezier(2.85,8.6)(2.85,8.6)(2.85,5.5)
\qbezier(2.85,1.4)(2.85,1.4)(1.2,3.05)
\qbezier(2.85,1.4)(2.85,1.4)(2.85,4.5)
\qbezier(3.15,8.6)(3.15,8.6)(3.15,5.5)
\qbezier(3.15,8.6)(3.15,8.6)(4.8,6.95)
\qbezier(3.15,1.4)(3.15,1.4)(3.15,4.5)
\qbezier(3.15,1.4)(3.15,1.4)(4.8,3.05) \color{black} \thinlines
\put(3,9){\line(-1,-1){2}} \put(3,9){\line(0,-1){3.5}}
\put(3,9){\line(1,-1){2}} \put(2.7,4.8){$\nu$}
\put(3,1){\line(-1,1){2}} \put(3,1){\line(0,1){3.5}}
\put(3,1){\line(1,1){2}}

\put(1.7,8){\line(0,-1){0.3}} \put(2,7.7){\line(-1,0){0.3}}
\put(4.3,8){\line(0,-1){0.3}} \put(4,7.7){\line(1,0){0.3}}
\qbezier(3,7.3)(3,7.3)(2.8,7.5) \qbezier(3,7.3)(3,7.3)(3.2,7.5)
\put(1.7,2){\line(1,0){0.3}} \put(2,2){\line(0,1){0.3}}
\put(4.3,2){\line(-1,0){0.3}} \put(4,2){\line(0,1){0.3}}
\qbezier(3,2.3)(3,2.3)(2.8,2.5) \qbezier(3,2.3)(3,2.3)(3.2,2.5)

\qbezier(3,9)(1.5,9)(1,10) 
\qbezier(3,1)(1.5,1)(1,0)
\qbezier(3,9)(4.5,9)(5,10) \qbezier(3,1)(4.5,1)(5,0)
\put(6.7,0){\line(0,1){10}}
\end{picture}
\begin{picture}(6,9)(0,1)
\put(3,9){\circle*{0.2}} \put(2.9,9.5){$b_0$}
\put(3,1){\circle*{0.2}} \put(2.9,0.1){$b_1$}

\put(0.5,5){$F$} \put(0.5,1){$\orient$}  \put(2.6,9.05){$\to$}
\put(2.6,0.55){$\to$}

\linethickness{0.5mm}
\color[rgb]{0.50,0.50,0.50} \qbezier(2.85,8.6)(2.85,8.6)(1.2,6.95)

\qbezier(2.85,1.4)(2.85,1.4)(1.2,3.05)

\qbezier(3.15,8.6)(3.15,8.6)(4.8,6.95)

\qbezier(3.15,1.4)(3.15,1.4)(4.8,3.05)

\qbezier(3.15,8.6)(3,8.8)(2.85,8.6)

\qbezier(3.15,1.4)(3,1.2)(2.85,1.4) \color{black}\thinlines

\put(3,9){\line(-1,-1){2}}
\multiput(3,9)(0,-0.4){9}{\line(0,-1){0.2}}
\put(3,9){\line(1,-1){2}} \put(2.7,4.8){$\nu$}
\put(3,1){\line(-1,1){2}} \multiput(3,1)(0,0.4){9}{\line(0,1){0.2}}
\put(3,1){\line(1,1){2}}

\put(1.7,8){\line(0,-1){0.3}} \put(2,7.7){\line(-1,0){0.3}}
\put(4.3,8){\line(0,-1){0.3}} \put(4,7.7){\line(1,0){0.3}}
\qbezier(3,7.3)(3,7.3)(2.8,7.5) \qbezier(3,7.3)(3,7.3)(3.2,7.5)
\put(1.7,2){\line(1,0){0.3}} \put(2,2){\line(0,1){0.3}}
\put(4.3,2){\line(-1,0){0.3}} \put(4,2){\line(0,1){0.3}}
\qbezier(3,2.3)(3,2.3)(2.8,2.5) \qbezier(3,2.3)(3,2.3)(3.2,2.5)

\qbezier(3,9)(1.5,9)(1,10) 
\qbezier(3,1)(1.5,1)(1,0)
\qbezier(3,9)(4.5,9)(5,10) \qbezier(3,1)(4.5,1)(5,0)
\put(5.5,5){$\rightsquigarrow$}
\end{picture}
\begin{picture}(5,9)(0,1)
\put(3,9){\circle*{0.2}} \put(2.9,9.5){$b_0$}
\put(3,1){\circle*{0.2}} \put(2.9,0.1){$b_1$}

\put(5,5){$F$} \put(5,1){$\orient$} \put(2.6,9.05){$\to$}
\put(2.6,0.55){$\to$}
\linethickness{0.5mm}
\color[rgb]{0.50,0.50,0.50} \qbezier(2.85,8.6)(2.85,8.6)(1.2,6.95)
\qbezier(2.85,8.6)(2.85,8.6)(2.85,5.5)
\qbezier(2.85,1.4)(2.85,1.4)(1.2,3.05)
\qbezier(2.85,1.4)(2.85,1.4)(2.85,4.5)
\color[rgb]{0.75,0.75,0.75} \qbezier(3.15,8.6)(3.15,8.6)(3.15,5.5)
\qbezier(3.15,8.6)(3.15,8.6)(4.8,6.95)
\qbezier(3.15,1.4)(3.15,1.4)(3.15,4.5)
\qbezier(3.15,1.4)(3.15,1.4)(4.8,3.05) \color{black}\thinlines

\put(3,9){\line(-1,-1){2}} \put(3,9){\line(0,-1){3.5}}
\put(3,9){\line(1,-1){2}} \put(2.7,4.8){$\nu$}
\put(3,1){\line(-1,1){2}} \put(3,1){\line(0,1){3.5}}
\put(3,1){\line(1,1){2}}

\put(1.7,8){\line(0,-1){0.3}} \put(2,7.7){\line(-1,0){0.3}}
\put(4.3,8){\line(0,-1){0.3}} \put(4,7.7){\line(1,0){0.3}}
\qbezier(3,7.3)(3,7.3)(2.8,7.5) \qbezier(3,7.3)(3,7.3)(3.2,7.5)
\put(1.7,2){\line(1,0){0.3}} \put(2,2){\line(0,1){0.3}}
\put(4.3,2){\line(-1,0){0.3}} \put(4,2){\line(0,1){0.3}}
\qbezier(3,2.3)(3,2.3)(2.8,2.5) \qbezier(3,2.3)(3,2.3)(3.2,2.5)

\qbezier(3,9)(1.5,9)(1,10) 
\qbezier(3,1)(1.5,1)(1,0) 
\qbezier(3,9)(4.5,9)(5,10) \qbezier(3,1)(4.5,1)(5,0)

\end{picture}
\end{center}
\caption{Before and after cutting along the arc $\nu$ -- the two
cases.}\label{b:rand}\end{figure}
 Taking the case $\nu\in \D_0(F_\al;2)$ (left-hand side of Figure \ref{b:rand}), when we cut along
$\nu$ we get one boundary component instead of two. So we get $k-1$
boundary components in this case. In the case $\nu\in \D_0(F_\al;1)$ (right-hand side of Figure \ref{b:rand}) cutting along $\nu$
splits the boundary component into two, so we get $k+1$ boundary
components.
\end{proof}
Combining Lemma \ref{l:nu} and Prop. \ref{p:genus}, we have proved,
\begin{cor}\label{c:Sgenus}For $\al\in\D_0(F;i)$, let $\al'=\al\cup \nu$ as in Lemma \ref{l:nu}. Then:
\begin{equation*}
    S(\al')=\left\{
                       \begin{array}{ll}
                         S(\al), & \hbox{if $\nu\in \D_0(F_\al;1)$;} \\
                         S(\al)+1, & \hbox{if $\nu\in \D_0(F_\al;2)$.}
                       \end{array}
                     \right.
\end{equation*}
and
\begin{equation*}
    \genus(F_{\al'})=\left\{
                       \begin{array}{ll}
                         \genus(F_\al)-1, & \hbox{if $\nu\in \D_0(F_\al;1)$;} \\
                         \genus(F_\al), & \hbox{if $\nu\in \D_0(F_\al;2)$.}
                       \end{array}
                     \right.
\end{equation*}\qed
\end{cor}

\begin{lem}\label{l:Slignul}Let $\al\in\D_0(F;i)$. Then $S(\al)=0$ if and only if
\begin{itemize}
  \item[$(i)$]for $i=1$, $P(\al)=\id$.
  \item[$(ii)$]for $i=2$, $P(\al)$ is a cyclic permutation, i.e. one of the following:
  \begin{equation*}
\id,[1\, 2\cdots n \,0], [2 \,3  \cdots n \,0 \,1], \cdots, [n\, 0\,
1\cdots n\!\!-\!\!1].
\end{equation*}
\end{itemize}
\end{lem}
\begin{proof} We prove ''only if''. The converse is clear, e.g. by Prop. \ref{p:grim} and Prop. \ref{p:genus} $(i)$.

By Cor. \ref{c:Sgenus}, any subsimplex of $\al$ has
genus equal to or lower than $S(\al)=0$, so any subsimplex of
$\al$ must have genus 0. If $\al\in \D_n(F;1)$, this means all
1-subsimplices must have permutation equal to the identity, and this
forces $P(\al)=\id$. If $\al\in \D_n(F;2)$ the condition on
1-subsimplices is vacuous, but for a $2$-subsimplex $\be$ of $\al$,
we see by Cor. \ref{c:Sgenus} that $S(\be)=0$ implies that $P(\be)$
is either $\id$, $[1\, 2\, 0]$, or $[2\, 0\, 1]$. For this to hold
for any 2-subsimplex of $\al$, $P(\al)$ must be as stated in $(ii)$.
\end{proof}

\subsection{More about permutations}
By Prop. \ref{p:grim}, given $\al\in\D_n(F;i)$, the
number $\sharp \dd N(\al)$ is a function only of $P(\al)$ and $i$.
By Prop. \ref{p:genus}$(i)$, the same is true for $S(\al)$. Thus,
given a permutation $\s\in\Si_{n+1}$, we can calculate these
quantities and simply define the numbers $\sharp\dd N(\s)$ and
$S(\s)$ by the formulas of Prop. \ref{p:grim} and
\ref{p:genus}$(i)$.

Now we are going to see that given a permutation $\s\in \Si_{n+1}$,
there exists $\al\in\D_n(F_{g,r};i)$ with $P(\al)=\s$
if at all possible, that is, provided the formula $(iii)$ of Prop. \ref{p:genus} for the genus of $F_\al$ gives a non-negative
result. Rearranging this conditions we have the following lemma,
also stated in \cite{Harer2}:

\begin{lem}\label{l:perm}Given a permutation $\s\in \Si_{n+1}$, let $s=S(\s)$ as above. There exists $\al\in\D_0(F;i)$ with $P(\al)=\s$ if and only if
\begin{equation}\label{e:spicies}
  s \ge n-g+2-i.
\end{equation}
\end{lem}
\begin{proof}
%

 Given a permutation $\s$, one can try to construct an
arc simplex $\al$ inductively with $P(\al)=\s$ by first choosing an
arc $\al_0\in\D_0(F;i)$ from $b_0$ to $b_1$,
and cutting $F$ up along it. This will give us two copies of $b_0$
and $b_1$, respectively, one to the left of our arc and one to the
right. The permutation determines from which copy of $b_0$ and $b_1$
a new arc will join.

Suppose we have constructed $k+1\le n+1$ arcs as above, i.e. a
$k$-simplex $\be=(\al_0,\ldots,\al_{k})$, and consider the cut surface $F_\be$. Inductively we assume that $F_\be$ is connected. Now we
must verify that when adding a new arc, $\nu$, as in Lemma \ref{l:nu},
the cut surface $(F_\be)_\nu$ is connected. If this holds,
$\be\cup\nu$ is a $(k+1)$-simplex, and we have completed the induction
step.

There are two cases. First assume that $\nu$ must join two
different boundary components of $F_\be$. Then
$(F_\be)_\nu$ is connected, no matter how we choose $\nu$,
since $F_\be$ is connected.

Secondly, if $\nu$ connects two points on the same boundary component of $F_\be$, we choose $\nu$ so that it winds
around a genus-hole in $F_\be$. This ensures that
$(F_\be)_\nu$ is connected, so we must prove that
genus$(F_\be)\ge 1$. From Prop. \ref{p:genus}, we know that
$\textup{genus}(F_\be)= g+S(\be)-(k+2-i)$,
and we want to prove
\begin{equation}\label{e:inequality}
    S(\be) - k \ge s - n + 1.
\end{equation}
Using this, we can complete the induction step:
\begin{equation*}
    \textup{genus}(F_\be)= g+S(\be)-k-2+i \ge g + s -n -1+i \ge
    1
\end{equation*}
by assumption \eqref{e:spicies}.

To prove \eqref{e:inequality}, recall that $S(\be)$ only depends on $P(\be)$, not on the surface $F$. So consider another surface $F'$ with genus $g'>n$. We can construct $\be'\in\D_{k}(F',i)$ with $P(\be')=P(\be)$, as above. We can further construct $\al'\in\D_n(F',i)$ with $\be'$ as a subsimplex and $P(\al')=\s$, simply by adding $n-k$ new arcs to $\be'$ which each wind around a genus-hole in $F'$. This is possible because $g'>n$. We claim
\begin{equation}\label{e:help}
S(\al')\le S(\be')+n-k-1.
\end{equation}
Applying Cor. \ref{c:Sgenus} $n-k$ times to $\be'$, we obviously get $S(\al')\le S(\be')+n-k$. We get the extra $-1$, because the first time we add an arc $\nu'$ to $\be'$ we have $\nu'\in \D_0(F'_{\be'};1)$, since $\nu\in\D_0(F_\be,1)$ by assumption. This proves  \eqref{e:help}. Since $P(\be')=P(\be)$ and $P(\al')=\s$, \eqref{e:help} implies
$s=S(\s)\le S(\be)+n-k-1$. This proves \eqref{e:inequality}.
\end{proof}

Combining Prop. \ref{p:inj} and Lemma \ref{l:perm} we have proved,
\begin{cor}
\label{c:bij}The permutation map
\begin{equation*}
P:\D_n(F;i)/\G(F) \To \Si_{n+1}
\end{equation*}
is bijective if $n\le g-2+i$.\qed
\end{cor}

%
%
\begin{lem}[\cite{Harer4}]\label{l:Hlignul}For $F=F_{g,b}$ with $g\ge 2$, the sequence
\begin{equation*}
 C_{p+1}(F;i)/\G(F) \stackrel{d^1}{\To} C_p(F;i)/\G(F) \stackrel{d^1}{\To} C_{p-1}(F;i)/\G(F)
\end{equation*}
is split exact for $1\le p\le g-2+i$.
\end{lem}
\begin{proof}Let $\Z\Si_*$ denote the chain complex with chain groups
$\Z\Si_n$, $n\ge 1$, and differentials
\begin{equation*}
    \dd:\Z\Si_{n+1}\To \Z\Si_{n}
\end{equation*}
given as follows: For $\s=[\s(0) \cdots \s(n)]\in\Si_{n+1}$, let
\begin{equation*}
\dd_j(\s)=[\s(0)\cdots\s(j-1)\, \s(j+1)\ldots\s(n)],
\end{equation*}
where the set $\set{0,1,\ldots,n}\fra\set{\s(j)}$ is identified with $\set{0,1,\ldots,n-1}$ by subtracting $1$ from all numbers exceeding $\s(j)$. Then we define
$\dd(\s)=\sum_{j=0}^n(-1)^j\dd_j(\s)$ and extend linearly. Extending
the permutation map $P$ linearly leads to the
commutative diagram
\begin{equation}\label{e:comm}
     \xymatrix{C_n(F;i)/\G(F)\ar[r]^d\ar[d]^P& C_{n-1}(F;i)/\G(F)\ar[d]^P\\
     \Z\Si_{n+1}\ar[r]^{\dd}&\Z\Si_{n}
     }
\end{equation}
i.e. a chain map $C_*(F;i)/\G(F) \To \Z\Si_*$.
By Prop. \ref{p:inj}, $P$ is injective, so $C_*(F;i)/\G(F)$ is
isomorphic to a subcomplex of $\Z\Si_*$, namely the subcomplex generated
by permutations $\s\in\Si_{n+1}$ with $S(\s)$ satisfying the
requirements of Lemma \ref{l:perm}. In particular, for $n\le g-2+i$, the
chain groups of $\Z\Si_*$ and of $C_*(F;i)/\G(F)$ are identified.

Define $D:\Z\Si_n\To \Z\Si_{n+1}$ by
\begin{equation}\label{e:D}
D(\s)=\hat\s=
[0\quad\s(0)\!+\!1\quad\s(1)\!+\!1\quad\cdots\quad\s(n)\!+\!1].
\end{equation}
It is an easy consequence of the definitions that $D\dd+ \dd D=1$,
so $D$ is a contracting homotopy and $\Z\Si_*$ is split exact. By the diagram \eqref{e:comm}, $C_*(F;i)/\G(F)$ is also split exact in the range where
\begin{equation}\label{e:billede}
D\circ P\Big(C_n(F;i)/\G(F)\Big)\del P\Big(C_{n+1}(F;i)/\G(F)\Big),
\end{equation}
since $D$ lifts to  a contracting homotopy $\bar{D}$ of $C_{*}(F;i)/\G(F)$.

We will first consider $C_*(F;1)/\G(F)$. By Cor. \ref{c:bij}, $P$ is
bijective for $n\le g-1$, so \eqref{e:billede} is satisfied for
$n\le g-2$. It remains to consider the degree $n=g-1$. We have the
commutative diagram,
\begin{equation*}
     \xymatrix{C_{g}(F;i)/\G(F)\ar[r]^d\ar@{^(->}[d]^P& C_{g-1}(F;i)/\G(F)\ar[r]^d\ar[d]^P_{\iso}& C_{g-2}(F;i)/\G(F)\ar[d]^P_{\iso}\\
     \Z\Si_{g+1}\ar[r]^{\dd}&\Z\Si_{g}\ar[r]^{\dd}&\Z\Si_{g-1}
     }
\end{equation*}
with the bottom sequence exact. We must show that
\begin{equation*}
P\circ d(C_{g}(F;i)/\G(F))=\dd(\Z\Si_{g+1}).
\end{equation*}
According to Cor. \ref{c:bij}, $P:C_{g}(F;1)/\G(F)\To\Z\Si_{g+1}$ hits everything except what is
generated by permutations $\s$ with $S(\s)=0$. Thus we must show
$\dd(\s)\in \im(P\circ d)=\im(\dd\circ P)$ for all $\s\in\Si_{g+1}$
with $S(\s)=0$. From Lemma \ref{l:Slignul} we know that the only such
permutation is the identity. As
\begin{equation*}
\dd([0\,1\,\cdots\,
g])=\sum_{j=0}^{g}(-1)^{j}[0\,1\,\cdots\,g\!-\!1]=\left\{
                                                             \begin{array}{ll}
                                                               0, & \hbox{if $g$ is odd,} \\
                                                               \id, & \hbox{if $g$ is even,}
                                                             \end{array}
                                                           \right.
\end{equation*}
we are done if $g$ is odd, and the desired contracting homotopy
$\bar{D}$ is obtained by lifting $D$ when $S(\al)>0$ and setting by
$\bar{D}(\al)=0$ when $S(\al)=0$.

If $g$ is even, consider $\t=[2\,0\,1\,3\,4\,\cdots\,g]\in
\Si_{g+1}$. Then by Lemma \ref{l:Slignul} $S(\t)>0$, and
\begin{eqnarray*}
\dd(\t)&=&[0\,1\,2\,\cdots \,g\!-\!1]-[1\,0\,2\,3\,\cdots\,g\!-\!1]+[1\,0 \,2\,3\,\cdots\,g\!-\!1]\\
&&+
\sum_{j=3}^{g}(-1)^{j}[2\,0\,1\,3\,4\,\cdots\,g\!-\!1]=[0\,1\,2\,\cdots\,g\!-\!1]
= \dd[0\,1\,2\,\cdots\,g].
\end{eqnarray*}
Thus we can obtain a contracting homotopy $\bar{D}$ by taking $\bar
D(\al)=P^{-1}(\t)$ when $S(\al)=0$.

For $C_*(F;2)/\G(F)$, Cor. \ref{c:bij} gives that $P$ is
bijective for $n\le g$, so we are left with $j=g$, where we use
exactly the same method as above. We must show that
$\dd(\s)\in\im(\dd\circ P)$ for all $\s\in\Si_{g+2}$ with $S(\s)=0$.
We only need to consider $\s\in\im(D)$, because $\im\dd=\im(\dd\circ D)$ by the equation $\dd D+D\dd=1$. The only $\s\in\Si_{g+2}$
with $S(\s)=0$ and $P\in\im D$ is the identity, according to Lemma
\ref{l:Slignul}. Now we are in the same situation as above, so we
can use $\t=[2\,0\,1\,3\,4\,\cdots\,g\,g\!+\!1]\in \Si_{g+2}$ which
has genus $S(\t)>0$ in $C_*(F;2)$, since $g\ge 2$.
\end{proof}

\newpage

\section{Homology stability of the mapping class group}

Let $F$ be a surface with boundary. Given $F$ we can glue on a
''pair of pants'', $F_{0,3}$, to one or two boundary components. We
denote the resulting surface by $\Si_{i,j}F$, the subscripts
indicating the change in genus and number of boundary components,
respectively.

\begin{figure}[!hbt]
\begin{center}
\setlength{\unitlength}{0.5cm}
\begin{picture}(7,2)(-2,1)
\put(-1,0.6){\large \emph{F}} \qbezier(-2,0)(-2,0)(1,0)
\qbezier(-2,2)(-2,2)(1,2) \qbezier(1,0)(0.3,1)(1,2)
 \qbezier(1,0)(2.5,0)(4,-0.5)
\qbezier(1,2)(3,2)(4,2.5) \qbezier(4,0.5)(3.5,0)(4,-0.5)
\qbezier(4,1.5)(3.5,2)(4,2.5) \qbezier(4,0.5)(4.5,0)(4,-0.5)
\qbezier(4,1.5)(4.5,2)(4,2.5) \qbezier(4,0.5)(2,1)(4,1.5)
\linethickness{0.05mm}\qbezier(1,0)(1.7,1)(1,2)
\end{picture}
\begin{picture}(7,2)(0,1)
\put(1.6,0.6){\large \emph{F}}
\qbezier(1,0)(2.5,0)(4,-0.5)\qbezier(1,2)(3,2)(4,2.5)
\qbezier(4,0.5)(3.5,0)(4,-0.5) \qbezier(4,1.5)(3.5,2)(4,2.5)
\qbezier(4,0.5)(2,1)(4,1.5) \qbezier(4,0.5)(6,1)(4,1.5)
\qbezier(7,0)(6.3,1)(7,2) \qbezier(7,0)(5.5,0)(4,-0.5)
\qbezier(7,2)(5,2)(4,2.5) \qbezier(7,0)(7.7,1)(7,2)
\linethickness{0.05mm} \qbezier(4,0.5)(4.5,0)(4,-0.5)
\qbezier(4,1.5)(4.5,2)(4,2.5)
\end{picture}
\end{center}\caption{\quad $\Si_{0,1}F$ \qquad and \qquad $\Si_{1,-1}F$.\quad\quad}\label{b:sigma}\end{figure}
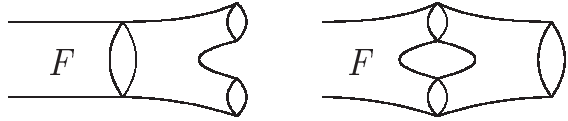These two operations induce homomorphisms between the mapping class
groups after extending a mapping class by the identity on the pair of pants;
\begin{equation*}
    \Si_{i,j}:\G(F)\To \G( \Si_{i,j}F).
\end{equation*}
Given a surface $F$, applying $\Si_{0,1}$ and then adding a disk at one of the
pant legs gives a surface diffeomorphic to $F$ (with a
cylinder glued onto a boundary component). It is easily seen
that the induced composition
\begin{equation*}
    \G(F)\To \G( \Si_{0,1}F)\To \G(F)
\end{equation*}
is the identity, so $\Si_{0,1}$ induces an
injection on homology
\begin{equation}\label{e:injektiv}
    H_n(\G(F))\into H_n(\G(\Si_{0,1}F)).
\end{equation}

For the proof of the stability theorems, the opposite operation is
essential: One expresses the surface $F$ as the result of cutting
$\Si_{0,1}F$ or $\Si_{1,-1}F$ along an arc representing a
$0$-simplex in one of the arc complexes of definition \ref{d:arc}:
\begin{equation*}
    F \iso (\Si_{0,1}F)_\al,\quad \textup{and} \quad F \iso (\Si_{1,-1}F)_\be,
\end{equation*}
for $\al\in \D_0(\Si_{0,1}F,2)$ and $\be\in \D_0(\Si_{1,-1}F,1)$ as indicated below
\begin{figure}[!hbt]\label{f:skaer}
\begin{center}
\setlength{\unitlength}{0.5cm}
\begin{picture}(9,3)(-2,0)
\put(-1,0.6){\large \emph{F}} \qbezier(-2,0)(-2,0)(1,0)
\qbezier(-2,2)(-2,2)(1,2) \qbezier(1,0)(0.3,1)(1,2)
 \qbezier(1,0)(2.5,0)(4,-0.5)
 \qbezier(1,2)(3,2)(4,2.5) \qbezier(4,0.5)(3.5,0)(4,-0.5)
\qbezier(4,1.5)(3.5,2)(4,2.5) \qbezier(4,0.5)(4.5,0)(4,-0.5)
\qbezier(4,1.5)(4.5,2)(4,2.5) \put(5,1){\vector(-3,-1){1}}
\thicklines\qbezier(4,0.5)(2,1)(4,1.5) \put(5.3,1){$\al$}
\thicklines\qbezier(4,0.5)(2,1)(4,1.5)
\linethickness{0.05mm}\qbezier(1,0)(1.7,1)(1,2)
\end{picture}
\begin{picture}(7,3)
\put(1.6,0.6){\large \emph{F}}
\qbezier(1,0)(2.5,0)(4,-0.5)\qbezier(1,2)(3,2)(4,2.5)
\qbezier(4,0.5)(3.5,0)(4,-0.5) \qbezier(4,1.5)(3.5,2)(4,2.5)
\qbezier(4,0.5)(2,1)(4,1.5) \qbezier(4,0.5)(6,1)(4,1.5)
\qbezier(7,0)(6.3,1)(7,2) \qbezier(7,0)(5.5,0)(4,-0.5)
\qbezier(7,2)(5,2)(4,2.5) \qbezier(7,0)(7.7,1)(7,2)
\put(7,-0.2){\vector(-3,2){1.2}} \put(7.2,-0.5){$\be$} \thicklines
\qbezier(5,1)(6,0.7)(6.65,0.7) \qbezier(7.3,1.3)(7,1.35)(6.7,1.3)
\linethickness{0.05mm} \qbezier(4,0.5)(4.5,0)(4,-0.5)
\qbezier(4,1.5)(4.5,2)(4,2.5) \qbezier(5,1)(6,1.3)(6.65,1.3)
\end{picture}
\end{center}\caption{$\al$ and $\be$.}\label{b:arc}\end{figure}

A diffeomorphism of $F_\al$ that fixes the points on the boundary pointwise extends to a diffeomorphism of $F$ by adding the identity on $N(\al)$, and this defines an inclusion $\G(F_\al)\To \G$ whose image is the stabilizer $\G_\al$.

\subsection{The spectral sequence for the action of the mapping class
group}\label{sc:lemmas} In this section, $F=F_{g,r}$ with $g\ge 2$ and $\G=\G(F)$. We shall consider the spectral sequences $E^{n}_{p,q}=E^{n}_{p,q}(F;i)$ from section \ref{sc:spectral} associated to the action of $\G$ on the arc complexes $C_*(F;i)$ for $i=1,2$. By Cor. \ref{c:spectral} and
Thm. \ref{s:connected}, we have $E^1_{0,q}=H_q(\G)$ and
\begin{equation}\label{e:E1F}
    E^1_{p,q}= \bigoplus_{\al\in\Bar\D_{p-1}} H_q(\G_\al) \Tto 0,
    \quad \textup{for }p+q \le 2g-2+i,
\end{equation}
where $\Dbar_{p-1}\del \D_{p-1}(F;1)$ is a set of representatives of the $\G$-orbits of $\D_{p-1}(F;i)$ in $C_*(F;i)$.

The permutation map
\begin{equation*}
P:\D_{p-1}(F;i)/\G \To \Si_{p}
\end{equation*}
is injective by Prop. \ref{p:inj}. Let $\Sibar_p$ be the image, and $T:\Sibar_p\stackrel{\sim}{\To}\Dbar_{p-1}\into\D_{p-1}(F;i)$ a section, $P\circ T=\id$. Then
\begin{equation}\label{e:E1Fny}
    E^1_{p,q}= \bigoplus_{\s\in\Sibar_p} E^1_{p,q}(\s),\quad E^1_{p,q}(\s)=H_q(\G_{T(\s)}).
\end{equation}

The first differential, $d^1_{p,q}:E^1_{p,q}\To E^1_{p-1,q}$, is
described in section \ref{sc:diff}. The diagrams
\begin{equation*}
    \xymatrix{
    \D_{p}(F;i)\ar[r]^{\dd_j}\ar[d] & \D_{p}(F;i)\ar[d] &\\
    \Sibar_{p+1}\ar[r]^{\dd_j} & \Sibar_p & j=0,\ldots,p
    }
\end{equation*}
commute, where $\dd_j$ omits entry $j$ as in Def. \ref{d:arc} and the vertical arrows divide out the $\G$ action and compose with $P$. Thus for each $\s\in\Sibar_{p+1}$, there is $g_j\in\G$ such that
\begin{equation}\label{e:Ib18}
    g_j\cdot\dd_j T(\s) = T(\dd_j\s),
\end{equation}
and conjugation by $g_j$ induces an isomophism $c_{g_j}:\G_{\dd_j T(\s)}\To \G_{T(\dd_j\s)}$. The induced map on homology is denoted $\dd_j$ again, i.e.
\begin{equation}\label{e:Ib19}
\xymatrix{
    \dd_j: H_q(\G_{T(\s)})\ar[r]^{\textrm{incl}_*}&
    H_q(\G_{\dd_j T(\s)})\ar[r]^{(c_{g_j})_*}&
    H_q(\G_{T(\dd_j\s)})
}.
\end{equation}
Note that $(c_{g_j})_*$ does not depend on the choice of $g_j$ in \eqref{e:Ib18}: Another choice $g_j'$ gives $c_{g_j'}= c_{g_j'g_j^{-1}}c_{g_j}$, and $g_j'g_j^{-1}\in\G_{T(\dd_j\s)}$ so $c_{g_j'g_j^{-1}}$ induces the identity on $H_q(\G_{T(\dd_j\s)})$. Then
\begin{equation}\label{e:Ib20}
    d^1=\sum_{j=0}^{p-1}(-1)^j\dd_j.
\end{equation}

The proof of the main stability Theorem depends on a partial calculation of the spectral sequence \eqref{e:E1F}. More specifically, the first differential $d^1:E^1_{1,q}\To E^1_{0,q}$ is equivalent to a stability map $H_q(\G_\al)\To H_q(\G)$, so the question becomes whether $d^1$ is an isomorphism resp. an epimorphism. In a range of dimensions the spectral sequence converges to zero, so that $d^1$ must be an isomorphism unless other (higher) differentials interfere. The next three lemma are the key elements that give sufficient hold of the spectral sequence. The first lemma gives the general induction step. The next two lemmas about $d^1:E^1_{p,q}\To E^1_{p-1,q}$ for $p=3,4$ are necessary for the improved stability.

\begin{lem}\label{l:E^2_p,q} Let $i=1,2$, and let $k,j\in \N$ with $k\le g-3+i$. For any $\al\in\D_{p-1}(F;i)$ and all $q\le k-j$, assume that
\begin{eqnarray}\label{e:assume}
  & &H_q(\G_\al)\stackrel{\scriptscriptstyle\iso}{\to} H_q(\G) \text{ is an isomorphism} \quad  \text{if } p+q\le k+1, \\
  & &H_q(\G_\al)\twoheadrightarrow H_q(\G) \text{ is surjective} \qquad\qquad\!  \text{if } p+q=
  k+2.
\end{eqnarray}
Then $E^2_{p,q}(F;i) = 0$ for all $p,q$ with $p+q=k+1$ and $q\le
k-j$.
\end{lem}

\begin{proof}Let $\overline{C}_n(F;i)=C_n(F;i)/\G$. By
\eqref{e:E1F} and the assumptions, we get for $q\le k-j$:
\begin{eqnarray}\label{e:conclude}
  & &E^1_{p,q}\iso \overline{C}_{p-1}(F;i)\otimes
H_{q}(\G) \quad  \text{if } p+q\le k+1, \\
  & &E^1_{p,q}\twoheadrightarrow
\overline{C}_{p-1}(F;i)\otimes H_{q}(\G) \quad \text{if } p+q=
  k+2.\nonumber
\end{eqnarray}
Now we have the following commutative diagram, for a fixed pair $p,
q$ with $q\le k-j$ and $p+q=k+1$:
\begin{equation}\label{e:diagram}
\xymatrix{E^1_{p-1,q}\ar[d]^{\iso}
  &E^1_{p,q}\ar[l]_{\quad d^1}\ar[d]^{\iso}&
  E^1_{p+1,q}\ar[l]_{d^1}\ar@{->>}[d]\\
\overline{C}_{p-2}(F;i) \otimes H_{q}(\G)& \overline{C}_{p-1}(F;i)
\otimes H_{q}(\G)\ar[l]_{\quad\bar{d}^1}& \overline{C}_{p}(F;i)
\otimes H_{q}(\G)\ar[l]_{\quad\bar{d}^1}
  }
\end{equation}
Using the formula \eqref{e:Ib20} for $\bar{d}^1$, $(c_{g_j})_*(\om)=\om$ for $\om\in H_*(\G)$, since conjugation induces the identity in $H_*(\G)$. Thus the bottom row of diagram \eqref{e:diagram} is just the sequence from Lemma \ref{l:Hlignul}, tensored with $H_{q}(\G)$. Since $p\le k+1\le g-2+i$ that sequence is split exact, so the bottom row of \eqref{e:diagram} is exact. We conclude that $E^2_{p,q}=0$ for all $p,q$ with
$q\le k-j$ and $p+q=k+1$, as desired.
\end{proof}

We next examine the chain complex
\begin{equation*}
\xymatrix{    \ldots\ar[r]^{d^1}& E^1_{3,q}(F,i) \ar[r]^{d^1}& E^1_{2,q}(F,i)\ar[r]^{d^1}& E^1_{1,q}(F,i)\ar[r]^{d^1}&E^1_{0,q}(F,i)}
\end{equation*}
associated  with $C(F;i)$, but first we need an easy geometric proposition.
Recall from definition \ref{d:N}, that for $\al \in \D_p(F;i)$ we write $F_\al =F\fra N(\al)$ for the surface cut along the arcs of $\al$.

\begin{prop}\label{p:8-tal} Let $\al\in \D_n(F;i)$ with permutation $P(\al)=\s$, and assume there is $k,l<n$ such that $\s(k)=l+1$ and $\s(k+1)=l$. Then there exists $f\in \G(F)$ with $f(\al_{k+1})=\al_k$, $f(\al_i)=\al_i$ for $i\notin \set{k, k+1}$ and $f|_{F_\al} =\id_{F_\al}$.
\end{prop}
\begin{proof} A (right) \emph{Dehn twist} in an annulus in $F$ is an element of $\G(F)$ given by performing a full twist to the right inside the annulus, and extending by the identity outside the annulus. Figure \ref{b:Dehn1} shows a Dehn twist $\g$ in an annulus, and its effect on a curve $\be$ intersecting the annulus.
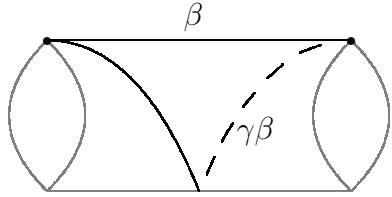
\begin{figure}[!hbt]
\begin{center}
\setlength{\unitlength}{1cm}
\begin{picture}(4,2)(-2,0.5)
\color[rgb]{0.50,0.50,0.50}
\qbezier(-2,0)(-2,0)(2,0)
\qbezier(-2,0)(-3,1)(-2,2)
\qbezier(-2,0)(-1,1)(-2,2)
\qbezier(2,0)(3,1)(2,2)
\qbezier(2,0)(1,1)(2,2)
\color{black}
\qbezier(-2,2)(-2,2)(2,2)
\qbezier(-2,2)(-0.8,2)(0,0)
\qbezier(0.1,0.22)(0.1,0.22)(0.2,0.45)
\qbezier(0.3,0.66)(0.3,0.66)(0.44,0.9)
\qbezier(0.58,1.12)(0.58,1.12)(0.73,1.3)
\qbezier(0.9,1.5)(0.9,1.5)(1.1,1.68)
\qbezier(1.33,1.84)(1.33,1.84)(1.6,1.94)
\put(-2,2){\circle*{0.1}}
\put(2,2){\circle*{0.1}}
\put(-.2,2.2){$\be$}
\put(0.5,.7){$\g\be$}
\end{picture}
\end{center}\caption{A Dehn twist $\g$ in an annulus.}\label{b:Dehn1}\end{figure}

Consider the curves $\al_k$ and $\al_{k+1}$. Take an annulus as depicted on Figure \ref{b:Dehn} below (in grey). By the requirements of the proposition it is easy to construct the annulus so that it only intersects $\al$ in $\al_k$ and $\al_{k+1}$. Let $f$ be the Dehn twist in this annulus.
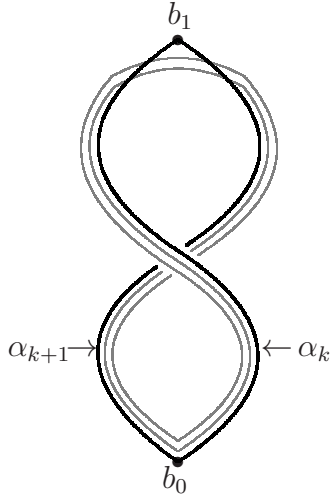
\begin{figure}[!hbt]
\begin{center}
\setlength{\unitlength}{0.7cm}
\begin{picture}(8,9)(-4,0)
\put(0,0){\circle*{0.2}}\put(-0.3,-.5){$b_0$}
\put(0,8){\circle*{0.2}}\put(-0.2,8.3){$b_1$}

\put(1.55,2){$\leftarrow$} \put(2.25,2){$\al_k$}
\put(-2.1,2){$\rightarrow$} \put(-3.2,2){$\al_{k+1}$}
 \color[rgb]{0.50,0.50,0.50}
\qbezier(0,0.2)(2.7,2)(0,3.8)\qbezier(0,3.8)(-2.55,5.5)(-1.2,7.1)
\qbezier(0,0.2)(-2.6,2)(-0.3,3.6)\qbezier(0.25,4)(2.5,5.5)(1.23,7.1)
\qbezier(-1.2,7.1)(0,7.8)(1.23,7.1)

\qbezier(0,0.4)(2.4,2)(0,3.6)\qbezier(0,3.6)(-2.85,5.4)(-1.25,7.3)
\qbezier(-0,0.4)(-2.4,2)(-0.13,3.53)\qbezier(0.38,3.92)(2.8,5.5)(1.25,7.3)
\qbezier(-1.25,7.3)(0,8)(1.25,7.3)

\color{black}
\thicklines \qbezier(0,0)(3,2)(0,4)\qbezier(0,4)(-3,6)(0,8)
\qbezier(0,0)(-2.8,2)(-0.4,3.7)\qbezier(0.15,4.1)(3,6)(0,8)
\end{picture}
\end{center}\caption{The Dehn twist $f$.}\label{b:Dehn}\end{figure}
Since $f$ is the identity outside the annulus, we have $f(\al_i)=\al_i$ for all $i\notin \{k, k+1\}$ and $f|_{F_\al} =\id_{F_\al}$. By Figure \ref{b:Dehn} it is easy to see that $f(\al_{k+1})=\al_k$.
\end{proof}

The stabilizer $\G_\al$ of $\al\in \D_p(F;i)$ depends up to conjugation only on the orbit $\G\al$, i.e. on $P(\al)\in \Si_{p+1}$. So when conjugation is of no importance we shall for $\s\in\overline{\Si}_{p+1}$ write $\G_\s$ for any of the conjugate subgroups $\G_\al$ with $P(\al)=\s$. If $\tau\in \overline{\Si}_{p}$ is a face of $\s\in \overline{\Si}_{p+1}$ then $\G_\s$ is conjugate to a subgroup of $\G_\tau$, and there is a homomorphism
\begin{equation*}
    H_q(\G_\s)\To H_q(\G_\tau),
\end{equation*}
well-determined up to isomorphism of source and target.

\begin{lem}\label{l:E^2_2,q}Let $c_1$ and $c_2$ be the isomorphism classes
\begin{equation*}
    c_1: H_q(\G_{[0\,2\,1]})\To H_q(\G_{[1\,0]}),\quad c_2: H_q(\G_{[1\,2\,0]})\To H_q(\G_{[0\,1]})
\end{equation*}
\begin{itemize}
  \item[$(i)$]If $c_1$ and $c_2$ are surjective, then  $d^1_{3,q}:E^1_{3,q}\To E^1_{2,q}$ is surjective, and
  $E^2_{2,q}=0$.
  \item[$(ii)$]If $c_1$ and $c_2$ are injective, then
\begin{equation*}
d^1_{3,q}:E^1_{3,q}([0\,2\,1])\oplus E^1_{3,q}([1\,2\,0])\To E^1_{2,q}
\end{equation*} is injective.
\end{itemize}
\end{lem}

\begin{proof}The target of $d^1$ is $E^1_{2,q}=E^1_{2,q}([0\,1])\oplus E^1_{2,q}([1\,0])$, and we first examine the component
\begin{equation}\label{e:Ib25}
   d^1_{3,q}: E^1_{3,q}([0\,2\,1])\To E^1_{2,q}([0\,1]).
\end{equation}
If $\be=T([0\,2\,1])$ with $\be= (\be_0,\be_1,\be_2)$, let $\g\in\G$ satisfy $(\g\be_0,\g\be_1)=T([0\,1])$, and write $\al=\g\be$. Then
\begin{equation*}
    (c_g)_*:E^1_{3,q}([0\,2\,1])\stackrel{\iso}{\To}H_q(\G_\al),
\end{equation*}
and the $E^1_{2,q}([0\,1])$-component of $d^1_{3,q}\circ (c_g)_*$ is the difference of
\begin{eqnarray}\label{e:Ib26}
  \dd_2:H_q(\G_\al) &\To& H_q(\G_{(\al_0,\al_1)}) \\
  \dd_1:H_q(\G_\al) &\To& H_q(\G_{(\al_0,\al_2)})\To H_q(\G_{(\al_0,\al_1)})\nonumber
\end{eqnarray}
where $f\cdot(\al_0, \al_2)=(\al_0, \al_1)$. By the previous proposition \ref{p:8-tal} we may choose $f$ such that $f|_{F_\al}=\id_{F_\al}$. It follows that $c_f:\G\To \G$ restricts to the identity on $\G_\al$, and hence that the two maps in \eqref{e:Ib26} are equal. Thus the component of $d^1_{3,q}$ in \eqref{e:Ib25} is zero. On the other hand, the component
\begin{equation*}
     d^1_{3,q}: E^1_{3,q}([0\,2\,1])\To E^1_{2,q}([1\,0])
\end{equation*}
is equal to $\dd_0$, so it belongs to the isomorphism class $c_1$. Thus it is surjective resp. injective under the assumptions $(i)$ resp. $(ii)$.

The restriction of $d^1_{3,q}$ to $E^1_{3,q}([1\,2\,0])$,
\begin{equation*}
   d^1_{3,q}: E^1_{3,q}([1\,2\,0])\To E^1_{2,q}([0\,1])\oplus E^1_{2,q}([1\,0]),
\end{equation*}
is treated in a similar fashion. This time there are two terms with opposite signs in $E^1_{2,q}([1\,0])$ which cancel by Prop. \ref{p:8-tal}, and the component
\begin{equation*}
   d^1_{3,q}: E^1_{3,q}([1\,2\,0])\To E^1_{2,q}([0\,1])
\end{equation*}
is in the isomorphism class of $c_2$. This proves the lemma.
\end{proof}

We next consider the situation of Lemma \ref{l:E^2_2,q}$(ii)$ where $c_1$ and $c_2$ are injective. If we further assume that $g(F)\ge 3$, then $\Sibar_3=\Si_3$ and $\Sibar_4=\Si_4\fra \set{\id}$. We consider the maps
\begin{eqnarray}\label{e:Ib27}
  c_3 &:& H_q(\G_{[1\,2\,3\,0]})\To H_q(\G_{[1\,2\,0]}) \nonumber\\
  c_4 &:& H_q(\G_{[0\,3\,2\,1]})\To H_q(\G_{[2\,1\,0]}) \\
  c_5 &:& H_q(\G_{[0\,2\,1\,3]})\To H_q(\G_{[1\,0\,2]}) \nonumber\\
  c_6 &:& H_q(\G_{[0\,3\,1\,2]})\To H_q(\G_{[2\,0\,1]}) \nonumber
\end{eqnarray}

\begin{lem}\label{l:E^2_3,q} Let $g\ge 3$ and assume that $c_1$ and $c_2$ of Lemma \ref{l:E^2_2,q} are injective and that the four maps in \eqref{e:Ib27} are surjective. Then $E^2_{3,q}(F;i)=0$ for $i=1,2$.
\end{lem}

\begin{proof}The group $E^1_{3,q}$ decomposes into six summands since $\Sibar_3=\Si_3$. By Lemma \ref{l:E^2_2,q}, to show that $E^2_{3,q}=0$ under the above conditions, it suffices to check that $d^1_{4,q}$ maps onto the four components not considered in Lemma \ref{l:E^2_2,q}. More precisely, let
\begin{equation*}
    \tilde{E}_{3,q}^1=  E^1_{3,q}([0\,1\,2])\oplus E^1_{3,q}([2\,1\,0])\oplus E^1_{3,q}([1\,0\,2])\oplus E^1_{3,q}([2\,0\,1]).
\end{equation*}
We must show that the composition
\begin{equation*}
    \bar{d}^1: E^1_{4,q}\stackrel{d^1}{\To} E^1_{3,q}\stackrel{\proj}{\To} \tilde{E}_{3,q}^1
\end{equation*}
is surjective. the argument is quite similar to the proof of Lemma \ref{l:E^2_2,q}, using Prop. \ref{p:8-tal} to cancel out elements. Then the components of $\bar{d}^1$ can be described as follows:
\begin{eqnarray*}
   \bar{d}^1=-\dd_3 &:& E^1_{4,q}([1\,2\,3\,0]) \To E^1_{3,q}([0\,1\,2]), \\
   \bar{d}^1=\dd_0  &:& E^1_{4,q}([0\,3\,2\,1]) \To E^1_{3,q}([2\,1\,0]), \\
   \bar{d}^1=\dd_0  &:& E^1_{4,q}([0\,2\,1\,3]) \To E^1_{3,q}([1\,0\,2]), \\
   \bar{d}^1=(\dd_0,-\dd_3) &:& E^1_{4,q}([0\,3\,1\,2]) \To E^1_{3,q}([2\,0\,1])\oplus E^1_{3,q}([0\,1\,2]).
\end{eqnarray*}
It follows from the surjections in \eqref{e:Ib27} that $\bar{d}^1$ is surjective, and hence that $E^1_{3,q}(F;i)=0$.
\end{proof}

\begin{rem}\label{r:third}Now we can state Harer's third assertion needed to improve our main stability Theorem by
''one degree'' (cf. the Introduction). It is easy to show that
$d^1_{2,2n}[1\, 0]$ is the zero map for all $n$. Then the homology
class $[\check{\k}_1^{\phantom{,}n}]$ of
$\check{\k}_1^{\phantom{,}n}$ with respect to $d^1$ is an element of
$E^2_{2,2n}$. The assertion is
\begin{itemize}
  \item[$(iii)$]$d^2_{2,2n}([\check{\k}_1^{\phantom{,}n}])=x\cdot[\check{\k}_1^{\phantom{,}n}]$
  for some Dehn twist $x$ around a simple closed curve in $F$. Here,
   $\cdot$ denotes the Pontryagin product in group homology.
\end{itemize}
\end{rem}

\subsection{The stability theorem for surfaces with
boundary}\label{ss:soeren} In this section we prove the first of the
two stability theorems listed in the introduction. Our proof is
strongly inspired by the 15 year old manuscript \cite{Harer2}, but
with two changes. We work with integral coefficients, and we avoid
the assertions made in \cite{Harer2} discussed in the introduction.
The theorem we prove is

\begin{thm}[Main Theorem]\label{t:Main}Let $F_{g,r}$ be a surface of
genus $g$ with $r$ boundary components.
\begin{itemize}
\item[$(i)$] Let $r\ge 1$ and let $i=\Si_{0,1}:\G_{g,r}\To
\G_{g,r+1}$. Then
\begin{equation*}
    i_*: H_k(\G_{g,r})\To H_k(\G_{g,r+1})
\end{equation*}
is an isomorphism for $2g\ge 3k$.

\item[$(ii)$]Let $r\ge 2$ and let $j=\Si_{1,-1}:\G_{g,r}\To
\G_{g+1,r-1}$. Then
\begin{equation*}
    j_*: H_k(\G_{g,r})\To H_k(\G_{g+1,r-1})
\end{equation*}
is surjective for $2g\ge 3k-1$, and an isomorphism for $2g\ge 3k+2$.
\end{itemize}
\end{thm}

\begin{proof}The proof is by induction in the homology degree $k$.
For $k=0$ the results are obvious, since $H_0(G,\Z)=\Z$ for any
group $G$. So assume now $k>0$ and that the theorem holds for
homology degrees less than $k$.

\subsubsection*{The case $\Si_{0,1}$}In this case we know from (\ref{e:injektiv}) that $\Si_{0,1}$ is
injective, so to prove that it is an isomorphism it is enough to
show surjectivity.

Assume $2g\ge 3k$ and write
$\G=\G_{g,r+1}$. We use that $\G_{g,r}$ is the stabilizer $\G_\al$ for $\al\in \D_0(F_{g,r+1;2}$ as on Figure \ref{b:arc}, $\G_{g,r}=\G_\al$. Now we use the spectral sequence \eqref{e:E1F} associated
with the action of $\G$ on $C_*(F_{g,r+1};2)$, and we recognize the
map $i_*: H_k(\G_\al)\To H_k(\G)$ as the differential
$d^1:E^1_{1,k}\To E^1_{0,k}$. The spectral sequence converges to zero at $E_{0,k}^n$. So it suffices to show that $E^2_{p,k+1-p}$ is zero for all $p\ge 2$.


We begin by proving $E^2_{2,k-1}=0$ using Lemma \ref{l:E^2_2,q}
$(i)$, noting that $g \ge 2$, since $k\ge 1$. We must verify that
$c_1$ and $c_2$ are surjective, and we will do this
inductively. Prop. \ref{p:grim} (or Example \ref{ex:perm}) and Prop.
\ref{p:genus} calculate the genus and the number of boundary
components of $\G_{\s}$. The figures below show the relevant simplices $\s\in \D_*(F_{g,r+1};2)$ so that the method in Example \ref{ex:perm} can easily be applied. The circles are the boundary components containing $b_0$ and $b_1$.
\setlength{\unitlength}{0.3cm}
\begin{center}
\begin{picture}(44,3)(0,-1.5)
\put(17,0){\circle*{0.3}} \put(13,0){\circle*{0.3}}
\put(12,0){\circle{2}}\put(18,0){\circle{2}}
\put(13,0){\line(1,0){1.7}}\put(17,0){\line(-1,0){1.7}}
\qbezier(13,0)(14,1.5)(15,0)\qbezier(17,0)(16,-1.5)(15,0)
\put(0,-0.2){$\G_{[1\,0]} = \G_{g-1,r+1},$}

\put(42,0){\circle*{0.3}} \put(38,0){\circle*{0.3}}
\put(37,0){\circle{2}}\put(43,0){\circle{2}}
\qbezier(38,0)(40,1.5)(42,0)
\qbezier(38,0)(39,0.3)(39.8,-0.3)\qbezier(42,0)(41,-1.5)(40.2,-0.7)
\qbezier(42,0)(41,0.3)(40,-0.5)\qbezier(38,0)(39,-1.5)(40,-0.5)
\put(25,-0.2){$\G_{[0\,2\,1]} = \G_{g-1,r},$}
\end{picture}

\begin{picture}(44,3)(0,-1.5)
\put(17,0){\circle*{0.3}} \put(13,0){\circle*{0.3}}
\put(12,0){\circle{2}}\put(18,0){\circle{2}}
\qbezier(13,0)(15,1.5)(17,0)\qbezier(13,0)(15,-1.5)(17,0)
\put(0,-0.2){$\G_{[0\,1]} = \G_{g-1,r+1},$}

\put(42,0){\circle*{0.3}} \put(38,0){\circle*{0.3}}
\put(37,0){\circle{2}}\put(43,0){\circle{2}}
\qbezier(38,0)(39,1)(40.25,0.25)\qbezier(40.25,0.25)(41,-0.2)(42,0)
\qbezier(42,0)(41,-1)(39.75,-0.25)\qbezier(39.75,-0.25)(39,0.2)(38,0)
\qbezier(38,0)(39,-1)(39.6,-0.4)\qbezier(40.4,0.4)(41,1)(42,0)
\qbezier(40.1,0.1)(40.1,0.1)(39.9,-0.1) \put(25,-0.2){$\G_{[1\,2\,0]} = \G_{g-2,r+2}.$}
\end{picture}
\end{center}
We see that
\begin{equation*}\begin{array}{rll}
c_1=(\Si_{0,1})_* :& H_{k-1}(\G_{g-1,r})\To
H_{k-1}(\G_{g-1,r+1}), & \text{and}\\
c_2=(\Si_{1,-1})_* :& H_{k-1}(\G_{g-2,r+2})\To
H_{k-1}(\G_{g-1,r+1})
\end{array}
\end{equation*}
are both surjective by induction. So $E^2_{2,k-1}=0$.

We now show that $E^2_{p,q}=0$ for $p+q=k+1$ and $p>2$, i.e. $q\le
k-2$, using Lemma \ref{l:E^2_p,q}, so we must verify
\eqref{e:assume} and (24). By Prop. \ref{p:genus} we have $\G_\al=\G_{g-p+s+1,r+p-2s-1}$, for $\al\in \overline{\D}_{p-1}$ of genus $s$. So for $q\le k-2$, we will show by induction:
\begin{eqnarray}\label{e:induktions}
    H_{q}(\G_{g-p+s+1,r+p-2s-1})\iso H_{q}(\G_{g,r+1}),&\text{for}&p+q\le k+1\\
    H_{q}(\G_{g-p+s+1,r+p-2s-1})\twoheadrightarrow
H_{q}(\G_{g,r+1}),&\text{for}&p+q= k+2.
\end{eqnarray}
The maps in \eqref{e:induktions} and (30) are induced from the
composition
\begin{equation*}\xymatrix{
\G_{g-p+s+1,r+p-2s-1}\ar[rr]^{\quad(\Si_{0,1})^{s+1}}&
&\G_{g-p+s+1,r+p-s} \ar[rr]^{\qquad(\Si_{1,-1})^{p-s-1}}&
&\G_{g,r+1} }.
\end{equation*}
The result follows by induction if
\begin{equation*}
    2(g-p+s+1)\ge 3q \quad \hbox{and} \quad 2(g-p+s+1) \ge
3q+2;
    \quad \textup{for }q\le k-2.
\end{equation*}

Let us prove (\ref{e:induktions}). We know that $2g\ge 3k$, and we
have $p+q\le k+1$. Let $q$ be fixed. Since more arcs
(greater $p$) and smaller genus of $\al$ implies a smaller genus of the cut
surface $F_\al$, it suffices to show the inequality for $p+q=k+1$ and $s=0$.
In this case
\begin{equation*}
2(g-p+1)= 2(g-k-1+q+1)\ge 3k-2k+2q=2q+k \ge 3q+2.
\end{equation*}
where in the last inequality we have used the assumption $q\le k-2$.
The proof of (31) is similar. Now by Lemma \ref{l:E^2_p,q},
$E^2_{p,q}=0$ for all $p+q=k+1$ with $q\le k-2$. This proves that
$d^1_{1,k}=(\Si_{0,1})_*$ is surjective.

\subsubsection*{Surjectivity in the case $\Si_{1,-1}$}Assume $2g\ge
3k-1$, and write $\G=\G_{g+1,r-1}$. Then $\G(F_{g,r})=\G_{\be}$ for $\be\in\D_0(F_{g+1,r-1};1)$ as on Figure \ref{b:arc}. In the spectral sequence \eqref{e:E1F} associated with the action of $\G$ on $C_*(F_{g+1,r-1};1)$, we
recognize the map $(\Si_{1,-1})_*: H_k(\G_{g,r})\To
H_k(\G_{g+1,r-1})$ as the differential $d^1_{1,k}:E^1_{1,k}\To
E^1_{0,k}$. It suffices to show that $E^2_{p,q}=0$ for $p+q=k+1$ and
$q\le k-1$.

We first show that $E^2_{2,k-1}=0$ using Lemma \ref{l:E^2_2,q}. As
before, the figures below show the relevant simplices
in $\D_*(F_{g+1,r-1};1)$, and the oval is the boundary component containing
$b_0$ and $b_1$. \setlength{\unitlength}{0.3cm}
\begin{center}
\begin{picture}(44,3)(0,-1.5)
\put(18,0){\circle*{0.3}} \put(12,0){\circle*{0.3}}
\put(15,0){\oval(6,2.5)}
\put(12,0){\line(1,0){2.6}}\put(18,0){\line(-1,0){2.6}}
\qbezier(12,0)(13.5,1.3)(15,0)\qbezier(18,0)(16.5,-1.3)(15,0)
\put(0,-0.2){$ \G_{[1\,0]} = \G_{g,r-1},$}

\put(43,0){\circle*{0.3}} \put(37,0){\circle*{0.3}}
\put(40,0){\oval(6,2.5)}
\qbezier(37,0)(40,1.5)(43,0)
\qbezier(37,0)(38.5,0.3)(39.8,-0.3)\qbezier(43,0)(41.5,-1.5)(40.2,-0.7)
\qbezier(43,0)(41.5,0.3)(40,-0.5)\qbezier(37,0)(38.5,-1.5)(40,-0.5)
\put(25,-0.2){$\G_{[0\,2\,1]} = \G_{g-1,r},$}
\end{picture}

\begin{picture}(44,3)(0,-1.5)
\put(18,0){\circle*{0.3}} \put(12,0){\circle*{0.3}}
\qbezier(12,0)(15,1.5)(18,0)\qbezier(12,0)(15,-1.5)(18,0) \thinlines
\put(15,0){\oval(6,2.5)}
\put(0,-0.2){$\G_{[0\,1]} = \G_{g-1,r+1},$}

\put(43,0){\circle*{0.3}} \put(37,0){\circle*{0.3}}
\put(40,0){\oval(6,2.5)}
\qbezier(37,0)(38.5,1)(40.25,0.25)\qbezier(40.25,0.25)(41.5,-0.2)(43,0)
\qbezier(43,0)(41.5,-1)(39.75,-0.25)\qbezier(39.75,-0.25)(38.5,0.2)(37,0)
\qbezier(37,0)(38.5,-1)(39.6,-0.4)\qbezier(40.4,0.4)(41.5,1)(43,0)
\qbezier(40.1,0.1)(40.1,0.1)(39.9,-0.1) \put(25,-0.2){$\G_{[1\,2\,0]} = \G_{g-1,r}.$}
\end{picture}
\end{center}
We see that
\begin{equation}\label{e:i1i2}\begin{array}{rll}
c_1=(\Si_{1,-1})_* :& H_{k-1}(\G_{g-1,r})\To
H_{k-1}(\G_{g,r-1}), & \text{and}\\ c_2=(\Si_{0,1})_* :
&H_{k-1}(\G_{g-1,r})\To H_{k-1}(\G_{g-1,r+1})
\end{array}
\end{equation}
are both surjective by induction. So $E^2_{2,k-1}=0$.

Next we show that $E^2_{3,k-2}=0$ using Lemma \ref{l:E^2_3,q}. To
verify the conditions, we calculate as before,
\begin{equation*}
\begin{array}{llll}
    \G_{[0\,1\,2]} &=& \G_{g-2,r+2},\\
    \G_{\s} &=& \G_{g-1,r} &\text{ for $\s\in \Si_3$ the remaining 3 permutations in \eqref{e:Ib27} }\\
    \G_{\s} &=& \G_{g-2,r+1}&\text{ for $\s\in \Si_4$ the remaining 4 permutations in \eqref{e:Ib27}}.
\end{array}
\end{equation*}
We see that
\begin{equation}\label{e:i3ij}\begin{array}{rll}
c_3=(\Si_{0,1})_* : &H_{k-2}(\G_{g-2,r+1})\To
H_{k-2}(\G_{g-2,r+2}), & \text{and}\\
c_j=(\Si_{1,-1})_* :& H_{k-2}(\G_{g-2,r+1})\To
H_{k-2}(\G_{g-1,r}) & \text{for } j=4,5,6.
\end{array}
\end{equation}
Inductively we can verify that these four maps are surjective. The
maps $c_1$ and $c_2$ we calculated in \eqref{e:i1i2}, and we
see by induction that they are injective in homology degree $k-2$.
So by Lemma \ref{l:E^2_3,q}, $E^2_{3,k-2}=0$.

Finally we prove that $E^2_{p,q}=0$ for $p+q=k+1$ and $q\le k-3$
using Lemma \ref{l:E^2_p,q}. This is done as in \textbf{The case $\Si_{0,1}$} so we'll skip the calculations, and just show the final inequality:
\begin{eqnarray*}
2(g-p+1) &=& 2g-2(k+1-q)+2 \quad\ge \quad 3k-1-2k+2q \\
         &=& k+2q -1 \ge q+3+2q -1 \quad = \quad 3q+2.
\end{eqnarray*}
So by Lemma \ref{l:E^2_p,q}, $E^2_{p,q}=0$ for $p+1= k+1$ and $q\le
k-3$. We conclude that $(\Si_{1,-1})_*=d^1_{1,k}$ is surjective.


\subsubsection*{Injectivity in the case $\Si_{1,-1}$} Assume $2g\ge
3k+2$ and let as in the above case $\G=\G_{g+1,r-1}$ and $E^n_{p,q}=
E^n_{p,q}(F_{g+1,r-1};1)$. We will show that
$(\Si_{1,-1})_*=d^1_{1,k}$ is injective. Since $E^n_{1,k}$ converges
to 0, it suffices to show that all differentials with target
$E^n_{1,k}$ are trivial. This holds if we can show that
$E^2_{p,q}=0$ for all $p+q=k+2$ with $q\le k-1$ and that
$d^1_{2,k}:E^1_{2,k}\To E^1_{1,k}$ is trivial.

We first prove that $d^1_{2,k}:E^1_{2,k}\To E^1_{1,k}$ is trivial by
proving that $d^1_{3,k}: E^1_{3,k}\To E^1_{2,k}$ is surjective,
using Lemma \ref{l:E^2_2,q}. We have already calculated $c_1$
and $c_2$, cf. \eqref{e:i1i2}:
\begin{equation*}\begin{array}{rll}
c_1=(\Si_{1,-1})_* :& H_{k}(\G_{g-1,r})\To H_{k}(\G_{g,r-1}), &
\text{and}\\ c_2=(\Si_{0,1})_* : &H_{k}(\G_{g-1,r})\To
H_{k}(\G_{g-1,r+1})
\end{array}
\end{equation*}
In this case we cannot use induction, since the homology degree is
$k$, but we can use the surjectivity result for $\Si_{0,1}$ and
$\Si_{1,-1}$ since we have already proved this. So by Theorem
\ref{t:Main} $(ii)$, $c_1$ and $c_2$ are surjective.

Next we prove that $E^2_{3,k-1}=0$, using Lemma \ref{l:E^2_3,q}. We
have already calculated $c_j$ for $j=1,2,3,4,5,6$ in the proof
of surjectivity of $(\Si_{1,-1})_*$, cf. \eqref{e:i1i2} and
\eqref{e:i3ij}, and in this case we get
\begin{equation*}\begin{array}{rll}
c_1=(\Si_{1,-1})_* :& H_{k-1}(\G_{g-1,r})\To H_{k-1}(\G_{g,r-1}), & \\
c_2=(\Si_{0,1})_* : &H_{k-1}(\G_{g-1,r})\To H_{k-1}(\G_{g-1,r+1}) \\
c_3=(\Si_{0,1})_* : &H_{k-1}(\G_{g-2,r+1})\To H_{k-1}(\G_{g-2,r+2}), &\text{and}\\
c_j=(\Si_{1,-1})_* :& H_{k-1}(\G_{g-2,r+1})\To H_{k-1}(\G_{g-1,r}) & \text{for } j=4,5,6.
\end{array}
\end{equation*}
Inductively we can verify that $c_1$ and $c_2$ are
injective, and that $c_j$ for $j=3,4,5,6$ are surjective. So by
Lemma \ref{l:E^2_3,q}, $E^2_{3,k-1}=0$.

Finally we prove that $E^2_{p,q}=0$ for $p+q=k+1$ and $q\le k-2$
using Lemma \ref{l:E^2_p,q}. As before we skip the calculations, and the final inequality is the same as in \textbf{Surjectivity in the case $\Si_{1,-1}$}.

\end{proof}

\begin{rem}Another possibility for proving the above result is to use another arc complex. Inspired by \cite{Ivanov1} we consider a subcomplex of $C(F;i)$ consisting of all $n$-simplices with a given permutation $\s_n$, $n\ge 0$. Ivanov takes $\s=\id$, which means the cut surfaces $F_{\al}$ have minimal genus. For the inductive assumption, it would be better to have maximal genus, which can be achieved by taking $\s_n=[n\,\,n\!-\!1 \,\cdots \,1 \,\,0]$. Potentially, this could give a better stability range, but it is not known how connected this subcomplex is, which means that the proof above cannot be carried through.
\end{rem}
\subsection{The stability theorem for closed surfaces}\label{S:Closed surface}
In this section we study $l=\Si_{0,-1}: \G_{g,1}\To \G_{g}$, the
homomorphism induced by gluing on a disk to the boundary circle. The
main result is
\begin{thm}\label{t:closed}
\begin{equation*}
    l_*:H_k(\G_{g,1})\To H_k(\G_{g})
\end{equation*}
is surjective for $2g \ge 3k-1$, and an isomorphism for $2g \ge 3k
+ 2$.
\end{thm}
The proof we give is modelled on \cite{Ivanov1}. See also \cite{CM}.

\begin{defn}Let $F$ be a surface, possibly with boundary. The arc
complex $D_*(F)$ has isotopy classes of closed, non-trivial, oriented,
embedded circles as vertices, and $n+1$ distinct vertices ($n\ge 0$)
form an $n$-simplex if they have representatives $(\al_0,\ldots
\al_{n})$ such that:
\begin{itemize}
  \item[$(i)$] $\al_i\cap\al_j=\emptyset$ and $\al_i\cap
  \dd(F)=\emptyset$,
  \item[$(ii)$]  $F\fra (\bigcup_{i=0}^n \al_i)$ is connected.
\end{itemize}
\end{defn}
We note that
\begin{equation}\label{e:cut}
(F_{g,r})_\al\iso F_{g-1,r+2}, \quad \text{for each vertex }\al
\text{ in }D(F_{g,r}).
\end{equation}
Indeed, for a vertex $\al$, $F_\al:= F\fra N(\al)$ has two more
boundary components than $F$, but the same Euler characteristic,
since $F = F\fra N(\al)\cup_{\dd N(\al)}N(\al)$, and
$\chi(N(\al))=0= \chi(\dd N(\al))$. Then \eqref{e:cut} follows from
$\chi(F_{g,r})=2-2g-r$.

We need the following connectivity result, which we state without proof:
\begin{thm}[\cite{Harer1}] The arc complex $D_*(F_{g,r})$ is
$(g-2)$-connected, and $\G_{g,r}$ acts transitively in each
dimension.
\end{thm}
We can now prove the stability theorem for closed surfaces:
\begin{proof}[Proof of Theorem \ref{t:closed}] We use the unaugmented
spectral sequences associated with the action of $\G(F_i)$
on $D_*(F_i)$, where $F_i=F_{g,i}$ for $i=0,1$. They converge to the
homology of $\G(F_i)$ in degrees less than or equal to $g-2$. Since
$\G(F_i)$ acts transitively on the set of $n$-simplices,
\begin{equation}\label{e:lukstab}E_{p,q}^1(F_i)\iso H_q(\G(F_i)_\al,\Z_\al) \Tto H_{p+q}(\G(F_i)), \quad \textrm{for
}i=0,1;
\end{equation}
where $\al$ is $p$-simplex in $D_p(F_{1})$, by identifying $\al$ with its image in $D_p(F_0)$ under the inclusion $l:F_1\To F_0$.

We use Moore's comparison theorem for spectral sequences, cf. \cite{Cartan}:
If $l_*:H_q(\G(F_1)_\al,\Z_\al)\To H_q(\G(F_0)_\al,\Z_\al)$ is an
isomorphism for $p+q\le m$ and surjective for $p+q\le m+1$, then
$l_*:H_k(\G(F_1))\To H_k(\G(F_0))$ is a isomorphism for $k\le m$ and
surjective for $k\le m+1$. To apply this, we will compare $H_q(\G(F_i)_\al,\Z_\al)$ and $H_q(\G((F_i)_\al))$ for a fixed $p$-simplex $\al$.

First we need to analyse $\G(F_i)_\al$ for $i=0,1$, and to ease the notation we call the surface $F$ and write $\G=\G(F)$. Unlike for $C_*(F;i)$, the stabilizer $\G_\al$ is not $\G(F_\al)$. For $\g\in\G_\al$,
\begin{itemize}
  \item[$(i)$] $\g$ need not stabilize $\al$
  pointwise and can thus permute the circles of $\al$;
  \item[$(ii)$] $\g$ can change the orientation of any circle in $\al$;
  \item[$(iii)$] $\g$ can rotate each circle $\al$ in $\al$.
\end{itemize}

In order to take care of $(i)$ and $(ii)$, consider the exact
sequence,
\begin{equation}\label{e:exact} 1\To \widetilde{\G_\al}\To
\G_\al\To (\Z/2)^{p+1}\ltimes \Si_{p+1}\To 1.
\end{equation}
Here $\widetilde{\G_\al}\del \G_\al$ consists of the mapping classes in $\G_\al$ fixing each vertex of $\al$ and its orientation.
We now compare $\widetilde{\G_\al}$ and ${\G}(F_\al)$,
\begin{equation}\label{e:exact2} 0\To \Z^{p+1}\To {\G}(F_\al)\To
\widetilde{\G_\al}\To 1.
\end{equation}
We must explain the map $\Z^{p+1}\To {\G}(F_\al)$. Let $\al=(\al_0,\ldots,\al_p)$, then the cut surface $F_\al$ has two boundary components, $\al_i^+$ and $\al_i^-$, for each circle $\al_i$. Then the standard generator $e_j=(0,\ldots,0,1,0,\ldots,0)\in\Z^{p+1}$, $j=0,\ldots,p$, maps to the mapping class making a right Dehn twist on $\al_j^+$ and a left Dehn twist on $\al_j^-$, and identity everywhere else. This is extended to a group homomorphism, i.e. $-e_j$ makes a left Dehn twist on $\al_j^+$ and a right Dehn twist on $\al_j^-$.

Let us see that \eqref{e:exact2} is exact. The hard part is injectivity of
$\Z^{p+1}\To {\G}(F_\al)$, so we only show this. Assume $m\ne n\in
\Z^{p+1}$, and say $m_0\ne n_0$. For $p\ge 1$, the surface $F_\al$
has at least four boundary components. Two of them come from cutting
up along the circle $\al_0$, call one of these $S$. If $p=0$, then
$\al=\al_0$, and $F_\al$ has genus $g-1\ge 2$ by \eqref{e:cut},
since $2g \ge 3k + 3\ge 6$. In both cases, there is a non-trivial
loop $\g$ in $F_\al$ starting on $S$ which does not commute with
the Dehn twist $f$ around $S$ in $\pi_1(F_\al)$. Since $F_\al $
has boundary, $\pi_1(F_\al)$ is a free group, so the subgroup
$\indre{\g}{f}$ is also free. The action of $m\in\Z^{p+1}$ on $\g$
is $f^{m_0}\g f^{-m_0}$, and since $f$ and $\g$ does not commute,
$f^{m_0}\g f^{-m_0} \ne f^{n_0}\g f^{-n_0}$ when $n_0\ne m_0$.
%

Consider $l_*: \G((F_1)_\al)\To \G((F_0)_\al)$. Both
surfaces $(F_i)_\al$ have non-empty boundary, so we can use Main Theorem \ref{t:Main}. We must relate $l_*$ to the maps $\Si_{0,1}$ and $\Si_{1,-1}$, so let $\hat{F}$ denote a surface such that
$\Si_{0,1}(\hat{F})=(F_1)_\al$. Then $\hat{F}$ has one less
boundary components than $(F_1)_\al$, so $\hat{F}$ and $(F_0)_\al$ are isomorphic. This gives the diagram:
\begin{equation*}
    \xymatrix{H_*(\G(\hat{F}))\ar[rr]^{\iso}\ar[dr]_{(\Si_{0,1})_*}& & H_*(\G((F_0)_\al))\\
    & H_*(\G((F_1)_\al))\ar[ur]_{l_*}&
    }
\end{equation*}
We see that $l_*$ is always surjective. By Theorem \ref{t:Main}, $(\Si_{0,1})_*:H_{s}({\G}(\hat F))\To H_{s}({\G}((F_1)_\al))$ is an isomorphism for $3s \le 2(g-p-1)$, so the same holds for $l_*$.

The Lynden-Serre spectral sequence of \eqref{e:exact2} for $F$ is
\begin{equation}\label{e:E2a}
    \bar{E}_{s,t}^2(F)\iso H_s(\widetilde{\G_\al},H_t(\Z^{p+1})) \Tto
H_{s+t}({\G}(F_\al)).
\end{equation}
We showed above that $l_*: H_{s+t}({\G}((F_1)_\al))\To H_{s+t}({\G}((F_0)_\al))$ is an isomorphism for $3(s+t)\le 2(g-p-1)$ and surjective always.
Note that $\Z^{p+1}$ lies in the center of $\G(F_\al)$, since the Dehn twists can take place as close to the boundary
of $F_\al$ as desired. By the Künneth formula, we have an isomorphism
\begin{equation*}
    \bar{E}_{s,t}^2(F) \iso \bar{E}_{s,0}^2(F) \tensor \bar{E}_{0,t}^2(F) = H_s(\widetilde{\G_{\al}})\tensor H_t(\Z^{p+1})
\end{equation*}
Now since $l_*: H_{s+t}({\G}((F_1)_\al))\To H_{s+t}({\G}((F_0)_\al))$ is an isomorphism for $3(s+t)\le 2(g-p-1)$ and always surjective, it follows by an easy inductive argument that $l_*: H_s(\widetilde{\G(F_0)_{\al}})\To H_s(\widetilde{\G(F_1)_{\al}})$ is an isomorphism for $3s\le 2(g-p-1)$ and surjective for $3s\le 2(g-p-1)+3$.

The Lynden-Serre spectral sequence of \eqref{e:exact} is
\begin{equation}\label{e:E2}
\tilde{E}_{r,s}^2(F)\iso H_r\left((\Z/2)^{p+1}\ltimes
\Si_{p+1};H_s(\widetilde{\G_\al};\Z_\al)\right) \Tto
H_{r+s}(\G_\al;\Z_\al).
\end{equation}
Since $\widetilde{\G_\al}$ preserves the orientation of the
simplices, we can drop the local coordinates to obtain
\begin{equation*}
\tilde E_{r,s}^2(F)\iso H_r\left((\Z/2)^{p+1}\times
\Si_{p+1},H_s(\widetilde{\G_\al})\tensor \Z_\al\right).
\end{equation*}
It follows from the above that $l_*:\tilde E_{r,s}^2(F_1)\To \tilde E_{r,s}^2(F_0)$ is an isomorphism for $3s\le 2(g-p-1)$ and surjective for $3s\le 2(g-p-1)+3$.
Then by Moore's comparison theorem,
\begin{equation*}
    l_*: H_{q}(\G(F_1)_\al;\Z_\al) \To H_{q}(\G(F_0)_\al;\Z_\al)
\end{equation*}
is an isomorphism for $3q\le 2(g-p-1)$ and surjective for $3q\le 2(g-p-1)+3$. Then in particular, it is an isomorphism for $3(p+q)\le 2g-2$ and surjective for $3(p+q)\le 2g-2+3$. Now a final application of Moore's comparison theorem on the spectral sequence in \eqref{e:lukstab} gives the desired result, as explained in the beginning of the proof.
\end{proof}

\newpage

\section{Stability with twisted coefficients}

\subsection{The category of marked surfaces}
\begin{defn}\label{d:MS}The category of marked surfaces $\mathfrak{C}$ is defined as follows: The objects are triples $F,x_0,(\dd_1F, \dd_2F, \ldots,\dd_rF)$, where $F$ is a compact connected orientable surface with non-empty boundary $\dd F= \dd_1 F \cup \cdots \dd_r F$, with a numbering $(\dd_1F,\ldots,\dd_rF)$ of the boundary components of $F$, and
$x_0\in\dd_1 F$ is a marked point.

A morphism $(\psi, \s)$ between marked surfaces $(F,x_0)$ and
$(G,y_0)$ is an ambient isotopy class of an embedding $\psi: F\To
G$, where each boundary component of $F$ is either mapped to the
inside of $G$ or to a boundary component of $G$. If $\psi(x_0) \in
\dd G$ then $\psi(x_0)=y_0$, else there is a embedded arc $\s$ in
$G$ connecting $x_0$ and $y_0$.
\end{defn}
The objects of $\mathfrak{C}$ is can be grouped
\begin{equation*}\textrm{Ob}\,\mathfrak{C}=\coprod_{g,r}\textrm{Ob}\,\mathfrak{C}_{g,r},
\end{equation*}
where $\mathfrak{C}_{g,r}$ consists of the surfaces with genus $g$
and $r$ boundary components.

\begin{defn} The morphisms $\Si_{1,0}$, $\Si_{0,1}$ in $\mathfrak{C}$  are the embeddings $\Si_{i,j}:F\To \Si_{i,j}F$ given by gluing onto $\dd_1F$ a torus with 2 disks cut out, or a pair of pants, respectively, as on Figure \ref{f:sigma}. The embedded arc $\s$ is also shown here. The boundary components of $\Si_{0,1}F$ are numbered such that the new boundary component from the pair of pants is $\dd_{r+1}(\Si_{0,1}F)$.

The morphism $\Si_{1,-1}$ in the subcategory of $\coprod_{r\ge 2}\textrm{Ob}\,\mathfrak{C}_{g,r}$ is the embedding given by gluing a pair of pants onto $\dd_1(F)$ and $\dd_2(F)$, as on Figure \ref{f:sigma}. The numbering is that $\dd_j(\Si_{1,-1}F)=\dd_{j-1}F$ for $j>1$.
\end{defn}
\setlength{\unitlength}{0.5cm}
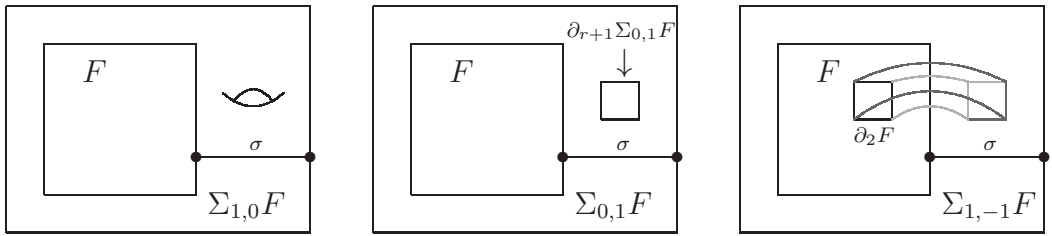
\begin{figure}[!h]
\begin{center}
\begin{picture}(8,5)(0,1)
\multiput(0,0)(0,6){2}{\line(1,0){8}}
\multiput(0,0)(8,0){2}{\line(0,1){6}} \put(2,4){$F$}
\multiput(1,1)(0,4){2}{\line(1,0){4}}
\multiput(1,1)(4,0){2}{\line(0,1){4}} \put(5,2){\circle*{.3}}
\put(6.4,2.2){$\scriptstyle\s$}
\put(5,2){\line(1,0){3}}\put(8,2){\circle*{.3}}\put(5.3,.5){$\Si_{1,0}F$}
\qbezier(5.7,3.7)(6.5,3)(7.3,3.7)
\qbezier(6,3.5)(6.5,4.1)(7,3.5)
\end{picture}\qquad
\begin{picture}(8,5)(0,1)
\multiput(0,0)(0,6){2}{\line(1,0){8}}
\multiput(0,0)(8,0){2}{\line(0,1){6}} \put(2,4){$F$}
\multiput(1,1)(0,4){2}{\line(1,0){4}}
\multiput(1,1)(4,0){2}{\line(0,1){4}} \put(5,2){\circle*{.3}}
\put(5,2){\line(1,0){3}}\put(8,2){\circle*{.3}}\put(5.3,.5){$\Si_{0,1}F$}
\put(6.4,2.2){$\scriptstyle\s$}
\multiput(6,3)(0,1){2}{\line(1,0){1}}
\multiput(6,3)(1,0){2}{\line(0,1){1}}
\put(5.1,5.2){$\scriptstyle\dd_{r+1}\Si_{0,1}F$}
\put(6.4,4.3){$\downarrow$}
\end{picture}\qquad
\begin{picture}(8,5)(0,1)
\multiput(0,0)(0,6){2}{\line(1,0){8}}
\multiput(0,0)(8,0){2}{\line(0,1){6}} \put(2,4){$F$}
\multiput(1,1)(0,4){2}{\line(1,0){4}}
\multiput(1,1)(4,0){2}{\line(0,1){4}} \put(5,2){\circle*{.3}}
\put(5,2){\line(1,0){3}}\put(8,2){\circle*{.3}}\put(5.3,.5){$\Si_{1,-1}F$}
\put(6.4,2.2){$\scriptstyle\s$}
\put(3,2.4){$\scriptstyle\dd_2F$}
\multiput(3,3)(0,1){2}{\line(1,0){1}}
\multiput(3,3)(1,0){2}{\line(0,1){1}}
\color[rgb]{0.40,0.40,0.40}
\put(6,3){\line(1,0){1}}
\put(7,3){\line(0,1){1}}
\color[rgb]{0.70,0.70,0.70}
\put(6,4){\line(1,0){1}}
\put(6,3){\line(0,1){1}}
\color[rgb]{0.40,0.40,0.40}
\qbezier(3,4)(5,5)(7,4)
\qbezier(3,3)(5,4.5)(7,3)
\color[rgb]{0.70,0.70,0.70}
\qbezier(4,4)(5,4.3)(6,4)
\qbezier(4,3)(5,3.7)(6,3)
\end{picture}
\end{center}
\caption{The morphisms $\Si_{1,0}$, $\Si_{0,1}F$, and $\Si_{1,-1}F$.}\label{f:sigma}
\end{figure}
In the figure, the black rectangles are boundary components of $F$ or $\Si_{i,j}F$, and the outer boundary component is always $\dd_1F$ with the marked point indicated. On the figure of $\Si_{1,-1}F$ the grey ''tube'' is a cylinder glued onto $\dd_2F$.

Now we will see how $\Si_{i,j}$ can be made into functors. First we define the subcategory $\mathfrak{C}(2)$ of $\mathfrak{C}$ to be the category with objects $\coprod_{r\ge 2}\textrm{Ob}\,\mathfrak{C}_{g,r}$ and whose morphisms $\f:F\To S$ must restrict to an orientation-preserving diffeomorphism $\f:\dd_2 F\To\dd_2 S$. Note that $\Si_{1,0}$ and $\Si_{0,1}$ are morphisms in this category.

$\Si_{1,0}$ and $\Si_{0,1}$ are functors from $\mathfrak{C}$ to itself, and $\Si_{1,-1}$ is a functor from $\mathfrak{C}(2)$ to $\mathfrak{C}$ in the following way: Given a morphism $\f:F\To S$ we must specify the morphism $\Si_{i,j}(\f)$, and this is done on the following diagram (drawn in the case of $\Si_{1,0}$). Here, the grey line shows how $\Si_{1,0}$ is embedded in $\Si_{0,1}S$ by $\Si_{1,0}(\f)$. Notice how the arc $\s$ determines the embedding.

\setlength{\unitlength}{0.3cm}
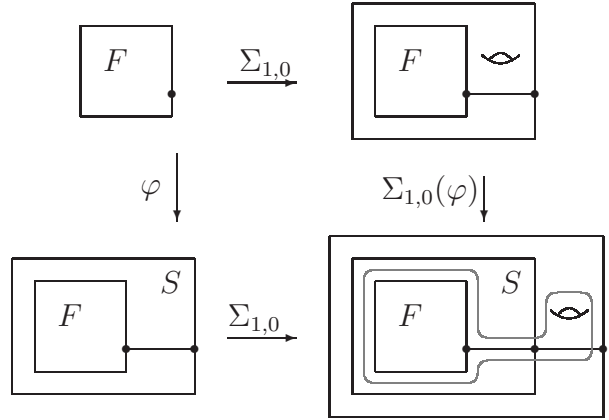
\begin{figure}[!h]
\begin{center}
\begin{picture}(8,5)(0,1)
\put(2,3){$F$}
\multiput(1,1)(0,4){2}{\line(1,0){4}}
\multiput(1,1)(4,0){2}{\line(0,1){4}} \put(5,2){\circle*{.3}}
\end{picture}
\begin{picture}(4,5)(0,1)
\put(-1,2.5){\vector(1,0){3}}
\put(-.5,3){$\Si_{1,0}$}
\end{picture}
\begin{picture}(10,5)(0,1)
\multiput(0,0)(0,6){2}{\line(1,0){8}}
\multiput(0,0)(8,0){2}{\line(0,1){6}} \put(2,3){$F$}
\multiput(1,1)(0,4){2}{\line(1,0){4}}
\multiput(1,1)(4,0){2}{\line(0,1){4}} \put(5,2){\circle*{.3}}
\put(5,2){\line(1,0){3}}\put(8,2){\circle*{.3}}
\qbezier(5.7,3.7)(6.5,3)(7.3,3.7)
\qbezier(6,3.5)(6.5,4.1)(7,3.5)
\end{picture}
\end{center}
\begin{center}
\begin{picture}(13,3)(0,1)
\put(3,4){\vector(0,-1){3}}
\put(1.4,2.2){$\f$}
\end{picture}
\begin{picture}(5,3)(0,1)
\put(3,3){\vector(0,-1){2}}
\put(-1.5,2){$\Si_{1,0}(\f)$}
\end{picture}
\end{center}
\begin{center}
\begin{picture}(8,5)(0,1)
\multiput(0,0)(0,6){2}{\line(1,0){8}}
\multiput(0,0)(8,0){2}{\line(0,1){6}} \put(2,3){$F$}
\multiput(1,1)(0,4){2}{\line(1,0){4}}
\multiput(1,1)(4,0){2}{\line(0,1){4}} \put(5,2){\circle*{.3}}
\put(5,2){\line(1,0){3}}\put(8,2){\circle*{.3}}
\put(6.5,4.5){$S$}
\end{picture}
\begin{picture}(6,5)(0,1)
\put(1,2.5){\vector(1,0){3}}
\put(1,3){$\Si_{1,0}$}
\end{picture}
\begin{picture}(12,5)(0,1)
\multiput(0,0)(0,6){2}{\line(1,0){8}}
\multiput(0,0)(8,0){2}{\line(0,1){6}} \put(2,3){$F$}
\multiput(1,1)(0,4){2}{\line(1,0){4}}
\multiput(1,1)(4,0){2}{\line(0,1){4}} \put(5,2){\circle*{.3}}
\put(5,2){\line(1,0){3}}\put(8,2){\circle*{.3}}
\put(6.5,4.5){$S$}
\multiput(-1,-1)(0,8){2}{\line(1,0){12}}
\multiput(-1,-1)(12,0){2}{\line(0,1){8}}
\qbezier(8.7,3.7)(9.5,3)(10.3,3.7)
\qbezier(9,3.5)(9.5,4.1)(10,3.5)
\put(8,2){\line(1,0){3}}\put(11,2){\circle*{.3}}
\color[rgb]{0.50,0.50,0.50}
\qbezier(.5,1)(.5,1)(.5,5)
\qbezier(.5,1)(.5,.5)(1,.5)
\qbezier(5,.5)(1,.5)(1,.5)
\qbezier(5,.5)(5.5,.5)(5.5,1)
\qbezier(5.5,1)(5.5,1.5)(6,1.5)
\qbezier(6,1.5)(6,1.5)(10,1.5)
\qbezier(10,1.5)(10.5,1.5)(10.5,2)
\qbezier(10.5,2)(10.5,2)(10.5,4)
\qbezier(10.5,4)(10.5,4.5)(10,4.5)
\qbezier(.5,5)(.5,5.5)(1,5.5)
\qbezier(1,5.5)(1,5.5)(5,5.5)
\qbezier(5,5.5)(5.5,5.5)(5.5,5)
\qbezier(5.5,5)(5.5,5)(5.5,3)
\qbezier(5.5,3)(5.5,2.5)(6,2.5)
\qbezier(6,2.5)(6,2.5)(8,2.5)
\qbezier(8,2.5)(8.5,2.5)(8.5,3)
\qbezier(8.5,3)(8.5,3)(8.5,4)
\qbezier(8.5,4)(8.5,4.5)(9,4.5)
\qbezier(9,4.5)(9,4.5)(10,4.5)
\end{picture}
\end{center}
\caption{The functor $\Si_{1,0}$.}\label{f:functor}
\end{figure}
Similar diagrams can be drawn for $\Si_{0,1}$ and $\Si_{1,-1}$. In the latter case $\Si_{1,-1}(\f)$ exists because when $\f\in\mathfrak{C}(2)$, $\f:F\To S$ has not done anything to $\dd_2(F)$, so that $\Si_{1,-1}F$ can be embedded in $\Si_{1,-1}S$ just as on Figure \ref{f:functor}.

\subsection{Coefficient systems}
We now define the coefficient systems we are interested in. We say that an abelian group $G$ is \emph{without infinite division} if the following holds for all $g\in G$: If $n\mid g$ for all $n\in\Z$, then $g=0$. By $n\mid g$ we mean $g=nh$ for some $h\in G$. Note that finitely generated abelian groups are without infinite division.

\begin{defn}A coefficient system is a functor from $\mathfrak{C}$ to
$\textrm{Ab}_{\textrm{wid}}$, the category of abelian groups without infinite division.
\end{defn}
We say that a constant coefficient system has degree 0 and make the general
\begin{defn}\label{d:coef}\cite{Ivanov1} A coefficient system $V$ has degree $\le k$ if the map $V(F){\To}V(\Si_{i,j}F)$ is split injective for $(i,j)\in\set{(1,0),(0,1), (1,-1)}$,  and the cokernel $\Delta_{i,j}V$ is a coefficient system of degree $\le k-1$
for $(i,j)\in\set{(1,0),(0,1)}$. The degree of $V$ is the smallest such $k$.
\end{defn}
\begin{ex}

\begin{itemize}
  \item[$(i)$] $V(F)=H_1(F,\dd F)$ is a coefficient system of degree $1$.
  \item[$(ii)$] $V^*_k(F)=H_k(\Map((F/\dd F),X)$. This is the coefficient system used in \cite{CM}. It has degree $\le \round{\frac{k}{d}}$ if $X$ is $d$-connected, which will be proved in Theorem \ref{t:stabil1}.
\end{itemize}
\end{ex}
We write $\Si_{i,j}V$ for the functor $F\rightsquigarrow V(\Si_{i,j}F)$, where $(i,j)\in\set{(1,0),(0,1)}$.
\begin{lem}[Ivanov]Let $V$ be a coefficient system of degree $\le k$. Then $\Si_{1,0}V$ and $\Si_{0,1}V$ are coefficient systems of degree $\le k$.
\end{lem}
\begin{proof}See \cite{Ivanov1} for $\Si_{1,0}V$. The case $\Si_{0,1}V$ can be handled similarly.
\end{proof}

\subsection{The inductive assumption}
Below I will use the following notational conventions: $F$ denotes a surface in $\mathfrak{C}$, and unless otherwise specified, $g$ is the genus of $F$. $\Si_{l,m}$ refers to any of $\Si_{1,0}$, $\Si_{0,1}$, $\Si_{1,-1}$.

\begin{defn}\label{d:Phi}
Given a morphism $\psi:F\To S$, $\Phi$ will denote a finite composition of $\Si_{0,1}$ and $\Si_{1,-1}$ such that $\Phi(\psi)$ is defined, i.e. makes the following diagram comutative
\begin{equation*}
    \xymatrix{
F\ar[r]^{\Phi}\ar[d]^{\psi}& \Phi(F)\ar@{-->}[d]^{\Phi(\psi)} \\
S\ar[r]^{\Phi}& \Phi(S)
}
\end{equation*}
By a finite composition we mean $\Phi=\Si_{i_1,j_1}\circ \cdots\circ \Si_{i_s,j_s}$ for some $s\ge0$, where $(i_k,j_k)\in\set{(0,1), (1,-1)}$ for each $k=1,\ldots,s$. We say that such a $\Phi$ is \emph{compatible} with $\psi:F\To S$.
\end{defn}

To prove our main stability result for twisted coefficients, we will study certain relative homology groups:
\begin{defn}\label{d:Rel}Let $\psi:F\To S$ be a morphism of surfaces, and let $\Phi$ be compatible. Let $V$ be a coefficient system. Then we define
\begin{equation*}
\Rel_n^{V,\Phi}(S,F) = H_n(\G(S),\G(F);V(\Phi(S)),V(\Phi(F))).
\end{equation*}
If $\Phi=\id$, we write $\Rel_n^{V}(G,F)$ for $\Rel_n^{V,\id}(G,F)$.
\end{defn}

\begin{thm}[Ivanov, Madsen-Cohen]\label{t:IvanovCoMa}For sufficiently large $g$:
\begin{itemize}
  \item[$(i)$]$\Rel_q^{V}(\Si_{1,0}F,F)=0.$
  \item[$(ii)$]$\Rel_q^{V}(\Si_{0,1}F,F)=0.$
 \item[$(iii)$]$\Rel_q^{V}(\Si_{1,-1}F,F)=0.$
\end{itemize}
\end{thm}
\begin{proof}
For $(i)$, see \cite{Ivanov1}. For $(ii)$, see \cite{CM}. Their proof only requires that the groups $V(\cdot)$ are without infinite division.

To prove $(iii)$, we use the following long exact sequence,
\begin{eqnarray*}
  H_q(F,V(F))&\To&  H_q(\Si_{1,-1}F,V(\Si_{1,-1} F))\To\Rel_q^{V}(\Si_{1,-1}F,F)\To \\
  H_{q-1}(F,V(F))&\To&  H_{q-1}(\Si_{1,-1}F,V(\Si_{1,-1}F))
\end{eqnarray*}
Thus to see that $\Rel_q^{V}(\Si_{1,-1}F,F)=0$ all we have to do is to see that the first map is surjective and that the last map is injective. Both of these maps are $\Si_{1,-1}$, so they fit into the following diagram, for $k \in \set{q, q-1}$:
\begin{equation*}\xymatrix{H_k(F,V(F))\ar[r]^{\Si_{1,-1}} & H_k(F,V(F))\\
  H_k(S,V)\ar[u]^{\Si_{0,1}}\ar[ur]^{\Si_{1,0}} & }
\end{equation*}
where $S$ is a surface with $\Si_{0,1}S=F$. Now by $(i)$ and $(ii)$, if $g$ is sufficiently large, both the diagonal and the vertical map is an isomorphism, so $\Si_{1,-1}$ is also an isomorphism.
\end{proof}



Define $\e_{l,m}$ by
\begin{equation*}
    \e_{l,m}=\left\{
               \begin{array}{ll}
                 1, & \hbox{if $(l,m)=(1,-1)$;} \\
                 0, & \hbox{if $(l,m)=(1,0)$ or $(0,1)$.}
               \end{array}
             \right.
\end{equation*}

\begin{inda}The inductive assumption $I_{k,n}$ is the following: For any coefficient system $W$ of degree
$k_W$, any surface $F$ of genus $g$, and any $\Phi$ compatible with $\Si_{l,m}:F\To\Si_{l,m}F$, we have
\begin{equation*}
    \Rel^{W,\Phi}_q(\Si_{l,m}F, F) = 0\quad\text{for}\quad 2g\ge 3q+k_W-\e_{l,m},
\end{equation*}
if either $k_W<k$, or $k_W=k$ and $q<n$.
\end{inda}
In the rest of this section I am going to assume $I_{k,n}$. Note
that $I_{k,m}$ for all $m\in \N$ is equivalent to $I_{k+1,0}$. Thus
the goal is to prove $I_{k,n+1}$. Let $V$ be a given coefficient
system of degree $k$.

\begin{lem}[Ivanov]\label{l:Ivanov}Let $F$ be a surface of genus $g$. If $2g\ge 3q+k-1-\e_{l,m}$ then for $(i,j)\in \set{(1,0),(0,1)}$
\begin{equation*}
    \Rel_q^{V,\Phi}(\Si_{l,m}F,F)\To \Rel_q^{V,\Si_{i,j}\Phi}(\Si_{l,m}F,F)
\end{equation*}
is surjective.
\end{lem}
\begin{proof}Since $\Rel_q^{V,\Si_{i,j}\Phi}(\Si_{l,m}F,F)=\Rel_q^{\Si_{i,j}V,\Phi}(\Si_{l,m}F,F)$ we have the following long exact sequence :
\begin{equation*}
    \Rel_q^{V,\Phi}(\Si_{l,m}F,F)\To \Rel_q^{V,\Si_{i,j}\Phi}(\Si_{l,m}F,F)\To \Rel_q^{\Delta_{i,j}V,\Phi}(\Si_{l,m}F,F)
\end{equation*}
Since $\Delta_{i,j}V$ is a coefficient system of degree $k-1$, the
assumption $I_{k,n}$ implies that $\Rel_q^{\Delta_{i,j}V,\Phi}(\Si_{l,m}
F,F)=0$, and the result follows.
\end{proof}

\begin{thm}\label{t:vis_indu}
Assume that $h$ satisfies $2h\ge 3n+k-1-\e_{l,m}$ and that the maps below are injective for all surfaces $F$ of genus $g\ge h$ and $\Phi$ compatible with $\Si_{l,m}:F\To \Si_{l,m}F$,
\begin{eqnarray*}
  \Rel_n^{V,\Phi\Si_{1,-1}}(\Si_{l,m} F,F) &\To&
   \Rel^{V,\Phi}_n(\Si_{l,m}\Si_{1,-1}F,\Si_{1,-1}F), \\
   \Rel_n^{\Si_{0,1}V}(\Si_{l,m} F,F) &\To&
   \Rel^{V}_n(\Si_{l,m}\Si_{0,1}F,\Si_{0,1}F).
\end{eqnarray*}
Then for any compatible $\Phi$, $\Rel_n^{V,\Phi}(\Si_{l,m} F,F)=0$ for $g\ge h$.
\end{thm}
\begin{proof}Assume $2g\ge 3n+k-1-\e_{l,m}$. Write $\Phi=\Si_{i_1,j_1}\circ \cdots\circ \Si_{i_s,j_s}$, where $(i_k,j_k)\in\set{(1,-1),(0,1)}$. Observe that we can write $\Phi= \Phi'\circ (\Si_{1,-1})^d$ for some $d$, where $\Phi'=\Si_{\la_1,\mu_1}\circ \cdots\circ \Si_{\la_t,\mu_t}$ with $(\la_k,\mu_k)\in\set{(1,0),(0,1)}$. Then by the first assumption in the theorem, we get by induction in $d$:
\begin{equation*}
   \Rel_n^{V,\Phi}(\Si_{l,m} F,F) \To   \Rel^{V,\Phi'}_n(\Si_{l,m}(\Si_{1,-1})^dF,(\Si_{1,-1})^dF)\\
\end{equation*}
is injective. Thus it suffices to show $\Rel_n^{V,\Phi'}(\Si_{l,m}(\Si_{1,-1})^dF,(\Si_{1,-1})^dF)=0$. Since $\genus((\Si_{1,-1})^dF)\ge g\ge h$, it is certainly enough to show $\Rel_n^{V,\Phi'}(\Si_{l,m} F,F)=0$, where $\Phi'$ is a finite composition of $\Si_{1,0}$ and $\Si_{0,1}$. By Lemma
\ref{l:Ivanov}, we get inductively that
\begin{equation*}
    \Rel_n^V(\Si_{l,m}F,F)\To \Rel_n^{V,\Phi'}(\Si_{l,m}F,F)
\end{equation*}
is surjective, so it suffices to show that $\Rel_n^V(\Si_{l,m}F,F)=0$.
Now by the second assumption in the Theorem, we know
\begin{equation*}
   \Rel_n^{\Si_{0,1}V}(\Si_{l,m} F,F) \To
   \Rel^{V}_n(\Si_{l,m}\Si_{0,1}F,\Si_{0,1}F)
\end{equation*}
is injective. Since $V$ is a coefficient system of degree $k$, $V(F)\To V(\Si_{0,1}F)$ and $V(F)\To V(\Si_{1,-1}F)$ are split injective, so the composition,
\begin{eqnarray*}
   \Rel_n^{V}(\Si_{l,m} F,F) \To \Rel_n^{\Si_{0,1}V}(\Si_{l,m} F,F) \To
   \Rel^{V}_n(\Si_{l,m}\Si_{0,1}F,\Si_{0,1}F)\\
 \To \Rel^{\Si_{1,-1}V}_n(\Si_{l,m}\Si_{0,1}F,\Si_{0,1}F) \To \Rel^{V}_n(\Si_{l,m}\Si_{1,0}F,\Si_{1,0}F)
\end{eqnarray*}
is injective, where the second and the last maps are the maps in the assumption and thus injective. Iterating this, we get an injective map
\begin{equation*}
   \Rel_n^{V}(\Si_{l,m} F,F) \To
   \Rel^V_n(\Si_{l,m}(\Si_{1,0})^dF,(\Si_{1,0})^dF)
\end{equation*}
for any $d\in \N$. But $\genus((\Si_{1,0})^dF)=g+d$, so
by Theorem \ref{t:IvanovCoMa},
$\Rel_n^{V}(\Si_{l,m} F,F)$ injects into zero. This proves
$\Rel_n^{V,\Phi}(\Si_{l,m} F,F)=0$.
\end{proof}

\subsection{The main theorem for twisted coefficients}
In the proof of stability for relative homology groups, we will use the relative version of the spectral sequence, cf. Theorem \ref{s:ss}, $E^1_{p,q}=E^1_{p,q}(\Si_{i,j}F;2-i)$ associated with the action of $\G(\Si_{i,j}F)$ on the arc complex
$C_*(\Si_{i,j}F;2-i)$ and the action of $\G(\Si_{l,m}\Si_{i,j}F)$ on
the arc complex $C_*(\Si_{l,m}\Si_{i,j}F;2-i)$. Let $b_0,b_1$ be the points in the definition of $C_*(\Si_{i,j}F;2-i)$; and $\tilde{b}_0, \tilde{b}_1$ be the corresponding points for $C_*(\Si_{l,m}\Si_{i,j}F;2-i)$. We demand that $b_0$, $\tilde{b}_0$ lie in the 1st boundary component, but is different from the marked point. To define the spectral sequence, $\Si_{l,m}$ must induce a map
\begin{equation}\label{e:udvidarc}
\Si_{l,m}:C_*(\Si_{i,j}F;2-i) \To C_*(\Si_{l,m}\Si_{i,j}F;2-i),
\end{equation}
which we now define: If $i=0$, $b_0$ and $b_1$ lie in different boundary components, and the map is given on $\al\in\D_k(\Si_{i,j}F)$ by a simple path $\g$ from $\tilde{b}_0\in \Si_{l,m}\Si_{i,j}F$ to $b_0\in \Si_{i,j}F$ inside $\Si_{l,m}\Si_{i,j}F\fra\Si_{i,j}F$. Then the arcs of $\al$ are extended by parallel copies of $\g$ that all start in $\tilde{b}_0$. Note that in this case $\tilde{b}_1=b_1$, so no extension is necessary here. If $i=1$, $b_0$ and $b_1$ lie on the same boundary component, and we choose disjoint paths for them to the new marked boundary component, and extend as for $i=0$.

Now the spectral sequence (typically) has $E^1$ page:
\begin{eqnarray}\label{e:E^1_p,q}
    E^1_{p,q}&=&\bigoplus_{\s\in\Sibar_p}E^1_{p,q}(\s)\nonumber\\
    E^1_{p,q}(\s) &=&H_q(\G(\Si_{i,j}\Si_{l,m}F)_{\Si_{l,m}T(\s)},\G(\Si_{i,j}F)_{T(\s)};\nonumber\\
&& \quad\:\: V(\Phi\Si_{i,j}\Si_{l,m}\Si_{s,t}(F)),V(\Phi\Si_{i,j}\Si_{s,t}(F)))\nonumber\\
&=&\Rel_q^{V,\Phi_\s}((\Si_{i,j}\Si_{l,m}F)_{\Si_{l,m}T(\s)},(\Si_{i,j}F)_{T(\s)})
\end{eqnarray}
Here, $\Phi_\s: (\Si_{i,j}F)_{T(\s)}\into \Si_{i,j}F$ is the inclusion, which is a finite composition of $\Si_{0,1}$ and $\Si_{1,-1}$. Furthermore, $\G_\s$ denotes the stabilizer of the $(p-1)$-simplex $\s$ in $\G$. The direct sum is over the orbits of $(p-1)$-simplices $\s$ in $C_*(\Si_{i,j}F;2-i)$, whose images under $\Si_{l,m}$ are also $(p-1)$-simplices in  $C_*(\Si_{l,m}\Si_{i,j}F;2-i)$. In most cases, $\Si_{l,m}$ induces a bijection on the representatives of orbits of $(p-1)$-simplices. Also recall that the set of orbits are in $1-1$
correspondence with a subset $\overline{\Si}_p$ of the permutation group $\Si_p$. Lemma \ref{l:perm} characterizes $\overline{\Si}_p$. As a general remark, note that if a permutation is represented in $C_*(F;2-i)$, then it is also represented in
$C_*(\Si_{l,m}F;2-i)$, since $\text{genus}(\Si_{l,m}F) \ge
\text{genus}(F)$. So we will only check the condition for
$C_*(F,2-i)$.

In certain cases we will either not have $\Si_{l,m}$ inducing
bijection on the representatives of orbits of $(p-1)$-simplices, or
they will not include the permutation used in the standard proof.
All such cases will be found in Lemma \ref{l:Si_p} below and taken care of in the \emph{Inductive start} section at the end of the proof.

The first differential, $d^1_{p,q}:E^1_{p,q}\To E^1_{p-1,q}$, is
described in section \ref{sc:diff}. The diagrams
\begin{equation*}
    \xymatrix{
    \D_{p}(F;i)\ar[r]^{\dd_j}\ar[d] & \D_{p}(F;i)\ar[d] &\\
    \Sibar_{p+1}\ar[r]^{\dd_j} & \Sibar_p & j=0,\ldots,p
    }
\end{equation*}
commute, where $\dd_j$ omits entry $j$ as in Def. \ref{d:arc} and the vertical arrows divide out the $\G$ action and compose with $P$. Thus for each $\s\in\Sibar_{p+1}$, there is $g_j\in\G$ such that
\begin{equation}\label{e:Ib18}
    g_j\cdot\dd_j T(\s) = T(\dd_j\s),
\end{equation}
and conjugation by $g_j$ induces an injection $c_{g_j}:\G_{T(\s)}\into \G_{T(\dd_j\s)}$. The induced map on homology is denoted $\dd_j$ again, i.e.
\begin{eqnarray}\label{e:Ib19}
\dd_j:H_q(\G(\Si_{i,j}\Si_{l,m}F)_{\Si_{l,m}T(\s)}, \G(\Si_{i,j}F)_{T(\s)};\textbf{V})\hookrightarrow \nonumber\\
H_q(\G(\Si_{i,j}\Si_{l,m}F)_{\Si_{l,m}\dd_jT(\s)}, \G(\Si_{i,j}F)_{\dd_jT(\s)};\textbf{V})    \stackrel{(c_{g_j})_*}{\To}  \\ H_q(\G(\Si_{i,j}\Si_{l,m}F)_{\Si_{l,m}T\dd_j(\s)}, \G(\Si_{i,j}F)_{T\dd_j(\s)};\textbf{V})\nonumber
\end{eqnarray}
Note that $(c_{g_j})_*$ does not depend on the choice of $g_j$ in \eqref{e:Ib18}: Another choice $g_j'$ gives $c_{g_j'}= c_{g_j'g_j^{-1}}c_{g_j}$, and $g_j'g_j^{-1}\in\G_{T(\dd_j\s)}$ so $c_{g_j'g_j^{-1}}$ induces the identity on the homology. Then
\begin{equation}\label{e:Ib20}
    d^1=\sum_{j=0}^{p-1}(-1)^j\dd_j.
\end{equation}
\begin{lem}\label{l:Si_p}Let $n\ge 1$. The subset $\Sibar_p\del \Si_p$, which is in $1-1$ correspondence with a set of representatives of the orbits of $\D_{p-1}(\Si_{i,j}F;2-i)$, has the following properties:
\begin{description}
  \item[Surjectivity of $\Si_{0,1}$:]Assume $2g\ge
3n+k-2-\e_{l,m}$. Then \\$\Sibar_{p}=\Si_p$ for $2\le p\le n+1$ and for $p=n+2= 3$, unless:
\begin{itemize}
  \item $(l,m)\neq (1,-1),\quad n=1,\quad g=1,\quad k=0,1$, \quad or
  \item $(l,m)= (1,-1),\quad n=1,\quad g=0,\quad k=0$, \quad or
  \item $(l,m)= (1,-1),\quad n=1,\quad g=1,\quad k=0,1,2$.
\end{itemize}
  \item[Surjectivity of $\Si_{1,-1}$:]Assume $2g\ge 3n+k-3-\e_{l,m}$. Then \\$\Sibar_{p}=\Si_p$ for $2\le p\le n+1$, and $\s\in \Sibar_p$ if $S(\s)\ge 1$ for $p=n+2\le 4$, unless:
\begin{itemize}
  \item $(l,m)\neq (1,-1),\quad n=1,\quad g=0,\quad k=0$,\quad or
  \item $(l,m)= (1,-1),\quad n=1,\quad g=0,\quad k=0,1$,\quad or
  \item $(l,m)= (1,-1),\quad n=2,\quad g=1,\quad k=0$.
\end{itemize}
  \item[Injectivity of $\Si_{1,-1}$:]Assume $2g\ge3n+k-\e_{l,m}$. Then\\ $\Sibar_{p}=\Si_p$ for $2\le p\le n+2$, and $\s\in \Sibar_p$ if $S(\s)\ge 1$ for $p=n+3=4$, unless:
\begin{itemize}
  \item $(l,m)= (1,-1),\quad n=1,\quad g=1,\quad k=0$.
\end{itemize}
\end{description}
 \end{lem}
\begin{proof}We only prove the first of the three cases, as the other two are completely analogous. So assume $2g\ge 3n+k-2-\e_{l,m}$, and let $\s\in\Si_p$ be a given permutation of genus $s$. Let $2\le p\le n+1$. By Lemma \ref{l:perm}, $\s\in\Sibar_p$ if and only if $s \ge p-1-g$. This inequality is certainly satisfied if $p-1-g\le 0$. The hardest case is $p=n+1$, so we must show $n-g\le 0$. By assumption,
\begin{equation*}
    2(n-g)\le 2n-(3n+k-2+\e_{l,m})=-n-k+2+\e_{l,m}\stackrel{?}{\le} 0,
\end{equation*}
For $n\ge 3$ this holds. If $n=2$, the assumption $2g\ge
3n+k-2-\e_{l,m}$ forces $g\ge 2$, so $n-g\le
0$. For $n=1$ and $(l,m)\neq (1,-1)$, we have $\e_{l,m}=0$, so $g\ge 1$, which means $n-g\le 0$. Last for $n=1$ and $(l,m)= (1,-1)$, we have $\e_{l,m}=1$, so we get one exception, $g=k=0$.

Now let $p=n+2=3$, so $n=1$. The requirement in Lemma \ref{l:perm} is $p-1-g\le 0$, i.e. $g\ge 2$. By assumption $2g\ge
3n+k-2-\e_{l,m}$, so if $g=1$, we have $k-\e_{l,m}-1\le 0$. Now for $(l,m)\ne (1,-1)$, the only exceptions are $k=0,1$, and for $(l,m)=(1,-1)$, the only exceptions are $k=0,1,2$. If $g=0$, we have $k-\e_{l,m}+1\le 0$, so the only exception is $(l,m)=(1,-1)$ and $k=0$.
This finishes the proof.
\end{proof}

\begin{prop}\label{p:8-tal2}Let $\al$ denote a simplex either in $\D_1(F;1)$ with $P(\al)=[1\,0]$, or in $\D_2(F;2)$ with $P(\al)=[2\,1\,0]$. Let $g$ be the genus of $F_\al$, and let $\Phi$ be compatible with $\Si_{l,m}:F\To \Si_{l,m}F$. Then if $2g\ge 3q+k_W-1-\e_{l,m}$, the maps $\dd_0=\dd_{1}$ are equal as maps from
\begin{equation*}Rel_n^{V,\Phi_\al}((\Si_{l,m}F)_{\Si_{l,m}\al},F_\al).
\end{equation*}
\end{prop}
\begin{proof}Write $\s=P(\al)$. First note that $\dd_0$ and $\dd_{1}$ have the same target, since $\dd_0(\s)=\dd_{1}(\s)=:\t$ by assumption. We can assume $T(\s)=\al$ and $T(\t)=\dd_0\al$. Then we can choose the element $g=g_{1}$ from \eqref{e:Ib18}, which must satisfy $g\cdot\dd_{1} \al = \dd_0\al$, to be as in Prop. \ref{p:8-tal}. Then $g$ commutes with the stabilizers $\G(\Si_{l,m}F)_{\al_0\cup\al_1}$, $\G(F)_{\al_0\cup\al_1}$ and thus also with $\G(\Si_{l,m}F)_{\al}$ and $\G(F)_{\al}$.

We now extend the arcs of $\al$ to arcs in $\Phi F$ as follows: If $\al\in\D_1(F;1)$ we use $\eqref{e:udvidarc}$ to obtain $\tilde\al=\Phi(\al)\in \D_1(\Phi F;1)$. If $\al\in\D_2(F;2)$, we extend, if possible, the 1-simplex $\al_0\cup\al_1$ to a 1-simplex $\tilde\al\in \D_1(\Phi F;1)$, i.e. the extended arcs start and end on the same boundary component in $\Phi F$. If this is not possible, we extend $\al$ to $\tilde\al\in\D_2(\Phi F;2)$. These extensions must satisfy the same requirements as $\eqref{e:udvidarc}$ does. Then we make the same extensions for $\be:=\Si_{l,m}\al$ to $\tilde\be$ in $\Phi\Si_{l,m}F$.
Now the conjugation $(c_g)_*$ acts as the identity on
\begin{equation*}
    H_n(\G(\Si_{l,m}F)_{\be},\G(F)_{\al};         V((\Phi\Si_{l,m}F)_{\tilde{\be}}),V((\Phi F)_{\tilde{\al}}))
\end{equation*}

If we are in the case $\tilde\al\D_1(\Phi F;1)$, then the inclusion map on the coefficients,
\begin{eqnarray}\label{e:i_*}
 i_*&: &  H_n(\G(\Si_{l,m}F)_{\be},\G(F)_{\al};         V((\Phi\Si_{l,m}F)_{\tilde{\be}}),V((\Phi F)_{\tilde{\al}}))\To \\
&&H_n(\G(\Si_{l,m}F)_{\be},\G(F)_{\al}; V(\Phi\Si_{l,m}F),V(\Phi F))=
\Rel_{n}^{V,\Phi_\al}((\Si_{l,m}F)_{\Si_{l,m}\al},F_\al)\nonumber
\end{eqnarray}
equals $\Si_{1,0}$ on the coefficient systems, and by Lemma \ref{l:Ivanov} it is surjective since $2g\ge 3n+k-1-\e_{l,m}$ by assumption. Now as $i_*$ is surjective and $(c_g)_* \circ i_* =i_*$ we see that $(c_g)_*$ is the identity on $\Rel_{n}^{V,\Phi_\al}(\Si_{l,m}F_\al,F_\al)$, and thus $\dd_{1}=(c_g)_*\dd_0 =\dd_0$.
For $\tilde\al\in\D_2(\Phi F;2)$ we do the same, except that we use $\al$ instead of only $\al_0\cup \al_1$. In this case $i_*$ in \eqref{e:i_*} is going to be $\Si_{1,0}\Si_{0,1}$ on the coefficient systems, which again by Lemma \ref{l:Ivanov} is surjective.
\end{proof}

By Theorem \ref{t:vis_indu}, to prove $I_{k,n+1}$ it
is enough to prove:

\begin{thm}\label{t:maintwist}
The map induced by $\Si_{i,j}$,
\begin{equation*}
   \Rel_n^{V,\Phi\Si_{i,j}}(\Si_{l,m} F,F) \To
   \Rel^{V,\Phi}_n(\Si_{i,j}\Si_{l,m}F,\Si_{i,j}F)
\end{equation*}
satisfies: \begin{itemize}
\item[$(i)$]For $\Si_{i,j}=\Si_{0,1}$, it is surjective for $2g\ge
3n+k-2-\e_{l,m}$, and if $\Phi=\id$ it is an isomorphism for $2g\ge3n+k-1-\e_{l,m}$. For $k=0$ it is always injective.
\item[$(ii)$]For $\Si_{i,j}=\Si_{1,-1}$,
it is surjective for $2g\ge 3n+k-3-\e_{l,m}$, and an isomorphism for
$2g\ge3n+k-\e_{l,m}$.
\end{itemize}
\end{thm}

\begin{proof}We prove the theorem by induction in the homology degree $n$. Assume $n\ge 1$. The induction start $n=0$ will be handled separately below, along with all exceptional cases from Lemma \ref{l:Si_p}. This means that in the main proof, any permutation is represented by an arc simplex (in some special cases only if its genus is $\ge 1$).

\subsubsection*{Surjectivity for $\Si_{0,1}$:}Assume $2g\ge 3n+k-2-\e_{l,m}$. We use the spectral sequence $E^1_{p,q}=E^1_{p,q}(\Si_{0,1}F;2)$, and claim that $E^1_{p,q}=0$ for $p+q=n+1$ with $p\ge 3$. Note that
$\G(\Si_{0,1}F)_\s=\G(\Si_{0,1}F_\s)$, and
$\text{genus}(\Si_{0,1}F_\s)=g-p+1+S(\s)\ge g-p+1$. We will use
the assumption $I_{k,n}$, and must show $2(g-p+1)\ge 3q+k-\e_{l,m}$ for $p\ge
3$. These inequalities follows from the one for $p=3$, which is
$2(g-2)\ge 3(n-2)+k-\e_{l,m}$, and this holds by assumption.

Now all we need is to show that $E^2_{2,n-1}=0$. We consider
\begin{equation*}
    E^1_{2,n-1}= E^1_{2,n-1}([0\,1])\oplus E^1_{2,n-1}([1\,0])
\end{equation*}
We wish to show that $d_1:E^1_{3,n-1}\To  E^1_{2,n-1}$ is surjective and thus $E^1_{2,n-1}=0$. We look at $E^1_{3,n-1}(\t)$ indexed by the
permutation $\t=[2\,1\,0]$. We will show that $d^1$ restricted to $E^1_{3,n-1}(\t)$ surjects onto $E^1_{2,n-1}([1\,0])$ without hitting
$E^1_{2,n-1}([0\,1])$.  Since $S(\t)=1$, $\Si_{0,1}F_\t$ is
$F_{g-1,r}$, and thus by Proposition \ref{p:8-tal2}, $\dd_0=\dd_1$. We then see
\begin{equation*}
    d_1=\dd_0-\dd_1+\dd_2= \dd_2
\end{equation*}
and $\dd_2:E^1_{3,n-1}(\t)\To E^1_{2,n-1}[1\,0]$ equals $\Si_{0,1}$ and so is surjective by induction, since $2(g-1)\ge 3(n-1)+k-2-\e_{m,l}$.
All that remains is to hit $E^1_{2,n-1}([0\,1])$ surjectively,
regardless of $E^1_{2,n-1}([1\,0])$. Consider the following
component of $d^1$:
\begin{equation*}
    \dd_0: E^1_{3,n-1}([2\,0\,1])\To
    E^1_{2,n-1}([0\,1]).
\end{equation*}
This is the map induced by $\Si_{1,-1}$. By induction this map is
surjective, since $2(g-2)\ge 3(n-1)+k-3-\e_{l,m}$ by assumption. This proves
that $E^2_{2,n-1}=0$.

\subsubsection*{Injectivity for $\Si_{0,1}$:}Assume $2g\ge 3n+k-1-\e_{l,m}$. For this proof we take another approach. Consider the following composite map,
\begin{eqnarray}\label{e:inj}
  \Rel_q^V(\Si_{l,m}F,F)&\To& \Rel_q^{\Si_{0,1}V}(\Si_{l,m}F,F) \nonumber \\
  &\stackrel{\Si_{0,1}}{\To}& \Rel_q^{V}(\Si_{l,m}\Si_{0,1}F,
\Si_{0,1}F)\nonumber\\
&\stackrel{p_*}{\To}& \Rel_q^{V}(\Si_{0,-1}\Si_{l,m}\Si_{0,1}F,
\Si_{0,-1}\Si_{0,1}F)\nonumber \\
&=&\Rel_q^{V}(\Si_{l,m}F, F)
\end{eqnarray}
Here $p:F_{g,r}\To F_{g,r-1}$ is the map that glues a disk onto a the unmarked boundary circle created by $\Si_{0,1}$. Since the composite map \eqref{e:inj} is induced by
gluing on a cylinder to the marked boundary circle of
$\Si_{l,m}F$ and $F$, it is an isomorphism. Now by Lemma \ref{l:Ivanov}, since $2g\ge 3n+k-1-\e_{l,m}$, the first map is surjective, so $\Si_{0,1}$ is forced to be injective. Note with constant coefficients ($k=0$), the first map is the identity, so here $\Si_{0,1}$ is always injective.

\subsubsection*{Surjectivity for $\Si_{1,-1}$:} Assume $2g\ge
3n+k-3-\e_{l,m}$. We use the spectral sequence $E^1_{p,q}=E^1_{p,q}(\Si_{1,-1}F;1)$. We show $E^1_{p,q}=0$ if $p+q=n+1$ and $p\ge 4$, using assumption $I_{k,n}$. We know $\G(\Si_{1,-1}F)_\s=\G((\Si_{1,-1}F)_\s)$, and
$\text{genus}((\Si_{1,-1}F)_\s)=g-p+1+S(\s)\ge g-p+1$. So we must show $2(g-p+1)\ge 3q+k-\e_{l,m}$ for all $p+q=n+1$, $p\ge 4$. This follows if we show it for
$p=4$, which is easy:
\begin{equation*}
    2(g-3) =2g-6 \ge 3n +k-3-\e_{m,l}-6 = 3(n-3)+k-\e_{m,l}.
\end{equation*}
To show that the map $d_1:E^1_{1,n}\To E^1_{1,n}$ is surjective, we
thus only need to show that $E^2_{2,n-1}=0$ and $E^2_{3,n-2}=0$. Consider $E^1_{2,n-1}$:
\begin{equation*}
    E^1_{2,n-1}=E^1_{2,n-1}([0\,1])\oplus E^1_{2,n-1}([1\,0]).
\end{equation*}
For $\s=[1\,0]$, since $S(\s)=1$, we have $\text{genus}((\Si_{1,-1}F)_\s)=g-p+1+S(\s)=g$. Thus by $I_{k,n}$, $E^1_{2,n-1}([1\,0])=0$, since $2g\ge 3n+k-1-\e_{m,l}=3(n-1)+k+2-\e_{l,m}$. Now
consider the summand in $E^1_{3,n-1}$ indexed by $\t=[2\,0\,1]$
which has genus 1. Then $(\Si_{1,-1}F)_\t=F_{g-1,r}$, so
$d_1$ on this summand is exactly the map induced by $\Si_{0,1}$
(since $d_1$ has 3 terms, only one of which hit
$E^1_{2,n-1}([0\,1])$). To show this is surjective onto
$E^1_{2,n-1}$, we use induction, and must check that $2(g-1)\ge
3(n-1)+k-\e_{l,m}$, which follows by assumption. So $d^1$ is surjective onto $E^1_{2,n-1}$, which implies that $E^2_{2,n-1}=0$.

Consider $E^1_{3,n-2}$. As above, by $I_{k,n}$, all summands are
zero, except for the one indexed by $\id=[0\,1\,2]$. Consider
$E^1_{4,n-2}(\t')$ indexed by $\t'=[3\,0\,1\,2]$, which has genus
$1$. Restricting $d^1$ to this summand, only one term hits
$E^1_{3,n-2}([0\,1\,2])$. As above, one checks that this
restriction of $d^1$ is exactly the map induced by $\Si_{0,1}$, so
by induction it is surjective.

\subsubsection*{Injectivity for $\Si_{1,-1}$:} Assume $2g\ge 3n+k+2-\e_{l,m}$.
We use the same spectral sequence as in the surjectivity of
$\Si_{1,-1}$. We claim $E^1_{p,q}=0$ if $p+q=n+2$ and $p\ge 4$. Again, $\G(\Si_{1,-1}F)_\s=\G(\Si_{1,-1}F_\s)$, and
$\text{genus}(\Si_{1,-1}F_\s)=g-p+1+S(\s)\ge g-p+1$. So we must show
$2(g-p+1)\ge 3q+k+2-\e_{m,l}$ for all $p+q=n+2$, $p\ge 4$, and this follows from $2g\ge 3n+k+2-\e_{m,l}$, as above.

To show that the map $d_1:E^1_{1,n}\To E^1_{0,n}$ is injective, we
thus only need to show that $E^2_{3,n-1}=0$ and $d^1:E^1_{2,n}\To
E^1_{1,n}$ is the zero-map. That $E^2_{3,n-1}=0$ is proved precisely as for $E^2_{3,n-2}$ in surjectivity for $\Si_{1,-1}$, so we omit it.
To show $d^1:E^1_{2,n}\To E^1_{1,n}$ is the zero-map, note that $E^1_{2,n}$ has two summands, $E^1_{2,n}([0\,1])$ and $E^1_{2,n}([1\,0])$. We get that $d^1$ is zero on $E^1_{2,n}([1\,0])$, since $d_1=\dd_0-\dd_1=0$ by Proposition \ref{p:8-tal2}. Next we consider $d^1: E^1_{3,n}\To E^1_{2,n}$. If we can show this is surjective onto $E^1_{2,n}([0\,1])$, we are done. Again we use the summand $E^1_{3,n}(\t)$, where $\t=[2\,0\,1]$. The restricted differential $d^1: E^1_{3,n}(\t)\To E^1_{2,n}([0\,1])$ is exactly the map induced by $\Si_{0,1}$, so we can show it is surjective, since we have already proved the Theorem for $\Si_{0,1}$. The relevant inequality is $2(g-1)\ge 3n+k-\e_{l,m}$, which holds by assumption. So $d^1:E^1_{2,n}\To E^1_{1,n}$ is the zero-map, and we have shown that $d_1:E^1_{1,n}\To E^1_{1,n}$ is injective.

\subsubsection*{Induction start and special cases:}
Here we handle the the inductive
start $n=0$, along with the cases missing in the general argument
above, namely the exceptions from Lemma \ref{l:Si_p}.

\paragraph{The induction start $n=0$.}For $n=0$ and $k=0$, we always get $\Rel_0^{V,\Phi}(\Si_{l,m} F,F)=0$ since $H_0(F,V(F))\To H_0(\Si_{l,m} F,V(\Si_{l,m} F))$ is an isomorphism when the coefficients are constant. So the theorem holds in this case. Now let $n=0$ and let $k$ be arbitrary. By considering the spectral
sequence, see Figure \ref{b:spectral0},
we see that $\Si_{i,j}$ is automatically surjective, since the
spectral sequence always converges to zero at $(0,0)$.

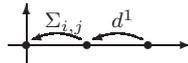
\begin{figure}[!hbt]
\begin{center}
\setlength{\unitlength}{0.8cm}
\begin{picture}(3,0.5)(0.5,0.5)
\put(-0.3,0){\vector(1,0){3}}\put(0,-0.3){\vector(0,1){1}}
\put(0,0){\circle*{0.15}}\put(1,0){\circle*{0.15}}
\put(2,0){\circle*{0.15}} \qbezier(0.9,0.1)(0.5,0.3)(0.1,0.1)
\put(.27,.1575){\vector(-3,-1){0.2}}\qbezier(1.9,0.1)(1.5,0.3)(1.1,0.1)
\put(1.27,.1575){\vector(-3,-1){0.2}} \put(0.3,0.3){$\lil
\Si_{i,j}$} \put(1.4,0.3){$\lil d^1$}
\end{picture}
\end{center}\caption{The spectral sequence for $n=0$.}\label{b:spectral0}\end{figure}

For the sake of the case $n=1$, note that the surjectivity
argument for $\Si_{0,1}$ when $n=0$ also works for any $k$ when using
the spectral sequence for \emph{absolute} homology for the action of
$\G(F_{0,r+1})$ on $C_*(F_{0,r+1};2)$.

For $\Si_{0,1}$, the injectivity argument used above holds for all $n$.
So we must show that $\Si_{1,-1}$ is injective. For $g\ge 1$, the argument from above works, since there are arc simplices
representing all the permutations used above. The problem is thus
$g=0$, which means $k=0,1$, but we will also show the result for $k=2$ since we will need in the case $n=1$ below.

As the complex we use, $C_*(F_{1,r-1};1)$, is
connected, the spectral sequence converges to $0$ for $p+q\le 1$, so
we can apply that spectral sequence. We must show that
$d^1=d^1_{2,0}$ in Figure \ref{b:spectral0} is the zero map. We
consider $(l,m)\in\{(1,0),(1,-1)\}$ and $(l,m)=(0,1)$ separately.
For $\Si_{0,1}$, $E^1_{2,0}=E^1_{2,0}([1\,0])$, since the
permutation $[0\,1]$ has genus $0$ and is by Lemma \ref{l:perm}
neither represented in $C_*(F_{1,r-1};1)$ nor
$C_*(\Si_{0,1}F_{1,r-1};1)$. Now the argument used to show injectivity
of $\Si_{1,-1}$ in general works here, too.

For $\Si_{1,0}$ or $\Si_{1,-1}$, $E^1_{2,0}=E^1_{2,0}([1\,0])\oplus
\tilde{E}^1_{2,0}([0\,1])$ where $\tilde{E}^1_{2,0}([0\,1])$ is the
\emph{absolute} homology group,
\begin{equation*}
    \tilde{E}^1_{2,0}([0\,1])= H_0(\G(\Si_{l,m}F_{1,r-1})_{T([0\,1])}; V(\Si_{l,m}F_{1,r-1})),
\end{equation*}
since $[0\,1]$ is represented in $C_*(\Si_{1,-1}F_{1,r-1};1)$ and $C_*(\Si_{1,0}F_{1,r-1};1)$, but not in $C_*(F_{1,r-1};1)$, see Theorem
\ref{s:spectralinj}. For $E^1_{2,0}([1\,0])$, the general argument
for injectivity of $\Si_{1,-1}$ shows that $d^1_{2,0}([1\,0])$ is
zero. That $d^1:\tilde{E}^1_{2,0}([0\,1])$ is the zero map
will follow if we show that $\tilde{E}^1_{3,0}$ hits
$\tilde{E}^1_{2,0}([0\,1])$ surjectively. But the $d^1$-component
$\tilde{E}^1_{3,0}([2\,0\,1])\To \tilde{E}^1_{2,0}([0\,1])$ is just
$\Si_{0,1}$ in the absolute case for $n=0$, $g=0$ and $k\le 2$. This
$d^1$-component is surjective onto $\tilde{E}^1_{2,0}([0\,1])$, by the remark on surjectivity for $n=0$.

\paragraph{Surjectivity when $n=1$.}
Now let $n=1$ and $k\le 2$. Consider the relative spectral sequence, as depicted in Figure \ref{b:spectral1}. If we show that the map $d^2_{2,0}:E^2_{2,0}\To E^2_{0,1}$ is zero,
we have shown surjectivity. We will show that $E^1_{2,0}=0$. Recall
by Theorem \ref{s:spectralinj}, $E^1_{2,0}=E^1_{2,0}([0\,1])\oplus E^1_{2,0}([1\,0])$, where
\begin{equation}\label{e:EEEEE}
   E^1_{2,0}(\s)=\left\{
                   \begin{array}{ll}
                     \Rel_0^{V,\Phi_\s}(\G(F_{g+i+l,r+j+m})_{\Si_{m,l}\s},\G(F_{g+i,r+j})_\s), & \hbox{if $\s\in\Sibar_1^{l,m}\cap\Sibar_1$;} \\
                     H_0(\G(F_{g+i+l,r+j+m})_{\Si_{m,l}\s}; V(\G(F_{g+i+l,r+j+m}))), & \hbox{if $\s\in \Sibar_1^{l,m}\fra\Sibar_1$;} \\
                     0, & \hbox{if $\s\notin \Sibar_1^{l,m}$.}
                   \end{array}
                 \right.
\end{equation}
and $\Sibar_1$, $\Sibar^{l,m}_1$ are the subsets of $\Si_1$ in $1-1$ correspondence with the orbits of $\D_1(\Si_{i,j}F;2-i)$ and $\D_1(\Si_{l,m}\Si_{i,j}F;2-i)$, respectively.

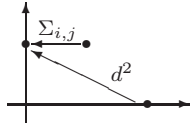
\begin{figure}[!hbt]
\begin{center}
\setlength{\unitlength}{0.8cm}
\begin{picture}(2.5,1.5)(0.5,0.5)
\put(-0.3,0){\vector(1,0){3}}\put(0,-0.3){\vector(0,1){2}}
\put(0,1){\circle*{0.15}}\put(1,1){\circle*{0.15}}
\put(2,0){\circle*{0.15}} \put(0.9,1){\vector(-1,0){0.8}}
\put(1.8,0.05){\vector(-2,1){1.7}} \put(0.2,1.15){$\lil \Si_{i,j}$}
\put(1.4,0.4){$\lil d^2$}
\end{picture}
\end{center}\caption{The spectral sequence for $n=1$.}\label{b:spectral1}\end{figure}

\paragraph{Surjectivity of $\Si_{1,-1}$ when $n=1$.}Assume $(l,m)=(0,1)$, $g=0$ and $k=0$. Then by Lemma \ref{l:perm} only $[1\,0]$ is represented
as an arc simplex, and by \eqref{e:EEEEE} above, $E^1_{2,0}$ is a relative homology group of degree 0 with constant coefficients, so $E^1_{2,0}=0$.

The remaining exceptions are $(l,m)\neq (0,1)$, $g=0$ and $k\le1$. By Lemma \ref{l:perm}, $[1\,0]$ is represented
as an arc simplex in both $F_{1+l,r+m}$ and $F_{1,r-1}$, so $E^1_{2,0}([1\,0])=0$ by Theorem \ref{t:vis_indu}. Now $[0\,1]$ is only represented in $F_{1+l,r+m}$, so by \eqref{e:EEEEE}, $E^1_{2,0}([1\,0])$ is an absolute homology group. To kill it, consider $E^1_{3,0}([2\,0\,1])$,. which is also an absolute homology group. The restricted differential and $d^{1}:E^1_{3,0}([2\,0\,1]) \To E^1_{2,0}([0\,1])$ equals $\Si_{0,1}$, so it is surjective by the case $n=0$, which as remarked also holds for absolute homology group.

\paragraph{Surjectivity of $\Si_{0,1}$ when $n=1$.}First assume $g=1$. The possible permutations $[0\,1]$ and $[1\,0]$ are by Lemma \ref{l:perm} represented as $1$-simplices in both arc complexes. Thus $E^1_{2,0}$ is a direct sum of two relative homology groups in degree 0 with coefficients of degree $k\le 2$. Then by the \emph{Induction start} $n=0$, $\Si_{0,1}$ and $\Si_{1,-1}$ are injective for $g\ge 0$, so by Theorem \ref{t:vis_indu}, $E^1_{2,0}=0$.

For $(m,l)=(1,-1)$, we have the special case $g=k=0$. We will show $H_1(\G_{1,r},\G_{0,r+1})=0$, by showing $\Si_{1,-1}:H_1(\G_{0,r+1};\Z)\To H_1(\G_{1,r};\Z)$ is surjective, and thus that any map into $H_1(\G_{1,r},\G_{0,r+1})$ is surjective. We use \cite{Harer3}, Lemma 1.1 and 1.2, which give sets of generators for $H_1(\G_{0,r+1};\Z)$ and $H_1(\G_{1,r};\Z)$, as follows. Let $\t_i$ be the Dehn twist around each boundary component $\dd_i F_{1,r}$, for $i=1,\ldots,r$, and let $x$ be the Dehn twist on any non-separating simple closed curve $\g$ in $F_{1,r}$. Then $H_1(\G_{1,r};\Z)$ is generated by $\t_2,\ldots,\t_{r},x$. We remark that Harer states this for $\Q$-coefficients, but in $H_1$ his proof also holds for $\Z$-coefficients. We can choose the curve $\g$ as the image of $\dd_2 F_{0,r+1}$ under $\Si_{1,-1}$. Similarly in $\G_{0,r+1}$, we have Dehn twists $\t_i'$ around each boundary component $\dd_i F_{0,r+1}$, and these are among the generators for $H_1(\G_{0,r+1};\Z)$. Then $\Si_{1,-1}$ maps $\t_{i+1}'\mapsto \t_i$ for $i=2,\ldots,r$ by construction of $\Si_{1,-1}$, and $\t_2'\mapsto x$ by the choice of $\g$. So $\Si_{1,-1}:H_1(\G_{0,r+1};\Z)\To H_1(\G_{1,r};\Z)$ is surjective.

\paragraph{Injectivity of $\Si_{1,-1}$ when $n=1$.} The only exception is $(l,m)=(1,-1)$, $g=1$ and $k=0$. For this we will use a different argument, drawing on the stability Theorem for $\Z$-coefficients. Consider the following exact sequence:
\begin{eqnarray}\label{e:k=0}
  &  &H_1(\G_{1,r};V)\twoheadrightarrow H_1(\G_{2,r-1};V)\To\Rel^V_1(\G_{2,r-1},\G_{1,r})\nonumber\\
&\To& H_0(\G_{1,r};V) \stackrel{\iso}{\To} H_0(\G_{2,r-1};V)
\end{eqnarray}

Since $k=0$ we have constant coefficients, so we can use Theorem
\ref{t:Main}. Since $2\cdot 1\ge 3\cdot 1-1$, the first map in \eqref{e:k=0} is surjective, and the last map is an isomorphism. Thus $\Rel^V_1(\G_{2,r-1},\G_{1,r})=0$ and any map from it is thus injective.
This finishes the special cases when $n=1$.
\paragraph{Surjectivity of $\Si_{1,-1}$ when $n=2$.}Again we have only one exception, namely $(l,m)=(1,-1)$, $g=1$ and $k=0$. It suffices to show $E^2_{2,1}=0$ and $E^2_{3,0}=0$. For $E^2_{2,1}$ the argument in \emph{Surjectivity of $\Si_{1,-1}$} works since all the permutations used there are in $\Sibar_2$. So consider $E^2_{3,0}$. Here for all permutations $\t$ except $[0\,1\,2]$ we have $\t\in \Sibar_3\cap\Si_3^{l,m}$ (for this notation, see \eqref{e:EEEEE}. Thus for these $\t$ we know that $E^1_{3,0}(\t)=0$, since it is a relative homology group in degree 0 with constant coefficients. But $[0\,1\,2]\in\Sibar_3^{1,-1}\fra \Sibar_3$, so $E^1_{3,0}([0\,1\,2])$ is an absolute homology group. However, this group is hit surjectively by $E^1_{4,0}[3\,0\,1\,2]$, since the restricted differential equals $\Si_{0,1}$ (see the remark for $n=0$). Thus $E^2_{3,0}=0$, as desired.
\end{proof}

\begin{rem}
As a Corollary to this result, we can be a bit more specific about what happens when stability with $\Z$-coefficients fails, cf. Theorem \ref{t:Main}. More precisely,
\begin{itemize}
\item[$(i)$] The cokernels of the maps
\begin{eqnarray*}
    \Si_{0,1}: H_{2n+1}(\G_{3n+1,r})\To H_k(\G_{3n+1,r+1}) \\
    \Si_{0,1}: H_{2n+2}(\G_{3n+2,r})\To H_k(\G_{3n+2,r+1})
\end{eqnarray*}
are independent of $r\ge 1$.
\item[$(ii)$]Let $r\ge 2$. Then the cokernel of the map
\begin{eqnarray*}
    \Si_{1,-1}: H_{2n+1}(\G_{3n,r})\To H_k(\G_{3n+1,r-1})
   \end{eqnarray*}
is independent of $r$.
\end{itemize}
\end{rem}
\begin{proof}
Since $\Si_{0,1}$ is always injective, it fits into the following long exact sequence,
 \begin{equation*}
   H_{2n+1}(\G_{3n+1,r})\To H_{2n+1}(\G_{3n+1,r+1})\To
   \Rel^{\Z}_{2n+1}(F_{3n+1,r+1},F_{3n+1,r})\To 0.
\end{equation*}
Since $2(3n+2)\ge 3(2n+2)-2$, we get by Theorem \ref{t:maintwist} that the cokernel is independent of $r$. The other case is similar. For $(ii)$ we get
\begin{eqnarray*}
  {\xymatrix{H_{q}(\G_{3n,r})\ar[r]^{\Si_{1,-1}}\ar[d] &H_{q}(\G_{3n+1,r-1})\ar[r]\ar[d]&
   \Rel^{\Z}_{q}(F_{3n+1,r-1},F_{3n,r})\ar[r]\ar[d]^{\iso}& H_{q-1}(\G_{3n,r})\ar[d]^{\iso}\\
H_{q}(\G_{3n,r+1})\ar[r]^{\Si_{1,-1}} &H_{q}(\G_{3n+1,r})\ar[r]&
   \Rel^{\Z}_{q}(F_{3n+1,r},F_{3n,r+1})\ar[r]& H_{q-1}(\G_{3n,r+1})}}
\end{eqnarray*}
(We have written $q=2n+1$ to save space.) As the last two vertical maps are isomorphisms, the cokernels of the first map in the top and bottom rows are equal.
\end{proof}

The above Theorem finishes the inductive proof of the assumption $I_{n,k}$. The reason for proving the inductive assumption is that we now get the following Main Theorem for homology stability with twisted coefficients:

\begin{thm}\label{t:abstwist}Let $F$ be a surface of genus $g$, and let $V$ be a
coefficient system of degree $k$. Let $(l,m)=(1,0)$, $(0,1)$ or $(1,-1)$.
Then the map
\begin{equation*}
   H_n(F; V(F)) \To
   H_n(\Si_{l,m}F; V(\Si_{l,m}F))
\end{equation*}
induced by $\Si_{l,m}$ satisfies: \begin{itemize}
\item[$(i)$]For $\Si_{l,m}=\Si_{0,1}$, it is an isomorphism for $2g\ge 3n+k$.
\item[$(ii)$]For $\Si_{l,m}=\Si_{1,0}$ or $\Si_{1,-1}$,
it is surjective for $2g\ge 3n+k-\e_{l,m}$, and an isomorphism for $2g\ge
3n+k+2$.
\end{itemize}
\end{thm}

\begin{proof}Consider the following exact sequence
\begin{equation*}
\Rel_{n+1}^V(\Si_{l,m}F,F)\To H_n(F;V)\To
H_n(\Si_{l,m}F;\Si_{l,m}V)\To \Rel_n^V(\Si_{l,m}F,F).
\end{equation*}
To show surjectivity, we must prove that $\Rel_n^V(\Si_{l,m}F,F)=0$.
By $I_{k,n+1}$ this is the case when $2g\ge 3n+k$. To show
injectivity, we first note that as usual, $\Si_{0,1}$ is always
injective. For $\Si_{1,-1}$, we get by $I_{k,n+2}$ that
$\Rel_{n+1}^V(\Si_{l,m}F,F)=0$ when $2g\ge 3(n+1)+k+2$. Finally, $\Si_{1,0}=\Si_{1,-1}\Si_{0,1}$ and thus also injective when $2g\ge 3(n+1)+k+2$.
\end{proof}
\newpage
\section{Stability of the space of surfaces}
In \cite{CM}, Cohen and Madsen consider the following type of
coefficients
\begin{equation*}
  V^X_n(F) := H_n(\Map(F/\dd F,X))
\end{equation*}
for $X$ a fixed topological space.

\begin{lem}\label{l:Eilenberg}
Let $K=K(G;k)$ be an Eilenberg-MacLane space with $k\ge 2$.
Assume $H_*(K)$ is without infinite division. Then $V^K_n$ is a
coefficient system of degree $\le \textstyle\round{\frac{n}{k-1}}$.
\end{lem}

\begin{proof}
To prove $V_n^K$ is a coefficient system of degree $\le
\textstyle\round{\frac{n}{k-1}}$, we must prove that the groups $V_n^K(F)$ are without infinite division, and that $V_n^K$ has the right degree.

We consider the degree first, and the proof is by induction on $n$.
Take $\Si=\Si_{1,0}$, the other cases are similar. We have the following homotopy cofibration:
\begin{equation*}
  S^1\wedge S^1 \To \Si F/\dd \Si F\To  F/\dd F
  \end{equation*}
Taking $\Map(-,K)$ leads to the following fibration:
\begin{equation}\label{e:fibration}
  \Map(F/\dd F,K) \To \Map(\Si F/\dd \Si F,K)\To \Omega(K)\times\Omega(K)
  \end{equation}
Since $K=K(G,k)$ is an infinite loop space it has a multiplication,
and consequently so has each space in the fibration
\eqref{e:fibration} above. Thus the total space is up to homotopy
the product of the base and the fiber. Using Künneth's formula, we
get:
\begin{equation}\label{e:Kunneth}
V^K_n(\Si F)=\bigoplus_{i=0}^n V^K_{n-i}(F)\tensor
H_{i}(\Omega(K)\times\Omega(K))
\end{equation}
Note for $n=0$ this says that $\Si$ induces an isomorphism, so
$V_0^K(F)$ has degree $0$. This was the induction start.

Now since $\Omega(K)=K(G,k-1)$ is $(k-2)$-connected and $k\ge 2$,
$H_{0}(\Omega(K)\times\Omega(K))=\Z$ and
$H_{j}(\Omega(K)\times\Omega(K))=0$ for $j\le k-2$. This means that
the cokernel of $\Si$ is:
\begin{equation*}
\D(V_n^K(F))=\bigoplus_{i=k-1}^n V_{n-i}^K(F)\tensor
H_{i}(\Omega(K)\times\Omega(K))
\end{equation*}

Since the degree of a direct sum is the maximum of the degrees of
its components, we get by induction that the degree of
$\D(V_n^K(F))$ is $\le
\textstyle\round{\frac{n-(k-1)}{k-1}}=\textstyle\round{\frac{n}{k-1}}-1$.
This shows that the degree of $V_n^K$ is $\le
\textstyle\round{\frac{n}{k-1}}$.

It remains to show that $V_n^K(F)$ is an abelian group without infinite division for any surface $F$. To prove this, we use a double induction in $n$ and $F$. There are two base cases.

First consider $n=0$, $F$ any surface. From \eqref{e:Kunneth} we see
that $V_0^K$ does not depend on the surface $F$. So we can calculate
$V_0^K(F)$ using $F=D$ a disk:
\begin{equation*}
    V_0^K(F)= H_0(\Map(D/\dd D,K))=\Z[\pi_2(K)]=\left\{
                                                                    \begin{array}{ll}
                                                                      \Z, & \hbox{$k>2$;} \\
                                                                      \Z[G], & \hbox{$k=2$.}
                                                                    \end{array}
                                                                  \right.
\end{equation*}
This is an abelian group without infinite division.

Secondly, let $F=D$ be a disk, and $n$ any natural number. We see
\begin{eqnarray*}
  V_n^K(D) &=& H_n(\Map(D/\dd D,K)) = H_n(\Map(S^2,K)) \\
  &=& H_n(\Map(S^0,\Omega^2(K))=H_n(\Omega^2(K))
\end{eqnarray*}
and according to our assumptions on $H_*(K)$, this is without infinite division.

The general case now follows from induction using \eqref{e:Kunneth}
and its counterpart for $\Si=\Si_{0,1}$, along with the fact that any surface $F$ with boundary can be obtained from a disk $D$ using $\Si_{1,0}$ and $\Si_{0,1}$ finitely many times.
\end{proof}


To prove the next theorem we need a couple of lemmas:

\begin{lem}\label{l:tensor}
Let $V$ and $W$ be coefficient systems of degrees $\le s$ and $\le t$, respectively. Then $V\tensor W$ is a coefficient system of degree $\le s+t$, and $V\oplus W$ is a coefficient system of degree $\le \max(s,t)$.
\end{lem}

\begin{proof}Since $V$ is a coefficient system, we have the split exact sequence:
\begin{equation*}
    0\To V(F)\To V(\Si F)\To \D(V(F))\To 0.
\end{equation*}
Likewise for $W$. Then for the tensor product we get the split exact sequence:
\begin{eqnarray*}
  0&\To& V(F)\tensor W(F)\To V(\Si F)\tensor W(\Si F)\\
  &\To&  \D(V(F))\tensor W(F)\oplus V(F)\tensor \D(W(F)) \To 0.
\end{eqnarray*}
\end{proof}

\begin{thm}\label{t:stabil1}Let $X$ be a $k$-connected space, $k\ge 1$. If $V^X_n(F)$ is without infinite division for any surface $F$, then $V^X_n$ is a coefficient system of degree $\le \round{\frac{n}{k}}$.
\end{thm}

\begin{proof}First note: If we prove the assertion concerning the degree as in Def. \ref{d:coef} (not including without infinite division), then since $V^X_n$ is assumed without infinite division, the cokernels $\D_{i,j}(V^X_n)$ (and their cokernels, etc) are automatically without infinite division, since they are direct summands of $V^X_n$.

The proof uses Postnikov towers and Lemma \ref{l:Eilenberg} above. The Postnikov tower of $X$ is a sequence $\set{X_m\To X_{m-1}}_{m\ge k}$ with each term a fibration
\begin{equation}\label{e:postnikov}
    K(\pi_m(X),m)\To X_m\To X_{m-1}.
\end{equation}
The proof is by induction in $m$, so assume for $l<m$ that $V^{X_{l}}_n$ is a coefficient system of degree $\le \round{\frac{n}{k}}$. To make the induction work, we also assume inductively that the splitting $s_l$ we then have by definition,
\begin{equation*}
    \xymatrix{0\ar[r]& V^{X_{l}}_n\ar[r]& \Si V^{X_{l}}_n\ar[r]& \D(V^{X_{l}}_n)\ar@/_/[l]_{s_l}\ar[r]&0
    }
\end{equation*}
is a natural transformation from $\D(V^{X_{l}}_n)$ to $\Si V^{X_{l}}_n$.

Now we take the induction step. Let $F$ be a surface. Then using $\Map(F,-)$ on \eqref{e:postnikov} yields a new fibration
\begin{equation*}
    \Map(F,K(\pi_m(X),m))\To \Map(F,X_m)\To \Map(F,X_{m-1}).
\end{equation*}
Serre's spectral sequence for this fibration has $E^2$-term:
\begin{eqnarray}\label{e:sss}
    E^2_{s,t}(F)&=& H_s(\Map(F,X_{m-1}))\tensor
    H_t(\Map(F,K(\pi_m(X),m))\nonumber \\
    &=& V^{X_{m-1}}_s(F)\tensor V^{K(\pi_m(X),m)}_t(F).
\end{eqnarray}
Now $X_{m-1}$ is $k$-connected, since $X$ is, and $K(\pi_m(X),m)$ is
at least $k$-connected. Then by induction and Lemma \ref{l:tensor}, $ E^2_{s,t}$ is a coefficient system of degree $\le \round{\frac{s}{k}} +\round{\frac{t}{k}}\le \round{\frac{s+t}{k}}$.

We now want to prove that $E^r_{s,t}$ is a coefficient system of degree $\textstyle\le\round{\frac{s+t}{k}}$ for all $r\ge 2$, by induction in $r$.
Let $V_1\stackrel{d}{\To} V\stackrel{d}{\To} V_2$ be groups in the $E^r$ term of the spectral sequence, where $d$ denotes the $r$th differential, and say $V$ has degree $\le q$. We assume by induction in $r$ that the splittings for $V$, $V_1$ and $V_2$ (see \eqref{e:splittings}) are natural transformations. For $r=2$ this holds according to \eqref{e:sss} by induction in $m$ and by \eqref{e:Kunneth} (the Eilenberg-MacLane space case). We want to show that the homology of $V$ with respect to $d$, $H(V)$, is a coefficient system of degree $\le q$, and that the splitting for $H(V)$ is also natural. Suppose by another induction that this holds for coefficient systems of degrees $<q$.

Then consider the following diagram, where $\Si$ as usual denotes either $\Si_{1,0}$ or $\Si_{0,1}$.
\begin{equation}\label{e:splittings}
\xymatrix{
0\ar[r]& V_1\ar[r]^{\Si}\ar[d]^{d} & \Si V_1\ar[r]\ar[d]^{d} & \Delta_1\ar[r]\ar[d]^{d}\ar@/_/[l] &0 \\
0\ar[r]& V\ar[r]^{\Si}\ar[d]^{d} & \Si V\ar[r]\ar[d]^{d} & \Delta\ar[r]\ar[d]^{d}\ar@/_/[l] &0 \\
0\ar[r]& V_2\ar[r]^{\Si} & \Si V_2\ar[r] &\Delta_2\ar[r]\ar@/_/[l] &0
}
\end{equation}
We know $\Si V= V\oplus \D$, and similarly for $V_1$ and $V_2$. By our induction hypothesis in $r$ we get that the splittings in the right-most squares above commute with $d$. Then the homology with respect to $d$ satisfies $H(\Si V)= H(V)\oplus H(\D)$, and the splitting for $H(V)$ is again natural. This shows that the cokernel $\D(H(V))$ of $\Si$ is $H(\D)$. Since $\D$ is a coefficient system of degree $\le q-1$, we get by induction in the degree that $H(V)$ is a coefficient system of degree $\le q$. For the degree-induction start, if $V$ is constant, $H(V)$ is also constant.

To finish the induction in $m$ we must prove that the splitting $s_m: \D(V^{X_{m}}_n)\To\Si V^{X_{m}}_n$ is a natural transformation. By the above, $E^r_{s,t}$ is a coefficient system of degree $\le \round{\frac{s+t}{k}}$ for all $r$, so the same is true for $E^\infty_{s,t}$. Since the spectral sequence converges to $V^{X_m}_n(F)$ for $n=s+t$, we get that $V^{X_m}_n(F)$ is a coefficient system of degree $\le \round{\frac{n}{k}}$.

%
%
The inverse limit of the Postnikov tower $\lim_{\leftarrow}X_m$ is
weakly homotopy equivalent to $X$, and the result follows.
\end{proof}

The space of surfaces mapping into a background space $X$ with
boundary conditions $\g$ is defined as follows: Let $X$ be a space
with base point $x_0\in X$, and let $\g:\coprod S^1\To X$ be $r$
loops in $X$. Then
\begin{eqnarray*}
    \mathcal{S}_{g,r}(X,\g) &=& \left\{(F_{g,r},\f,f)\mid
    F_{g,r}\del \R^\infty\times [a,b], \f:\sqcup S^1\To \dd
F_{g,r}\text{ is a para-}\right.\\
   &&\left. \text{metrization},
    f: F_{g,r}\To X
    \text{ is continuous with }f\circ \f=\g\right\}
\end{eqnarray*}

Assume now $X$ is simply-connected. Then we observe that the
homotopy type of $\mathcal{S}_{g,r}(X,\g)$ does not depend on $\g$:
For consider the space of surfaces with no boundary conditions, call
it $\overline{\mathcal{S}_{g,r}(X)}$. The restriction map to the
boundary of the surfaces,
\begin{equation*}
    \mathcal{S}_{g,r}(X,\g) \To \overline{\mathcal{S}_{g,r}(X)}\To (LX)^r
\end{equation*}
is a Serre fibration. Here, $LX=\textrm{Map}(S^1,X)$ is the free
loop space, so as $X$ is simply-connected, $(LX)^r$ is connected, so
the fiber is independent of the choice of $\g\in (LX)^r$. So when
$X$ is simply-connected, we use the abbreviated notation
$\mathcal{S}_{g,r}(X)=\mathcal{S}_{g,r}(X,\g)$ for any choice of
$\g$.

\begin{thm}
Let $X$ be a simply-connected space such that $V^X_m$ is without infinite division for all $m\le n$. Then
\begin{equation*}
    H_n(\mathcal{S}_{g,r}(X))
\end{equation*}
is independent of $g$ and $r$ for $2g\ge 3n+3$ and $r\ge 1$.
\end{thm}

\begin{proof}
Let $\Si$ be either $\Si_{1,0}$ or $\Si_{0,1}$. From the definition
we observe that
\begin{equation*}
\mathcal{S}_{g,r}(X)\cong
\textrm{Emb}(F_{g,r},\R^\infty)\times_{\textrm{Diff}(F_{g,r},\dd)}\Map(F_{g,r},X),
\end{equation*}
and since $\textrm{Emb}(F_{g,r},\R^\infty)$ is contractible, we get
\begin{equation*}
\mathcal{S}_{g,r}(X)\cong E(\textrm{Diff}(F_{g,r},\dd))
\times_{\textrm{Diff}(F_{g,r},\dd)}\Map(F_{g,r},X).
\end{equation*}
So there is an obvious fibration sequence
\begin{equation*}
\Map(F_{g,r},X)\To\mathcal{S}_{g,r}(X)\To
B(\textrm{Diff}(F_{g,r},\dd),
\end{equation*}
and thus we can apply Serre's spectral sequence, which has $E^2$
term:
\begin{equation*}
E^2_{s,t}= H_s(B(\textrm{Diff}(F_{g,r},\dd);H_t(\Map(F_{g,r},X)))
\end{equation*}
where the coefficients are local. The path components of
$\textrm{Diff}(F_{g,r},\dd)$ are contractible, so we get an
isomorphism
\begin{equation}\label{e:E^2ib}
E^2_{s,t}\iso H_s(\G(F_{g,r});H_t(\Map(F_{g,r},X)))
\end{equation}
Consider the map induced by $\Si$ on this spectral sequence
\begin{equation*}
    \Si_*: H_s(\G(F_{g,r});H_t(\Map(F_{g,r},X))) \To H_s(\G(\Si F_{g,r});H_t(\Map(\Si F_{g,r},X)))
\end{equation*}
By Theorem \ref{t:stabil1} and \ref{t:abstwist}, we know that this map is surjective for
$2g\ge 3s+t$, and an isomorphism for $2g\ge 3s+t+2$. We use Zeeman's comparison theorem to carry the result to $E^\infty$. To get the optimum
stability range, we must find the maximal $N=N(g)\in\Z$ such that
for $t\ge 1$,
\begin{eqnarray*}
  s+t\le N &\Tto& 2g\ge 3s+t+2 \quad \text{(isomorphism)}\\
  s+t =N+1&\Tto & 2g\ge 3s+t \quad \text{(surjectivity)}
\end{eqnarray*}

Zeeman's comparison theorem then says that $\Si_*$ induces
isomorphism on $E^\infty_{s,t}$ for $s+t\le N(g)$ and a surjection
for $s+t=N(g)+1$. Since the spectral sequence converges to
$H_n(\mathcal{S}_{g,r}(X))$, we get stability for $n\le N(g)$.

Clearly, the hardest requirement is $t=0$ (surjectivity), where we
get the inequality $2g\ge 3N+3$. One checks that this satisfies all
the other cases. So $H_n(\mathcal{S}_{g,r}(X))$ is independent of
$g,r$ for $2g\ge 3n+3$.
\end{proof}

Using this we can improve the stability range in Cohen-Madsen's stability result for the homology of the space of surfaces to the following, cf \cite{CM} Theorem 0.1:

\begin{thm}Let $X$ be a simply connected space such that $V^X_m$ is without infinite division for all $m$. Then for $2g\ge 3n+3$ and $r\ge 1$ we get an isomorphism
\begin{equation*}
    H_n(\mathcal{S}_{g,r}(X)_\bullet)\cong  H_n(\Om^\infty(\mathbb{CP}^\infty_{-1} \wedge X_+ )_\bullet).
\end{equation*}
\end{thm}

\end{document}